\RequirePackage{fix-cm}
\documentclass[smallextended]{svjour3_no_journal_info}       
\smartqed  
\usepackage[letterpaper,top=3cm,bottom=3cm,left=3cm,right=3cm,marginparwidth=1.75cm]{geometry}
\usepackage{lipsum}
\usepackage{amsfonts}
\usepackage{graphicx,wrapfig}
\usepackage{epstopdf}
\usepackage{algorithmic}
\ifpdf
  \DeclareGraphicsExtensions{.eps,.pdf,.png,.jpg}
\else
  \DeclareGraphicsExtensions{.eps}
\fi

\usepackage[utf8]{inputenc} 
\usepackage[T1]{fontenc}    
\usepackage{url}            
\usepackage{booktabs}       
\usepackage{amsfonts}       
\usepackage{nicefrac}       
\usepackage{microtype}
\usepackage{mathtools}
\usepackage{natbib}
\usepackage{graphicx}
\usepackage{subcaption}
\usepackage{amsmath,amssymb,amsfonts,amsxtra,bm}
\usepackage{enumerate}
\usepackage{enumitem}
\usepackage{hyperref}       
\usepackage{algorithm,algorithmic}
\usepackage{booktabs, 
            makecell, multirow, tabularx} 

\clearpage{}%

\PassOptionsToPackage{numbers, compress}{natbib}

\usepackage{mysymbolmp}
\usepackage{myacronym}
\usepackage{mypackagemp}

\newcommand{\linktoPwop}{\hyperref[problem]{\textcolor{black}{\ensuremath{\ccalP}}}}
\newcommand{\linktoBPwop}{\hyperref[problem:residual-learning]{\textcolor{black}{\ensuremath{\mathcal{BP}}}}}
\newcommand{\linktoIPwop}{\hyperref[problem:implicit-regularization]{\textcolor{black}{\ensuremath{\mathcal{IP}}}}}
\newcommand{\linktoIBPwop}{\hyperref[problem:IBP]{\textcolor{black}{\ensuremath{\mathcal{IBP}}}}}

\newcommand{\AnalogSGD}{\texttt{Analog\allowbreak\;SGD}}
\newcommand{\DigitalSGD}{\texttt{Digital\allowbreak\;SGD}}

\newcommand{\ResidualLearning}{\texttt{Residual\allowbreak\;Learning}}

\newtheorem{assumption}{\hspace{0pt}\bf Assumption}

\newcommand{\FullTitle}{
Optimization under Persistent State-Dependent Bias: Gradient-based Method and Complexity Analysis}

\title{
Optimization under Persistent State-Dependent Bias: Gradient-based Method and Complexity Analysis
}

\titlerunning{
Optimization under Persistent State-Dependent Bias: Gradient-based Method and Complexity Analysis
}        

\author{
        Zhaoxian Wu \and
        Quan Xiao \and
        Tayfun Gokmen \and
        Tianyi Chen 
}

\authorrunning{
    Zhaoxian Wu, Quan Xiao, Tayfun Gokmen, Tianyi Chen
    } 

\institute{
    Zhaoxian Wu \and Quan Xiao \and Tianyi Chen, 
    Department of Electrical and Computer Engineering, Cornell University \\
    Tayfun Gokmen, 
    IBM T. J. Watson Research Center\\
    \email{zw868@cornell.edu; qx232@cornell.edu; tgokmen@us.ibm.com; tianyi.chen@cornell.edu} \\
    This submission is an extended version of the conference paper \cite{wu2025analog}.
}
\date{}

\begin{document}

\maketitle
\begin{abstract}
    This paper studies the convergence of stochastic gradient descent (SGD) when the implemented updates are subject to a persistent and state-dependent bias, in which the desired update is scaled by response functions component-wise. Our first contribution is to demonstrate that SGD in this setting implicitly optimizes a penalized problem whose minimizer does not coincide with the true minimizer. To mitigate this convergence failure, we reformulate the original task as an equivalent bilevel optimization problem and propose a gradient-based algorithm, termed \ResidualLearning. Theoretical analysis shows that {\ResidualLearning} finds a solution to the original, unbiased optimization problem despite the hardware imperfections. Beyond exact convergence, we quantify how the response functions affect convergence complexity via the hardware condition number and show that a polynomial dependence on it is unavoidable in general, via a construction of a hard instance. The theoretical results are supported by numerical simulations that demonstrate the effectiveness of the proposed algorithm. 

\keywords{Stochastic optimization; bilevel optimization; approximate computing; in-situ optimization}
\subclass{
    {90C26 \and 90C15 \and 90C06 \and 90C60 \and 49M37 \and 68Q25}
}
\end{abstract}

\section{Introduction}
\label{section:introduction}

This paper studies stochastic first-order methods under persistent and state-dependent update bias.
A canonical problem in training machine learning models is the unconstrained stochastic optimization task with a differentiable objective $f(\,\cdot\,):\reals^D\to\reals$ over $W\in\reals^D$, given by
\begin{align}
    \label{problem}
    (\ccalP):~~
    W^* = \argmin_{W\in\reals^{D}}~ f(W) := \mbE_{\xi}[f(W; \xi)]
\end{align}
where $\xi$ is a random sample. 
Based on Boolean arithmetic and discrete gate manipulation, the "workhorse" algorithm, stochastic gradient descent (SGD) \citep{robbins1951stochastic}, is currently implemented on digital accelerators such as CPUs and GPUs. 
At each iteration $k$, given a stochastic gradient $\nabla f(W_k; \xi_k)$, SGD on digital hardware (\DigitalSGD) updates $W_k$ via
\begin{align}
    \label{recursion:digital-SGD}
    \texttt{Digital~SGD}\qquad
    W_{k+1} = W_k - \alpha \nabla f(W_k; \xi_k)
\end{align}
where $\alpha$ is a (decaying) stepsize, e.g., {$\alpha=\Theta(\log K/K)$} where $K$ is the iteration budget.
Since the intermediate results are represented by discrete states, {\DigitalSGD} follows the dynamics \eqref{recursion:digital-SGD} up to an error that is proportional to machine precision.
By implicitly assuming sufficiently high precision, the update \eqref{recursion:digital-SGD} corresponds to an ideal digital computing model where hardware-induced errors are negligible.
For a strongly convex and smooth objective $f(\,\cdot\,)$, {\DigitalSGD} ensures that the iterates can converge to the optimum $f(W_K)\rightarrow f^*$ when iteration $K\rightarrow+\infty$, where $f^*$ is the optimal value of {\linktoPwop} \citep{hazan2014beyond,karimi2016linear}.

\begin{wrapfigure}[20]{r}{0.36\linewidth}
    \vspace{-1.9em}
    \centering
    \includegraphics[width=\linewidth]{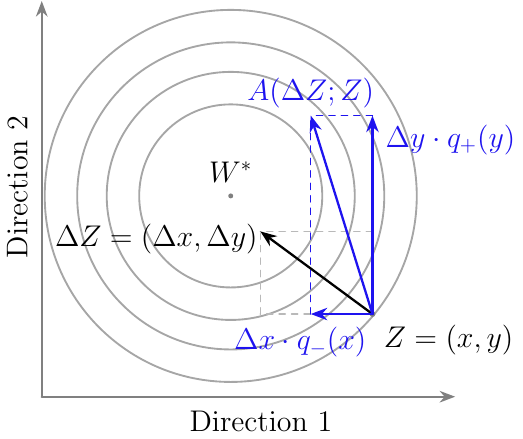}
    \vspace{-1.8em}
    \caption{Illustration of persistent state-dependent bias in $\reals^2$.
    Any desired update $\Delta Z$, which is typically a descent direction of the objective function, is scaled by state-dependent response functions $q_+(\cdotc)$ and $q_-(\cdotc)$ component-wise.
    The bias is exacerbated when the hardware condition number increases.}
    \label{fig:hardware-challenge}
    \vspace{-2.5em}
\end{wrapfigure}
This paper investigates a distinct computational model involving persistent bias that arises in in-situ optimization on analog-based physical computing systems. 
Let $\{Z_k : k\in\naturals\}$ be a variable sequence generated by \textit{any} optimization algorithm, and let $\Delta Z_k$ be the \textit{desired update} at the $k$-th iteration. 
On top of that, there are two \textit{unknown state-dependent response functions} $q_+(\cdotc), q_-(\cdotc) : \reals\to\reals$. 
For each coordinate $d$, the hardware may induce a different pair of scalar response functions $q_{+,d}(\cdotc)$ and $q_{-,d}(\cdotc)$. To keep the notation light, we suppress the coordinate index and write the update in component-wise form.
With these notations, the dynamics of $Z_k$ are given by $Z_{k+1} = Z_k + A(\Delta Z_k; Z_k)$
where the analog operator $A(\cdotc;\cdotc) : \reals^{D}\times\reals^D\to\reals^D$ is applied component-wise and defined as
\begin{align}
    \label{analog-update}
    A(\Delta z; z) := \begin{cases}
        \Delta z \cdot q_{+}(z),~~~\Delta z \ge 0, \\
        \Delta z \cdot q_{-}(z),~~~\Delta z < 0.
    \end{cases}
\end{align}
With 
$F(w) := \frac{1}{2}(q_{-}(w)+q_{+}(w))$ and $G(w) := \frac{1}{2}(q_{-}(w)-q_{+}(w))$,
the operator can be decomposed into $A(\Delta z; z) = \Delta z \odot F(z) -|\Delta z| \odot G(z)$, and hence
the dynamics of $Z_k$ can be rewritten as
\begin{align}
    \label{biased-update}
    \texttt{In-situ Update Dynamics}\qquad
    Z_{k+1} = Z_k + \Delta Z_k \odot F(Z_k) -|\Delta Z_k| \odot G(Z_k)
\end{align}
where 
$|\,\cdot\,|$ and $\odot$ represent the component-wise absolute value and multiplication, respectively. 
For example, replacing $Z_k$ with the optimization variable $W_k$, and replacing $\Delta Z_k$ with stochastic gradient, {\AnalogSGD} evolves following the dynamics
\begin{align}
    \label{recursion:analog-SGD}
    \texttt{Analog~SGD}\qquad
    W_{k+1} =&\ W_k - \alpha \nabla f(W_k; \xi_k)\odot F(W_k)
    - \alpha|\nabla f(W_k; \xi_k)|\odot G(W_k).
\end{align}
Compared with {\DigitalSGD}, the stochastic gradient $\nabla f(W_k; \xi_k)$ in \eqref{recursion:analog-SGD} is scaled by $F(\,\cdot\,)$, and an extra term that depends on the magnitude of $\nabla f(W_k; \xi_k)$ is introduced. 

\subsection{Motivated application: Analog-based physical computing}
\label{section:introduction-preliminary}
This computational model is motivated by analog computing for optimization problems, 
which arises from the gap between the computing requirements of large AI models and the limited hardware development. 
Modern hardware accelerators, such as GPUs and TPUs \citep{jouppi2023tpu}, are usually \textit{digital computing hardware}, in which frequent data movement is involved in computation.
To further improve computational efficiency, an analog computing paradigm has been proposed \citep{haensch2019next, Y2020sebastianNatNano}, wherein the optimization variables are represented by continuous physical quantities (e.g., conductance states of resistive elements).
This novel computing approach leverages Kirchhoff's and Ohm's laws to accelerate matrix-vector multiplication, a bottleneck in AI computing.
Without moving variables during optimization processes, analog computing hardware achieves 10$\times$-10,000$\times$ energy efficiency than GPU \citep{jain2019neural,cosemans2019towards,papistas202122} in the inference, and $1,000\times$ in the training \citep{wang2020ssm, huang2020overcoming, nandakumar2020}. 
Unlike digital hardware, analog hardware updates variables in-situ by applying electrical pulses to resistive elements. Since the resistive element responds to positive and negative pulses differently, the desired update in $\Delta z$ is scaled by $q_+(\cdotc)$ or $q_-(\cdotc)$ depending on its sign.
Empirical evidence shows that the response is state-dependent and asymmetric, which can be captured by the dynamics \eqref{analog-update} \citep{burr2015experimental, chen2015mitigating, wu2024towards}.

\subsection{Main results}
\label{section:main-results}
This paper aims to establish a rigorous optimization foundation under the assumption that all variables follow the dynamics \eqref{biased-update}. 
A primary challenge is the \textit{asymmetric update}.
That is, for most of $z$, the response to positive and negative stimuli is different, i.e., $q_+(z) \neq q_-(z)$. Therefore, intended positive and negative updates of equal magnitude produce unequal changes in the variable.  
Specifically, a point $w^\diamond$ satisfies 
$q_+(w^\diamond)=q_-(w^\diamond)$,
we say $w^\diamond$ is a \textit{symmetric point}.
Except for the symmetric point, the update is asymmetric elsewhere.
As a stepping stone, we first answer the following question:
\begin{center}
    \textbf{Q1)} 
    {\em What is the impact of the persistent state-dependent bias on the optimization? }
\end{center}
The study begins by investigating its impact on SGD.
We show that the stationary point of {\AnalogSGD} is not the minimizer of $f(\cdotc)$ but an approximate minimizer $\tdW^*$ of a penalized objective $f(W)+\varphi(W)$, where $\varphi(W)$ is an \textit{implicit penalty} that will be defined in Theorem \ref{theorem:implicit-regularization}.
In other words, instead of optimizing problem \linktoPwop, {\AnalogSGD} optimizes another penalized problem, called problem \linktoIPwop. We show that the penalty coefficient is proportional to the noise level of the stochastic gradients, and the minimizers of $\varphi(W)$ are the symmetric point $W^\diamond\in\reals^D$, which is obtained by replicating $w^\diamond$ $D$ times.
Since $W^\diamond\ne W^*$, $W_k$ can not converge to $W^*$, and $f(W_k)$ does not converge to $f^*$; see \S\ref{section:preliminary-response-function}.
It motivates us to ask the following question:
\begin{center}
    \textbf{Q2)} {\em How to reformulate problem {\linktoPwop} so that it becomes exactly solvable under state-dependent bias? }
\end{center}

\begin{wrapfigure}[14]{r}{0.45\linewidth}
    \vspace{-0.6cm}
    \includegraphics[width=1\linewidth]{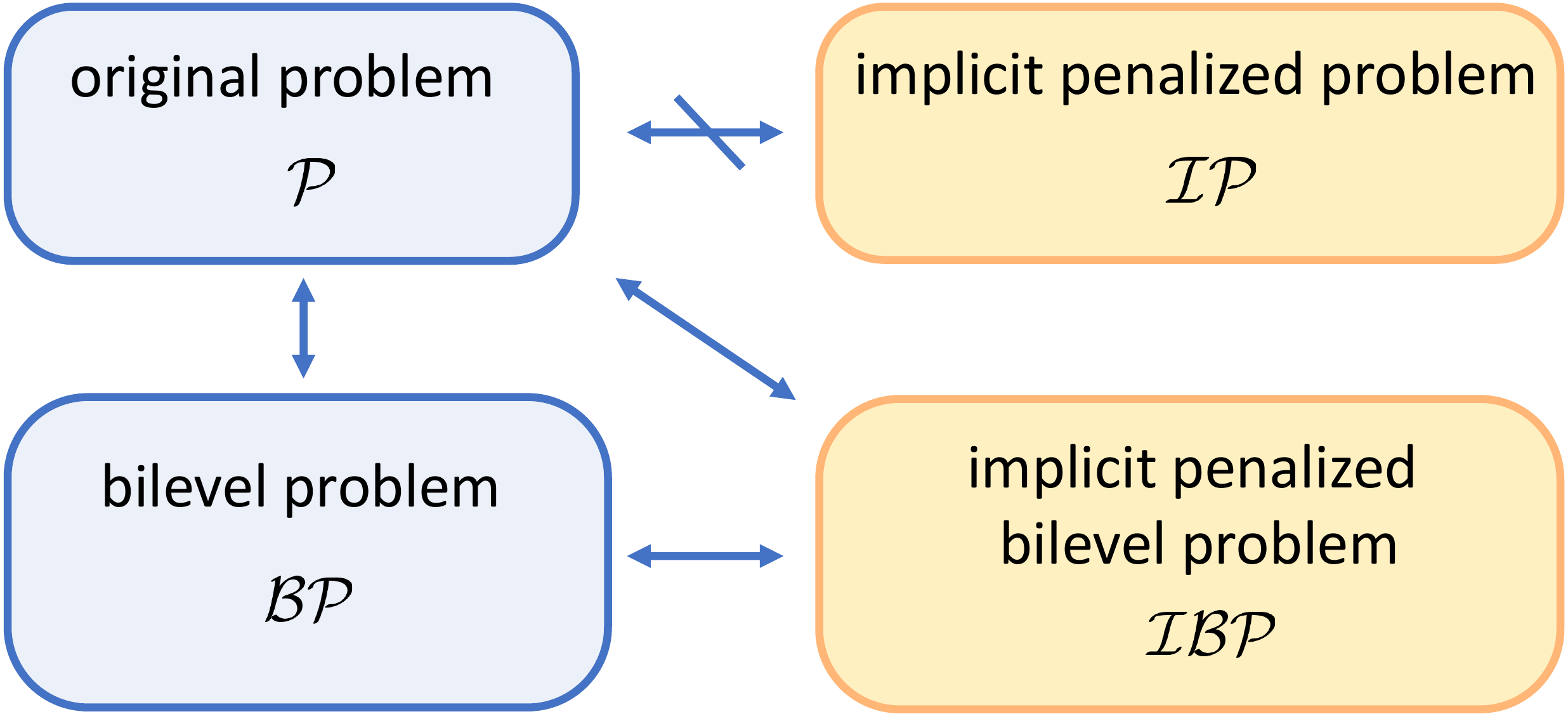}
    \vspace{-1.1em}
    \caption{Illustration of the relationship between four different problems. The blue boxes denote the target problem of interest, while the yellow boxes denote the penalized problem solved under persistent state-dependent bias.}
    \label{fig:four-problem}
    \vspace{-1.2em}
\end{wrapfigure}
To address inexact convergence, we formulate a bilevel optimization problem \linktoBPwop.
Our key idea is to reformulate the original problem {\linktoPwop} so that the implicit penalties are applied to an auxiliary residual objective whose minimizer aligns with the minimizer of the penalty $\varphi(\cdotc)$, while the resulting bilevel problem {\linktoBPwop} remains equivalent to {\linktoPwop}.
This reformulation enables the design of a hardware-implementable method, termed {\ResidualLearning}.
Similar to {\AnalogSGD}, {\ResidualLearning} also implicitly optimizes an implicit penalized bilevel problem \linktoIBPwop.
However, we show that problem {\linktoIBPwop} is equivalent to {\linktoBPwop}, and hence equivalent to {\linktoPwop}, as illustrated by Fig. \ref{fig:four-problem}. Consequently, despite the implicit penalty, {\ResidualLearning} converges exactly towards the solution of {\linktoPwop}; see \S\ref{section:TT}.

The equivalence result above qualitatively resolves the inexact convergence issue. 
From an optimization perspective, however, exactness alone does not determine whether the hardware-constrained problem is as easy to solve as its digital counterpart. It remains to understand whether the state-dependent bias introduces additional optimization difficulty. 
This raises the third natural question:
\begin{center}
    \textbf{Q3)} {\em How to quantify the optimization difficulty introduced by the state-dependent bias?}
\end{center}
To answer this question, we define the \textit{hardware condition number} $\kappa_2$ as the ratio between the maximum and minimum values of the response functions. 
This is analogous to the classical objective condition number $\kappa_1$, defined for a twice-differentiable objective as the ratio between the largest and smallest singular values of its Hessian \citep{nesterov2013introductory}.
Our analysis shows that {\ResidualLearning} preserves the optimal $\log K/K$ rate of the stochastic algorithm class up to logarithmic factors \citep{bottou2018optimization, agarwal2009information}.
The upper bound depends on $\kappa_1\kappa_2^4$, suggesting a substantially stronger sensitivity to the hardware condition number than to the objective condition number.
To assess whether this dependence is fundamental, we establish a lower bound of order $\kappa_2^2$ via a hard-instance construction, showing that a polynomial dependence on $\kappa_2$ is unavoidable.
Determining the sharp dependence on $\kappa_2$ and developing an optimal algorithm remain open problems; see Table \ref{table:mu-sc-classical-bounds} and \S\ref{section:convergence:RL}.

We verify the claims above empirically by simulating the training dynamics of {\AnalogSGD} and {\ResidualLearning} with various response functions. We visualize the impact of the implicit penalty and the hardware condition number. 
In addition to toy examples, we demonstrate that {\ResidualLearning} outperforms {\AnalogSGD} and achieves performance comparable to {\DigitalSGD} on real datasets and in an analog hardware simulation toolkit; see \S\ref{section:experiments}.

\begin{table}[t]
    \centering
    \renewcommand{\arraystretch}{1.25}
    \begin{tabular}{c|cl|cl}
        \toprule
        \textbf{Setting} &
        \multicolumn{2}{c|}{\texttt{Stochastic Gradient Descent} (Digital)} &
        \multicolumn{2}{c}{{\ResidualLearning} (Analog)}\\
        \midrule
        Upper Bound &
        $\displaystyle \ccalO\lp \kappa_1\cdot\frac{\sigma^2}{\mu K}\rp$ &
        [\citealp{bottou2018optimization}]
        &
        $\displaystyle \tilde\ccalO\lp \kappa_1 \kappa_2^4 \cdot \frac{\sigma^2}{\mu K}\rp$ &
        \textnormal{[Theorem \ref{theorem:TT-convergence-scvx}]} \\
        Lower Bound &
        $\displaystyle \Omega\lp \frac{\sigma^2}{\mu K}\rp$ &
        \textnormal{[{\citealp{agarwal2009information}}]} &
        $\displaystyle \tilde\Omega\!\left(
            \kappa_2^2\cdot \frac{\sigma^2}{\mu K}
        \right)$ &
        \textnormal{[Theorem \ref{theorem:lower-bound}]} \\
        \bottomrule
    \end{tabular}
    \caption{
        Comparison between the classical complexity benchmark for the last-iterate function gap in stochastic optimization in the digital setting and the corresponding complexity guarantee established in this paper for the last-iterate function gap in the hardware-constrained setting.
        $K$ is the iteration budget and $\sigma^2$ is the noise variance. $\kappa_1$ and $\kappa_2$ are the objective and hardware condition numbers, respectively, defined in \S\ref{section:convergence:RL}. $\tilde\ccalO$ and $\tilde\Omega$ hide logarithmic factors.
    }
    \label{table:mu-sc-classical-bounds}
\end{table}

\vskip 0.3\baselineskip
\noindent
\textbf{Notations.}
Let $W\in\reals^D$ be a vector.
$\|W\|$ and $\|W\|_\infty$ are the $\ell_2$- and $\ell_\infty$-norm of $W$, respectively.
$|W|$ is the component-wise absolute value. 
$\odot$ is the component-wise product of two vectors and $\frac{G}{F}$ is the component-wise division of two vectors $G, F\in\reals^D$. 
$[W]_d$ takes the $d$-th element of $W$.
Let $[D]:=\{1, 2, \cdots, D\}$ denote the index set.

\subsection{Literature review}
We review prior work from two aspects: 
optimization via biased gradient oracles
and {bilevel optimization}.

\paragraph{\bf Optimization via biased gradient oracles.}
{\DigitalSGD}, equipped with an unbiased gradient oracle, has become a standard method for solving the problem {\linktoPwop}.
However, errors in gradient calculations may appear deterministically in various contexts. For example, in bilevel optimization, the problem typically involves an inner-loop or approximation sequence to compute hyper-gradients \citep{arbel2021amortized,chen2021closing}; in distributed optimization, gradients can be compressed or quantized to reduce communication \citep{alistarh2017qsgd, magnusson2020maintaining, stich2018sparsified, wangni2018gradient, safaryan2022uncertainty}.
Without any special structure, general biased gradient oracles typically incur asymptotic error and lead to inexact convergence. For example, consider a general \textit{biased SGD} recursion $W_{k+1} = W_k - \alpha (\nabla f(W_k; \xi_k) + b(W_k))$ with the assumption that there exist constants $0\le A_1<1, A_2\ge 0$ such that $\|b(W)\|^2\le A_1 \|\nabla f(W)\|^2 + A_2$,
biased SGD \citep{ajalloeian2020convergence} has non-vanishing asymptotic error $\ccalO(A_2 / (1-A_1))$ if $A_2\ne 0$. 
In some applications, the bounds of the asymptotic error are not artifacts and are even tight; e.g, in the distributed computing system, the gradient estimation inevitably becomes biased if part of the workers become unreliable \citep{wu2020federated,karimireddybyzantine}.
A crucial insight in modern stochastic optimization, however, is that bias is not inherently incompatible with convergence. For example, in finite-sum optimization, several techniques can be incorporated into SGD to improve the convergence performance, such as without-replacement sampling orders \citep{haochen2019random, gurbuzbalaban2021random, huang2021improved} or variance reduction \citep{nguyen2017sarah, schmidt2017minimizing, cutkosky2019momentum}. In these methods, even though biased gradient oracles are employed, the gradient oracles are \textit{asymptotically unbiased}, i.e., the induced bias vanishes as the variables converge to the solution.
In general, \cite{khanh2024new} provides a condition under which biased SGD converges exactly. 
If asymptotic unbiasedness does not hold, one approach is to introduce auxiliary variables to mitigate bias during optimization, such as error feedback \citep{karimireddy2019error}.
Unlike existing work, which assumes ideal updates to auxiliary variables, this paper considers a more challenging setting in which any auxiliary variable suffers from the biased update \eqref{biased-update}.

\paragraph{\bf Gradient-based methods for bilevel optimization.} Bilevel optimization has a long history in optimization that dates back to \citep{bracken1973mathematical,ye1995optimality,vicente1994bilevel,colson2007overview}, where the upper-level objective depends on the optimal solution of the lower-level problem. Recent theoretical works have focused on developing efficient gradient-based bilevel methods with non-asymptotic convergence guarantees \citep{ghadimi2018approximation,ji2021bilevel,hong2020two,chen2021closing}. 
Since the gradient of the bilevel objective depends on the Hessian inverse of the lower-level objective, existing literature proposed different Hessian inversion approximation methods including unrolling differentiation \citep{franceschi2017forward,franceschi2018bilevel,grazzi2020iteration,shaban2019truncated}, implicit differentiation \citep{chen2021closing,ghadimi2018approximation,hong2020two,pedregosa2016hyperparameter,khanduri2021near,chen2022single}, conjugate gradient \citep{ji2021bilevel,yang2021provably} and its warm-started single-loop versions \citep{arbel2021amortized,li2022fully,liu2023averaged,xiao2023generalized,arbel2022nonconvex}. All of the above approximations are based on second-order information. 
In addition, researchers have proposed first-order methods, such as equilibrium backpropagation \citep{scellier2017equilibrium,scellier2021deep,zucchet2022beyond} and penalty approach \citep{shen2023penalty,liubome,kwon2023fully,kwon2023penalty,chen2024finding,lu2023first,jiang2024primal,yao2024overcoming,mehra2021penalty}. 
However, even when the lower-level problem is strongly convex, directly applying the above methods typically results in only stationary convergence, due to distortions caused by the nested structure. In fact, achieving global convergence in bilevel optimization is known to be {NP-hard} \citep{vicente1994descent}, and is only possible in finite time for those with linear or quadratic lower-level problems \citep{wang2021fast, wang2022solving, jeyakumar2016convergent, xiao2024unlocking}.
Unlike the aforementioned work, this paper considers a special bilevel problem {\linktoBPwop} and shows that this problem is solvable despite the persistent state-dependent bias.

\section{Gradient Methods with Persistent State-Dependent Bias and Implicit Penalization}
\label{section:preliminary-response-function}
This section first examines how various forms of state-dependent bias, mathematically characterized by the response functions, affect the update and proposes the discrete-time dynamics of analog training. 
After that, this section introduces a class of response functions that will be considered in our framework. 
Built on that, this section shows that {\AnalogSGD} in \eqref{biased-update} cannot, in general, converge to a critical point of the original problem $\ccalP$ in the presence of the state-dependent bias.

\textbf{Regular condition of response functions.}
Generally, the term $G(W)$ in \eqref{biased-update} is nonzero since $q_+(w)\neq q_-(w)$ for almost all $w$. The equality can hold only at the symmetry point $w^\diamond$ (cf. Fig.~\ref{fig:response-factor}).
Let $W^\diamond \in \reals^D$ denote the vector obtained by stacking the scalar symmetry points $w^\diamond$. Then $G(W^\diamond)=0$, and $G(W)$ remains small for $W$ in a neighborhood of $W^\diamond$; consequently, the resulting dynamics are locally close to those of {\DigitalSGD}.

We next impose mild regularity assumptions on the device response functions.
First, a positive pulse increases the conductance and a negative pulse decreases it; accordingly, we assume the response functions are strictly positive, i.e., $q_+(w)>0$ and $q_-(w)>0$ for all $w$.
Second, we assume the responses are uniformly bounded: there exist constants $0<q_{\min}\le q_{\max}<\infty$ such that
$q_{\min}\le q_+(w)\le q_{\max}$ and $q_{\min}\le q_-(w)\le q_{\max}$ for all $w$,
which models operation away from conductance saturation during training \citep{wu2024towards}.
Finally, for tractability, we assume $q_+(\cdotc)$ and $q_-(\cdotc)$ are differentiable and Lipschitz continuous.

\begin{figure*}[t]
    \centering
    \includegraphics[width=0.8\linewidth]{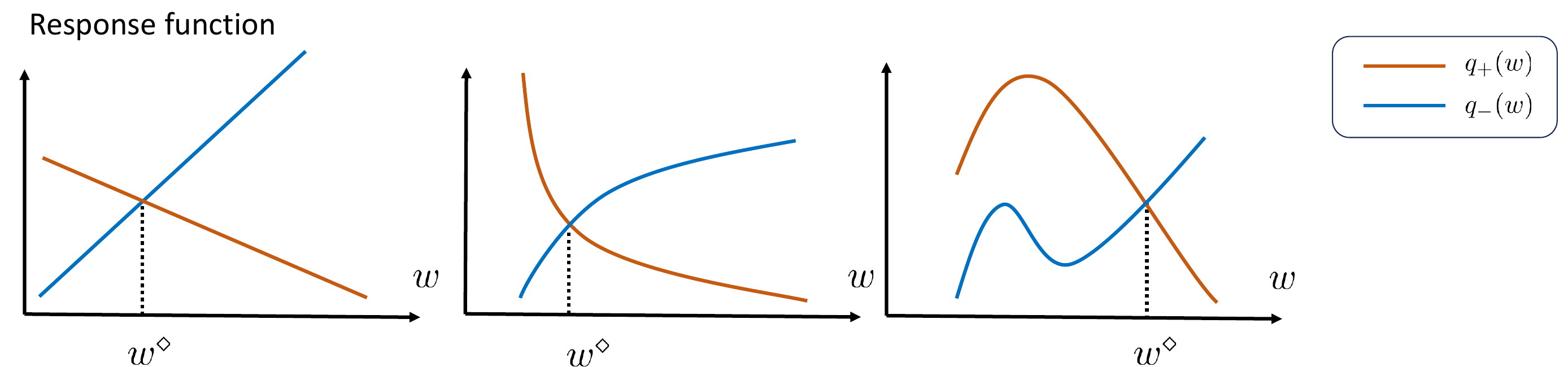}
    \vspace{-0.2cm}
    \caption{Examples of response functions from Assumption \ref{assumption:response-factor}; $w^\diamond$ is a symmetric point.
    }
    \label{fig:response-factor}
    \vspace{-0.4cm}
\end{figure*}

\begin{definition}[Regularity of response functions]
    \label{assumption:response-factor}
    Response functions $q_+(\,\cdot\,)$ and $q_-(\,\cdot\,)$ are training-friendly if they satisfy
    \begin{itemize}[leftmargin=2em, itemsep=0.8em]
        \item \textbf{(Positive-definiteness)} There exist positive constants $q_{\min}>0$ and $q_{\max}>0$ such that
        $q_{\min} \le q_+(w) \le q_{\max}$ and $q_{\min} \le q_-(w) \le q_{\max}, \forall w$; and,
        \item \textbf{(Differentiable)} The response functions $q_+(\,\cdot\,)$ and $q_-(\,\cdot\,)$ are differentiable.
        \item \textbf{(Lipschitz continuous)} The response functions $q_+(\,\cdot\,)$ and $q_-(\,\cdot\,)$ are $L_S$-Lipschitz continuous.
    \end{itemize}
\end{definition}
Under the regularity conditions, $F(\,\cdot\,)$ is always positive, which ensures that the update is non-zero when the gradient is non-zero. Furthermore, the differentiability ensures the $F(\cdotc)$ and $G(\,\cdot\,)$ are continuous, differential, and integrable.
Definition \ref{assumption:response-factor} covers a wide range of response functions in emerging hardware, including but not limited to PCM \citep{Burr2016,2020legalloJPD}, ReRAM \citep{Jang2014, Jang2015, stecconi2024analog}, ECRAM \citep{tang2018ecram, onen2022nanosecond}.
Fig. \ref{fig:response-factor} showcases three examples from the response functions class, including linear, non-linear but monotonic, and even non-monotonic functions.
\textbf{Warm-up: Weight drift on the optimal solution.}
Now we examine the conditional mean dynamics of {\AnalogSGD} and {\DigitalSGD}, which reveal a systematic \emph{weight drift} caused by the bias term in {\AnalogSGD}.
This provides intuition and motivates the implicit-penalty characterization established later in this section.
{\DigitalSGD} is stable at a critical point in expectation. Specifically, if
$\mbE_{\xi}[\nabla f(W_k;\xi)]=0$, then
$\mbE_{\xi_k}[W_{k+1}]
    = W_k - \alpha\,\mbE_{\xi_k}\!\left[\nabla f(W_k;\xi_k)\right]
    = W_k$,
where $\mbE_{\xi_k}[\cdotc]$ denotes expectation over $\xi_k$ conditional on $W_k$.
In contrast, it holds for {\AnalogSGD} that
\begin{align}
\label{equality:analog-SGD-drift}
    \mbE_{\xi_k}[W_{k+1}]
    &= W_k - \alpha\,\mbE_{\xi_k}\!\left[\nabla f(W_k;\xi_k)\right]\odot F(W_k)
       - \alpha\,\mbE_{\xi_k}\!\left[|\nabla f(W_k;\xi_k)|\right]\odot G(W_k) \\
    &= W_k - \alpha\,\mbE_{\xi_k}\!\left[|\nabla f(W_k;\xi_k)|\right]\odot G(W_k),
    \nonumber
\end{align}
and the right-hand side is generally not equal to $W_k$ whenever $G(W_k)\neq 0$. 

To illustrate the effect, consider the scalar case ($D=1$) and assume $G(\cdotc)$ is strictly monotone with
$G(w^\diamond)=0$, $G(w)\ge 0$ for $w\ge w^\diamond$, and $G(w)\le 0$ for $w\le w^\diamond$.
Then \eqref{equality:analog-SGD-drift} implies
$\mbE_{\xi_k}[w_{k+1}]\le w_k$ when $w_k\ge w^\diamond$ and
$\mbE_{\xi_k}[w_{k+1}]\ge w_k$ when $w_k\le w^\diamond$,
i.e., the conditional mean iterate drifts toward the symmetry point $w^\diamond$.
Unless the minimizer coincides with the symmetry point, this drift pushes the iterates away from the minimizer $W^\star$; consequently, {\AnalogSGD} does not, in general, optimize the original objective $f(\cdotc)$.

\textbf{Implicit penalty of {\AnalogSGD}.}
The weight drift identified above suggests that {\AnalogSGD} does not optimize $f(\cdotc)$ itself, but a \emph{penalized} objective that draws the iterate toward the symmetry point $W^\diamond$. 
To achieve the formulation of the penalized term, we construct a function $R(\cdotc)$ whose "derivative" is proportional to the bias term $\mbE_{\xi_k}\left[|\nabla f(W_k;\xi_k)|\right]\odot G(W_k)$. The following definition formalizes the construction and the resulting penalized objective.
\begin{definition}[Asymmetry ratio and penalized term]
\label{def:implicit-penalty}
Define the \emph{asymmetry ratio} $R(\,\cdot\,):\reals^D\to\reals^D$ and its component-wise anti-derivative $R_c(\,\cdot\,):\reals^D\to\reals^D$ by
\begin{align}
    R(W) := \frac{G(W)}{F(W)}
    \quad\text{and}\quad
    [R_c(W)]_d := \int_{[W^\diamond]_d}^{[W]_d} [R(W')]_d \,\diff{[W']_d},
    \quad d\in[D],
\end{align}
where the division is component-wise. Let $\Sigma := \mbE_\xi[|\nabla f(W^*; \xi)|]\in\reals^D$ denote the component-wise first moment of the stochastic gradient at the minimizer $W^*$. Define the \emph{penalty term} as $\varphi(W) := \la \Sigma, R_c(W)\ra$.
\end{definition}

The anti-derivative satisfies $\frac{\diff{}}{\diff{[W]_d}}[R_c(W)]_d = [R(W)]_d$ for each $d\in[D]$, so that $\nabla \la C, R_c(W)\ra = C \odot R(W)$ for any $C\in\reals^D$ and, in particular, 
$\nabla \varphi(W) = \Sigma\odot R(W)$, which equals the bias term divided by $F(W^*)$ in \eqref{equality:analog-SGD-drift} at $W=W^*$.
The penalty $R_c$ vanishes at $W^\diamond$, where $R(W^\diamond)=0$; if $R(\,\cdot\,)$ is moreover strictly monotone, then $R_c(W)\ge 0$ and attains its minimum at $W^\diamond$. 
Hence, the penalty shifts the minimizer of $f_\Sigma$ from $W^*$ toward $W^\diamond$, with strength set by the noise scale $\Sigma$.
To certify that {\AnalogSGD} optimizes \linktoIPwop, we do not characterize its exact limit; instead, we exhibit a single point that is simultaneously (i) nearly critical for the penalized objective $f(W)+\varphi(W)$ and (ii) nearly stationary for the mean {\AnalogSGD} dynamics. These two properties are measured, respectively, by the gradient norm $\|\nabla f_{\Sigma}(\cdotc)\|$ and by the \emph{scaled effective update}
\begin{align}\label{equality:scale_analog-SGD-drift}
    T(W) :=
    \mbE_\xi[\nabla f(W; \xi)] + \mbE_\xi[|\nabla f(W; \xi)|] \odot \frac{G(W)}{F(W)},
\end{align}
obtained by dividing the mean increment in \eqref{equality:analog-SGD-drift} by the positive factor $F(\cdotc)$. The zeros of $T(\cdotc)$ are exactly the conditional-expectation fixed points of {\AnalogSGD}: if $T(W_k)=0$, then the recursion \eqref{recursion:analog-SGD} (equivalently, \eqref{biased-update} with $\Delta Z_k=-\alpha\nabla f(W_k;\xi_k)$) is stable at $W_k$, i.e., $\mbE_{\xi_k}[W_{k+1}]=W_k$.
The two measures agree to leading order near $W^*$, since $\mbE_\xi[|\nabla f(W;\xi)|]\to\Sigma$ and hence $T(W)\to\nabla f_\Sigma(W)$ as $W\to W^*$.
The following theorem makes the claim precise: it produces a point at which both measures are an order smaller, in $\|W^\diamond-W^*\|$, than at either $W^*$ or $W^\diamond$, so that {\AnalogSGD} implicitly optimizes \linktoIPwop.

\begin{restatable}[Implicit Penalty]{theorem}{ThmImplicitRegularization}
    \label{theorem:implicit-regularization}
    Consider an open convex set $\ccalS$ such that $W^*, W^\diamond\in\ccalS$.
    Assume in the set $\ccalS$, $R(\cdotc)$ is continuously differentiable with a locally Lipschitz Jacobian, $f(\cdotc)$ is twice differentiable with a locally Lipschitz Hessian, and $\nabla f(\cdotc; \xi)$ is Lipschitz continuous with a common Lipschitz constant for all $\xi$.
    Let $J_R(W^\diamond)\in\reals^{D\times D}$ denote the Jacobian of $R(\cdotc)$ at $W^\diamond$ (since $R(\cdotc)$ acts component-wise,  $J_R(W^\diamond)$ is diagonal).
    Assume the Hessian satisfies $\mu^\prime I\preceq \nabla^2 f(W)\preceq L^\prime I$ in $\ccalS$ with positive $\mu^\prime$ and $L^\prime$, and there exists a positive constant $c_0$ such that $q'_+(W)<-c_0, q'_-(W)>c_0$ for all $W\in\ccalS$.
    Assume further that all components of $\Sigma$ are positive.
    The matrix
    $H := \nabla^2 f(W^*)+\Diag(\Sigma)\,J_R(W^\diamond)$
    is non-singular. Define
    \begin{align}
    \label{problem:implicit-regularization-solution}
    \tdW^*
    := H^{-1}\Big(\nabla^2 f(W^*)\,W^*+\Diag(\Sigma)\,J_R(W^\diamond)\,W^\diamond\Big)
    \end{align}
    where the operator \(\Diag(\Sigma) \in \mathbb{R}^{D \times D}\) denotes the diagonal matrix whose diagonal entries are given by $\Sigma \in \mathbb{R}^D$.
    Then {\AnalogSGD} implicitly optimizes the penalized problem
    \begin{align}
        \label{problem:implicit-regularization}
        (\mathcal{IP}):\quad
        \min_{W\in\reals^D}~ f_{\Sigma}(W) = f(W) + \la \Sigma, R_c(W)\ra,
    \end{align}
    with $R_c$ and $\Sigma$ as in Definition \ref{def:implicit-penalty}, in the sense that there exist constants $c_1, c_2>0$ such that
    \begin{align}
        \|\nabla f_{\Sigma}(\tdW^*)\| \le&\ c_1 \|W^\diamond-W^*\|^2, &
        \|\nabla f_{\Sigma}(W^*)\| \ge&\  c_2\|W^\diamond-W^*\|, &
        \|\nabla f_{\Sigma}(W^\diamond)\| \ge&\  c_2\|W^\diamond-W^*\|, 
        \nonumber\\
        \|T(\tdW^*)\| \le&\ c_1 \|W^\diamond-W^*\|^2,&
        \|T(W^*)\| \ge&\  c_2\|W^\diamond-W^*\|, &
        \|T(W^\diamond)\| \ge&\  c_2\|W^\diamond-W^*\|.
        \nonumber
    \end{align}
\end{restatable}

The proof of Theorem \ref{theorem:implicit-regularization} is deferred to Appendix \ref{section:proof-implicit-regularization}.
Theorem \ref{theorem:implicit-regularization} is local in nature and pertains to the regime where $W^*$ is close to the symmetry point $W^\diamond$.
In the degenerate case $W^*=W^\diamond$, the implicit penalty vanishes.
More generally, when $\|W^\diamond-W^*\|$ is small, Theorem \ref{theorem:implicit-regularization} implies that $\tdW^*$ provides a substantially better approximation than $W^*$ or $W^\diamond$ to both (i) a critical point of $f_\Sigma$ and (ii) a stationary point of the mean {\AnalogSGD} dynamics.
Accordingly, we interpret {\AnalogSGD} as implicitly optimizing the penalized objective \eqref{problem:implicit-regularization}.

Theorem \ref{theorem:implicit-regularization} suggests that the implicit penalty $R_c(W)$ forces the weight towards a symmetric point $W^\diamond$, which will be demonstrated by a simplified setting. Suppose that (i) the weight $W$ is a scalar ($D=1$), and $\nabla^2 f(W^*)$ and $J_R(W^\diamond)$ reduce to $f''(W^*)$ and $R'(W^\diamond)$, respectively; (ii) $f(\,\cdot\,)$ is strictly convex (i.e., $f''(W) > 0, \forall W$); and (iii) $R(W)$ is increasing monotonically, i.e., $R'(W^\diamond) \ge 0$, (e.g., it happens when $q_+'(W)\le 0, q_-'(W)\ge 0, \forall W$).
At that time, $\tdW^* = \frac{f''(W^*) W^*+R'(W^\diamond)\Sigma W^\diamond}{f''(W^*)+R'(W^\diamond)\Sigma}$. 
Since $R'(W^\diamond) \ge 0$, $\tdW^*$ is an interpolation of the optimal solution $W^*$ and the symmetric point $W^\diamond$. 
In the limiting deterministic case $\Sigma=0$, the implicit penalty vanishes and $\tdW^*=W^*$, although the lower bounds in Theorem \ref{theorem:implicit-regularization} require positive components of $\Sigma$.
On the other hand, if the noise is large, i.e., $\Sigma\to\infty$, the implicit penalty dominates the objective and $\tdW^*$ approaches $W^\diamond$.

Due to the implicit penalty, {\AnalogSGD} can not converge to the optimal value $f^*$ exactly. It motivates us to find approaches to enable exact convergence on analog hardware.
\section{Exact Solution under State-Dependent Bias via Bilevel Reformulation}
\label{section:TT}
In this section, we first investigate the conditions under which {\AnalogSGD} converges exactly based on the insights from Theorem \ref{theorem:implicit-regularization}. When the condition is not satisfied, we propose a bilevel optimization problem whose optimal solutions coincide with the original ones despite the implicit penalty, and we develop a generalized hyper-gradient algorithm, which we call {\ResidualLearning}, to solve it. 

\textbf{Construction of the bilevel optimization problem.}
Theorem \ref{theorem:implicit-regularization} suggests that {\AnalogSGD} implicitly optimizes the problem {\linktoIPwop}, whose minimizer is approximately $\tdW^*$, which differs from $W^*$ if the symmetric and optimal points do not coincide ($W^\diamond\ne W^*$) and the gradients are stochastic ($\Sigma\ne 0$). Therefore, the crucial technique that ensures convergence to $W^*$ is to avoid both conditions holding simultaneously. It motivates us to decouple the original problem into two coupled subproblems: (i) a shifted problem where the symmetric and optimal points coincide, and (ii) optimizing the shifting term whose deterministic gradient is available.
This paper assumes $W^\diamond$ is known, 
which is practically viable by alternatively applying positive and negative pulses \citep{kim2019zero}.
Under this assumption, it holds that $R_c(W^\diamond)=0$ by definition of $R(\cdotc)$.
\begin{assumption}[Known symmetric point]
    \label{assumption:zero-sp}
A symmetric point $W^\diamond$ is known.
\end{assumption}

\begin{figure}
    \centering
    \includegraphics[width=0.7\linewidth]{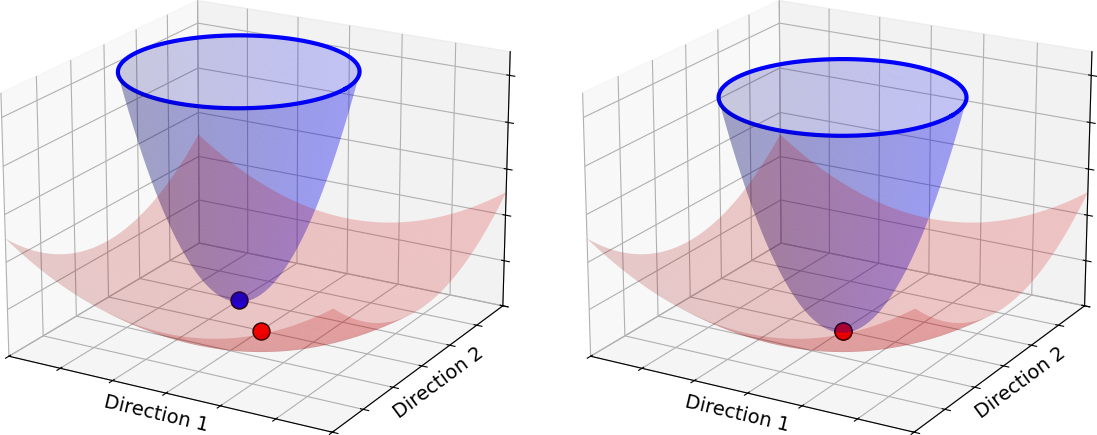}
    \caption{
    Illustration of the shifted problem in 2-dimensional space ($D=2$). 
    The blue and red surfaces are the landscape of $f(\cdotc)$ and $\la \Sigma, R_c(\cdotc)\ra$ 
    in problem $\mathcal{IP}$, respectively. 
    \textbf{(Left)} Before shifting, variable is $W$;
    \textbf{(Right)} After shifting, variable is $P$.
    }
    \label{fig:shift-problem}
\end{figure}

 \paragraph{Construct a shifted problem with its solution $W^\diamond$.}
For simplicity of illustration, we first assume that the minimizer $W^*$ is unique, and we will later relax this assumption to accommodate non-unique settings. Define a mapping $W\mapsto W+\gamma (P-W^\diamond)$ with a constant $\gamma > 0$, which adds a shift $\gamma(P-W^\diamond)$ on the original variable $W$. Instead of $W$, we optimize the auxiliary variable $P\in\reals^D$, which leads to a shifted problem $\min_{P\in\reals^D} f(W + \gamma (P-W^\diamond))$. 
With stochastic gradient $\nabla f(W+(P_k-W^\diamond); \xi_k)$ and a given $W$, Theorem \ref{theorem:implicit-regularization} suggests that running {\AnalogSGD} on $P$ implicitly optimizes the penalized problem
\begin{equation}
    \label{problem:shifted-P}
    \min_{P\in\reals^D} f(W + \gamma (P-W^\diamond)) + \la \Sigma, R_c(P)\ra.
\end{equation}
It can be verified that if $W=W^*$,  $P=W^\diamond$ is the optimal solution of \eqref{problem:shifted-P}, at which the optimal value is $f^*$ since 
$f(W + \gamma (P-W^\diamond)) + \la \Sigma, R_c(P)\ra \ge f(W) + \la \Sigma, R_c(W^\diamond)\ra \ge f^*$.
At that time, the stationary point of {\AnalogSGD} and the symmetric point coincide, and $P=W^\diamond$ remains an optimal solution despite the presence of the implicit penalty term; see Fig. \ref{fig:shift-problem} for the illustration of this idea.
Since $W^*$ is unknown in general, we propose to optimize the following non-ideal shifted problem for a given $W_k$, which converges to the ideal one \eqref{problem:shifted-P} when $W_k\to W^*$ 
\begin{equation}\label{eq.sub_p}
    P^*(W_k) = \argmin_{P\in\reals^D}\, f(W_k + \gamma (P-W^\diamond)).
\end{equation} 
Note that the shift $P^*(W)-W^\diamond = (W^* - W)/\gamma$, can be interpreted as the \emph{residual} from $W$ to $W^*$. From this perspective, \eqref{eq.sub_p} optimizes the $P$ to find the residual of $W$. As the stationary point and the symmetric point coincide only if $W=W^*$, the next step is to ensure the residual $\|W^*-W\|^2\to 0$.

\paragraph{Minimize the residual.} 
We construct another problem 
$\min_{W\in\reals^D} \frac{1}{2}\|W^*-W\|^2$.
However, since $W^*$ is unknown, we can not directly optimize the residual. Instead, we will construct another problem whose optimal solution is $W^*$ and its gradient can be approximated by $P^*(W)-W^\diamond$.
Since $P^*(W) = \frac{1}{\gamma}(W^*-W) + W^\diamond$, this is tantamount to minimizing
\begin{tcolorbox}[emphblock]
\centering
\vspace{-0.6\baselineskip}
\begin{align}
    \label{problem:residual-learning}
    (\mathcal{BP}):~~
    \min_{W\in\reals^D} ~\frac{\gamma^2}{2}\|P^*(W)-W^\diamond\|^2, \quad
    \st~
    P^*(W) = \argmin_{P\in\reals^D} ~f(W + \gamma (P-W^\diamond)).
\end{align}
\vspace{-1.0\baselineskip}
\end{tcolorbox}
\noindent
Therefore, if the auxiliary variable $P_k$ approximates the minimizer of problem \eqref{eq.sub_p}, i.e., $P_k\approx P^*(W_k)$, the gradient can be approximated by $\nabla [\frac{\gamma^2}{2}\|P^*(W_k)-W^\diamond\|^2] \approx - \gamma(P_k-W^\diamond)$.
Conditioned on $P_k$, this gradient is deterministic, leading to the variance term $\Sigma=0$ in Theorem \ref{theorem:implicit-regularization}. Therefore, optimizing $\|P^*(W)-W^\diamond\|^2$ via {\AnalogSGD} allows $W_k$ to converge to $W^*$. 
The following theorem shows that the bilevel reformulation {\linktoBPwop} is equivalent to {\linktoPwop}, which can be proved by plugging in the definition of $P^*(W)$ and the optimality of $W^*$ and verifying {\linktoPwop} has unique minimizer. 
\begin{theorem}[Equivalence between {\linktoPwop} and \linktoBPwop]
    \label{theorem:optimum-of-bilevel-problem}
    {\linktoBPwop} has the unique global minimizer $W^*$.
\end{theorem}

Now we propose a gradient-based method to solve the bilevel problem {\linktoBPwop}.
Let $\bar W_k := W_k + \gamma (P_k-W^\diamond)$ denote the shifted weight.
The lower-level variable $P$ is updated by the stochastic gradient $\nabla f(\bar W_k; \xi)$, following the recursion of {\AnalogSGD}.
For the upper-level variable $W$, since $P^*(W)=(W^*-W)/\gamma+W^\diamond$, the \textit{hyper-gradient} of the upper-level objective $\frac{\gamma^2}{2}\|P^*(W)-W^\diamond\|^2$ is proportional to $-\gamma(P^*(W)-W^\diamond)$. By approximating $P^*(W_k)$ with the auxiliary variable $P_{k+1}$, we update $W$ along $\gamma(P_{k+1}-W^\diamond)$.
Since both $W_k$ and $P_k$ incur the hardware-inspired bias, we replace $\Delta Z_k$ in \eqref{biased-update} by $-\alpha\nabla f(\bar W_k; \xi_k)$ and $\beta\gamma(P_{k+1}-W^\diamond)$, respectively, where $\beta$ is a stepsize. We then obtain the update 
\begin{tcolorbox}[emphblock]
\centering
\vspace{-0.8\baselineskip}
\begin{align}
    \label{recursion:HD-P}
    \vspace{0.5em}
    P_{k+1} = &\ P_k - \alpha\nabla f(\bar W_k; \xi_k)\odot F(P_k) 
    - \alpha|\nabla f(\bar W_k; \xi_k)|\odot G(P_k) \\
    \label{recursion:HD-W}
    \raisebox{0.53\baselineskip}[\height][\depth]{\vphantom{a}\texttt{Residual Learning}  \hspace{0.8em}}
    W_{k+1} =&\ W_k + \beta\gamma (P_{k+1}-W^\diamond)\odot F(W_k) - \beta\gamma|P_{k+1}-W^\diamond|\odot G(W_k).
\end{align}
\vspace{-1.2\baselineskip}
\end{tcolorbox}
\noindent
Featuring optimizing $P$ to estimate the residual and minimizing the residual itself at the lower- and upper-levels, respectively, the proposed method is termed {\ResidualLearning}.
It is noteworthy that, although state-dependent biases persist in every iteration of \eqref{recursion:HD-P} and \eqref{recursion:HD-W}, their cumulative impact is effectively mitigated over the long term. This property is rigorously substantiated by the theoretical analysis presented in the subsequent sections.

\textbf{Retaining exact convergence under the implicit penalty.}
Unlike in \S\ref{section:preliminary-response-function}, where the implicit penalty drives {\AnalogSGD} away from $W^*$, the implicit penalty here turns out to be harmless.
Involving stochastic gradients, update \eqref{recursion:HD-P} incurs the implicit penalty. Similar to \S\ref{section:preliminary-response-function}, we investigate the underlying penalized bilevel problem. 
Since $P_k$'s update involves the stochastic gradient $\nabla f(\bar W_k; \xi_k)$, an implicit penalty term $\la \Sigma, R_c(P)\ra$ is introduced in the lower-level problem. In the upper level, the approximated hyper-gradient $-\gamma(P_{k+1}-W^\diamond)$ is deterministic given $P_{k+1}$, and hence no implicit penalty term is imposed. Therefore, {\ResidualLearning} implicitly optimizes the following  problem
\begin{align}
    \label{problem:IBP}
    (\mathcal{IBP}):~
    \min_{W\in\reals^D} \frac{\gamma^2}{2}\|P^*(W)-W^\diamond\|^2,\quad
    P^*(W) = \argmin_{P\in\reals^D} f(W + \gamma (P-W^\diamond)) 
    + {\la \Sigma, R_c(P)\ra}.
\end{align}
Note that the implicit penalty term does not change the lower-level response at an optimal $W$, since $P=W^\diamond$ minimizes both $f(W+\gamma(P-W^\diamond))$ and $\la \Sigma, R_c(P)\ra$ if $W$ is a minimizer of $f(\cdotc)$. The following theorem demonstrates that problem {\linktoIBPwop} and problem {\linktoPwop} share the same solution in $W$.
\begin{theorem}[Equivalence between {\linktoPwop} and {\linktoIBPwop}]
    \label{theorem:optimum-of-bilevel-penalized-problem}
    {\linktoIBPwop} has the unique global minimizer $W^*$.
\end{theorem}
\begin{proof}
    Since $f(\cdotc)\ge f^*$ and $\la \Sigma, R_c(\cdotc)\ra \ge 0$, for any $P\in\reals^D$ it holds that
    $f(W^*+\gamma(P-W^\diamond))+\la \Sigma, R_c(P)\ra \ge f^*$, which holds at $P=W^\diamond$ because $R_c(W^\diamond)=0$; moreover, equality requires $W^*+\gamma(P-W^\diamond)=W^*$, and hence $P=W^\diamond$. Thus $P^*(W^*)=W^\diamond$, and the upper-level objective of {\linktoIBPwop} attains value $0$ at $W^*$. Since the upper-level objective is nonnegative, $W^*$ is a solution of {\linktoIBPwop}.
\end{proof}
Theorem \ref{theorem:optimum-of-bilevel-penalized-problem} ensures that even though there is an implicit penalty term in the lower-level objective, the minimizer of the penalized objective is the same as the original minimizer since the minimizer of the original objective coincides with the minimizer of the penalty term.

\section{Convergence Analysis of Residual Learning}
\label{section:convergence:RL}
The equivalence between the solution of {\linktoIBPwop} and the solution of {\linktoPwop} offers a useful intuition for residual learning.
This section moves beyond this equivalence and analyzes the convergence properties and the optimization difficulty introduced by state-dependent bias.
We first introduce a series of assumptions about the objective landscape, as well as the noise distribution.

\begin{assumption}[$L$-smoothness and $\mu$-strong convexity]  \label{assumption:Lip}
    The objective $f(W)$ is $L$-smooth
    and $\mu$-strongly convex,
    i.e., for any $W, W' \in \reals^{D}$, $
    \mu\|W-W'\| \le
    \|\nabla f(W)-\nabla f(W')\| \le L\|W-W'\|$.
\end{assumption}

\begin{assumption}[Unbiasness and bounded variance]
    \label{assumption:noise}
    The stochastic gradient is unbiased and has bounded variance, i.e., $\mbE_{\xi_k}[\nabla f(W_k;\xi_k)] = \nabla f(W_k)$ and $\mbE_{\xi_k}[\|\nabla f(W_k;\xi_k)-\nabla f(W_k)\|^2]\le\sigma^2$. 
\end{assumption}
Assumptions \ref{assumption:Lip}--\ref{assumption:noise} are standard in stochastic optimization \citep{bottou2018optimization}. The strong convexity in Assumption \ref{assumption:Lip} ensures a unique global minimizer and provides the curvature needed for our convergence analysis.
Now we establish the convergence of {\ResidualLearning}.

\begin{restatable}[Convergence rate of \ResidualLearning]{theorem}{ThmTTConvergenceScvx}
    \label{theorem:TT-convergence-scvx}
    Suppose Definition \ref{assumption:response-factor} and Assumptions
    \ref{assumption:zero-sp}--\ref{assumption:noise} hold.
    Let
    $\alpha=\Theta\lp \dfrac{q_{\max}\log K}{\gamma\mu q_{\min}^{2} K}\rp$,
    $\beta = \Theta\lp\dfrac{\alpha\gamma\mu q_{\min}}{q_{\max}}\rp$, and
    $\gamma \ge \Omega\lp \dfrac{L_S\sigma\sqrt{q_{\max}}}{\mu q_{\min}^{3/2}}\rp$.
    It holds that
    \begin{align}
        \label{inequality:TT-convergence-upperbound-faster}
        \mbE\Big[
        \underbrace{f(W_K + \gamma (P_K-W^\diamond)) - f^*}_{\text{Lower-level objective}}
        + C\underbrace{\lnorm P^*(W_K)-W^\diamond\rnorm^2}_{\text{Upper-level objective}}
        \Big] \le
        \tilde\ccalO\lp \frac{\kappa_1 \kappa_2^4 \sigma^2}{\mu K}\rp
    \end{align}
    where $C$ is a positive constant, $\kappa_1 := L/\mu$, and $\kappa_2 := q_{\max}/q_{\min}$.
\end{restatable}

\begin{proof}[Proof]
The proof relies on the following two lemmas, which provide sufficient descent of the lower- and upper-level objectives, respectively.

\begin{restatable}[Sufficient descent of lower-level objective]{lemma}{LemmaTTbarWDescentScvx}
    \label{lemma:TT-barW-descent}
    Under Assumptions \ref{assumption:Lip}-\ref{assumption:noise}, it holds that
    \begin{align}
        \label{inequality:TT-barW-descent}
        \restateeq{inequality-saved:TT-barW-descent}
    \end{align}
    Keeping $W_k$ fixed, define $\bar W_{k+\frac{1}{2}} := W_k + \gamma(P_{k+1}-W^\diamond)$; we also have
    \begin{align}
        \label{inequality:TT-P-in-barW-descent}
        \restateeq{inequality-saved:TT-P-in-barW-descent}
    \end{align}
\end{restatable}

Recall that we use the square norm of $P^*(W) = (W^* - W) / \gamma + W^\diamond$ to measure the convergence of the upper-level problem.

\begin{restatable}[Sufficient descent of upper-level objective]{lemma}{LemmaTTWDescentScvx}
    \label{lemma:TT-W-descent}
    Under Definition \ref{assumption:response-factor}, if $\beta\le1/(2q_{\max})$, then it holds that
    \begin{align}
        \label{inequality:TT-W-descent}
        \restateeq{inequality-saved:TT-W-descent}.
    \end{align}
\end{restatable}

The proofs of Lemmas \ref{lemma:TT-barW-descent} and \ref{lemma:TT-W-descent} are deferred to Appendix \ref{section:proof-lemma:TT-barW-descent} and \ref{section:proof-lemma:TT-W-descent}, respectively.
Lemmas \ref{lemma:TT-barW-descent} claims that the lower-level objective has sufficient descent in expectation, with three types of error terms in the \ac{RHS}: (i) $\lnorm G(P_k)\rnorm^2_\infty$ captures the bias introduced by the implicit penalty; (ii) $\alpha^2\gamma^2 L q_{\max}^2 \sigma^2$ comes from the variance of stochastic gradients and vanishes as $\alpha$ is sufficiently small; 
(iii) $\|W_{k+1}-W_k\|^2$ captures the impact of the upper-level variable drift on the lower-level objective.
Lemma \ref{lemma:TT-W-descent} suggests the sufficient descent of the upper-level objective with the error term $\|P_{k+1} - P^*(W_k)\|^2$ that captures the optimality of the upper-level update, and it also provides a negative $\|W_{k+1}-W_k\|^2$ term that will cancel the lower-level drift. Now we separately provide the estimates needed for the remaining error terms.

\textbf{(I) Intermediate term $\mbE_{\xi_k}[\|P_{k+1} - P^*(W_k)\|^2]$.}
By the definition of $P^*(W_k)$ and QG condition in Lemma \ref{lemma:QG}, we bound it by
\begin{align}
    \label{inequality:QG-Wk}
    \frac{2}{\mu}(f(W_k + \gamma (P_{k+1}-W^\diamond)) - f^*) \ge \|W_k + \gamma (P_{k+1}-W^\diamond) - W^*\|^2
    = \gamma^2 \|P_{k+1} - P^*(W_k)\|^2.
\end{align}
By inequality \eqref{inequality:TT-P-in-barW-descent} in Lemma \ref{lemma:TT-barW-descent} and \eqref{inequality:QG-Wk}, we bound the $\|P_{k+1} - P^*(W_k)\|^2$ term by
\begin{align}
    \label{inequality:TT-convergence-Lya-2-T5}
    &\ \mbE_{\xi_k}[\|P_{k+1} - P^*(W_k)\|^2]
    \\
    \lemark{a}&\ \frac{2}{\mu\gamma^2}  \mbE_{\xi_k}[f(W_k + \gamma (P_{k+1}-W^\diamond)) - f^*]
    \nonumber\\
    \lemark{b}&\ \frac{2}{\mu\gamma^2}
    \lp
        f(\bar W_k) - f^*
        -\frac{\alpha\gamma q_{\min}}{2}\|\nabla f(\bar W_k)\|^2
        + \frac{\alpha\gamma\sigma^2}{q_{\min}}\lnorm G(P_k)\rnorm_\infty^2
        + \frac{\alpha^2\gamma^2 L q_{\max}^2 \sigma^2}{2}
    \rp
    \nonumber\\
    \lemark{c}&\ \frac{2}{\mu\gamma^2}(f(\bar W_k) - f^*)
    + \frac{2\alpha\sigma^2}{\mu\gamma q_{\min}}\lnorm G(P_k)\rnorm_\infty^2
    + \frac{\alpha^2 L q_{\max}^2 \sigma^2}{\mu}
    \nonumber
\end{align}
where $(a)$ comes from \eqref{inequality:QG-Wk};
$(b)$ comes from \eqref{inequality:TT-P-in-barW-descent} in Lemma \ref{lemma:TT-barW-descent};
$(c)$ drops the non-positive gradient term.

\textbf{(II) State-dependent error term $\lnorm G(P_k) \rnorm^2_\infty$.} As $W^\diamond$ is a symmetric point, it holds that $G(W^\diamond)=0$. Together with the Lipschitz continuity of the response function (Definition \ref{assumption:response-factor}), we have 
\begin{align}
    \label{inequality:TT-convergence-Lya-4-T2}
    \lnorm G(P_k) \rnorm^2_\infty
    \le \lnorm G(P_k) \rnorm^2
    = \lnorm G(P_k)-G(W^\diamond)\rnorm^2
    \le L_S^2 \|P_k-W^\diamond\|^2.
\end{align}
Using $\|P_k-W^\diamond\|^2 \le 2\|P_k-P^*(W_k)\|^2 + 2\|P^*(W_k)-W^\diamond\|^2$, quadratic growth condition (Lemma~\ref{lemma:QG}) as well as the definition of $P^*(W_k)$, we have
\begin{align}
    \label{inequality:TT-convergence-Lya-4-T2-2}
    \lnorm G(P_k) \rnorm^2_\infty
    \le&\ \frac{4L_S^2}{\mu\gamma^2}(f(\bar W_k)-f^*) + 2L_S^2\|P^*(W_k)-W^\diamond\|^2.
\end{align}

With all the inequalities and lemmas above, we are ready to prove Theorem \ref{theorem:TT-convergence-scvx}.
Define a Lyapunov function by
\begin{align}
    \label{eq:lyapunov-def}
    V_k :=&\ f(\bar W_k)-f^*
    + \frac{C\gamma^2}{2}\|P^*(W_k)-W^\diamond\|^2.
\end{align}
\noindent
By Lemmas \ref{lemma:TT-barW-descent} and \ref{lemma:TT-W-descent}, and bounding the residual term $\|P_{k+1}-P^*(W_k)\|^2$ and the hardware-induced term $\lnorm G(P_k)\rnorm_\infty^2$ via \eqref{inequality:TT-convergence-Lya-2-T5} and
\eqref{inequality:TT-convergence-Lya-4-T2-2},
$V_k$ has sufficient descent in expectation:
\begin{align}
    \label{inequality:TT-convergence-Lya-1}
		&\
    \mbE_{\xi_k}[V_{k+1}]
    =
    \mbE_{\xi_k}\lB
        f(\bar W_{k+1})-f^*
        + \frac{C\gamma^2}{2}\|P^*(W_{k+1})-W^\diamond\|^2
    \rB
    \\
    \le&\ \lp 1-\frac{\alpha\gamma\mu q_{\min}}{2}\rp
    (f(\bar W_k)-f^*)
    + \frac{\alpha\gamma\sigma^2}{q_{\min}}\lnorm G(P_k)\rnorm^2_\infty
    + \alpha^2\gamma^2 L q_{\max}^2 \sigma^2
    + \frac{2}{\alpha\gamma q_{\min}}~\mathbb{E}_{\xi_k}\lB\|W_{k+1}-W_k\|^2\rB
    \nonumber\\
    &\ 
    + \frac{C\gamma^2}{2} \Bigg(~
        \lp 1- \beta q_{\min}\rp\|P^*(W_k)-W^\diamond\|^2
        + \beta q_{\max}
        \mbE_{\xi_k}[\|P_{k+1} - P^*(W_k)\|^2]
        \bkeq\hspace{3em}
        -\frac{1}{2\beta\gamma^2 q_{\max}}
        \mathbb{E}_{\xi_k}\lB\|W_{k+1}-W_k\|^2\rB
    ~\Bigg)
    \nonumber\\
    \le&\ V_k
    -\frac{\alpha\gamma\mu q_{\min}}{2}(f(\bar W_k)-f^*)
    + \frac{\alpha\gamma\sigma^2}{q_{\min}}\lnorm G(P_k)\rnorm^2_\infty
    + \alpha^2\gamma^2 L q_{\max}^2 \sigma^2
    \nonumber\\
    &\ + \frac{\beta C\gamma^2q_{\max}}{2}
    \mbE_{\xi_k}[\|P_{k+1} - P^*(W_k)\|^2]
    - \frac{\beta q_{\min}}{2} C\gamma^2
    \|P^*(W_k)-W^\diamond\|^2
    \nonumber\\
    &\ + \lp\saveeq{definition:coefficient-A3}{
        \frac{2}{\alpha\gamma q_{\min}}
        - \frac{C}{4\beta q_{\max}}
        }\rp
    \mathbb{E}_{\xi_k}\lB\|W_{k+1}-W_k\|^2\rB
    \nonumber\\
    \le&\ V_k
    - \underbrace{
        \lp\saveeq{definition:coefficient-A1}{
        \frac{\alpha\gamma\mu q_{\min}}{2}
        - \frac{2}{\mu\gamma^2}\cdot\frac{\beta C\gamma^2q_{\max}}{2}
        - \frac{8\alpha\sigma^2 L_S^2}{\mu\gamma q_{\min}}
        }\rp
    }_{=: A_1}
    (f(\bar W_k)-f^*)
    + \frac{3\alpha^2\gamma^2 L q_{\max}^2 \sigma^2}{2}
    \nonumber\\
    &\ - \underbrace{
        \lp\saveeq{definition:coefficient-A2}{
        \frac{\beta q_{\min}}{2} C\gamma^2
        - \frac{4\alpha\gamma\sigma^2 L_S^2}{q_{\min}}
        }\rp
    }_{=: A_2}
    \|P^*(W_k)-W^\diamond\|^2
    + \underbrace{
        \lp\saveeq{definition:coefficient-A3}{
        \frac{2}{\alpha\gamma q_{\min}}
        - \frac{C}{4\beta q_{\max}}
        }\rp
    }_{=: A_3}
    \mathbb{E}_{\xi_k}\lB\|W_{k+1}-W_k\|^2\rB.
    \nonumber
\end{align}
The coefficient of $(f(\bar W_k)-f^*)$ generated by \eqref{inequality:TT-convergence-Lya-2-T5} is kept in $A_1$, whereas the noise and hardware terms it produces are bounded using
$\frac{2}{\mu\gamma^2}\cdot\frac{\beta C\gamma^2q_{\max}}{2}\le 1$, which holds under a sufficiently small $\beta$ choice below.

Now we instantiate the hyper-parameters according to the following lemma.

\begin{restatable}[Coefficient selection]{lemma}{LemmaTTCoefficientLowerBoundsScvx}
    \label{lemma:TT-coefficient-lower-bounds}
    Under the choice of $\alpha, \beta, C, \gamma$ as follows, 
    \begin{align}
        \label{eq:alpha-beta-choice}
        \alpha
        := \frac{32 q_{\max}}{\gamma\mu q_{\min}^{2} K}
        \log&\lp\frac{\mu^2 q_{\min}^{4} V_0 K}{2048Lq_{\max}^{4}\sigma^2}\rp
        = \tdTheta\lp \frac{q_{\max}}{\gamma \mu q_{\min}^{2}K}\rp,
        \qquad
        \beta := \frac{\alpha\gamma\mu q_{\min}}{16 q_{\max}} = \tdTheta\lp \frac{1}{q_{\min}K}\rp. \\
        \label{eq:C-gamma-choice}
        C :=&\ \frac{16\beta q_{\max}}{\alpha\gamma q_{\min}}
        = \mu,
        \qquad
        \gamma \ge \frac{16L_S\sigma\sqrt{q_{\max}}}{\mu q_{\min}^{3/2}} = \tdOmega\lp\frac{L_S\sigma\sqrt{q_{\max}}}{\mu q_{\min}^{3/2}}\rp,
    \end{align}
    the coefficients defined in \eqref{inequality:TT-convergence-Lya-1} satisfy $A_1 > \rho$, $A_2 > \rho C\gamma^2/2$, and $A_3\le0$ with $\rho := 
    \dfrac{\beta q_{\min}}{2}
    =\dfrac{\alpha\gamma\mu q_{\min}^{2}}{32 q_{\max}}$.
\end{restatable}

The proof of Lemma \ref{lemma:TT-coefficient-lower-bounds} is deferred to \S\ref{section:proof-lemma:TT-coefficient-lower-bounds}.
Lemma \ref{lemma:TT-coefficient-lower-bounds} together with \eqref{inequality:TT-convergence-Lya-1} gives the following one-step contraction
\begin{align}
    \label{inequality:TT-convergence-Lya-5}
    \mbE_{\xi_k}[V_{k+1}]
    \le \lp 1-\dfrac{\alpha\gamma\mu q_{\min}^{2}}{32 q_{\max}}\rp V_k
    + 2\alpha^2\gamma^2 L q_{\max}^2 \sigma^2.
\end{align}
Telescoping \eqref{inequality:TT-convergence-Lya-5} from $k=0$ to $K-1$ and taking full expectation yield
\begin{align}
    \label{inequality:TT-convergence-Lya-6}
    \mbE[V_K]
    \le&\ \lp 1- \frac{\alpha\gamma\mu q_{\min}^{2}}{32 q_{\max}}\rp^K V_0
    + \frac{64\alpha\gamma L q_{\max}^{3} \sigma^2}{\mu q_{\min}^{2}}
    \le \exp\lp - \frac{\alpha\gamma\mu q_{\min}^{2} K}{32q_{\max}}\rp V_0
    + \frac{64\alpha\gamma L q_{\max}^{3} \sigma^2}{\mu q_{\min}^{2}}.
\end{align}
With the choice of $\alpha$ in \eqref{eq:alpha-beta-choice}, inequality \eqref{inequality:TT-convergence-Lya-6} becomes
\begin{align}
    \label{inequality:TT-convergence-Lya-7}
    \mbE[V_K]
    =&\ \exp\lp - \log\lp\frac{\mu^2 q_{\min}^{4} V_0 K}{2048Lq_{\max}^{4}\sigma^2}\rp\rp V_0
    + \frac{2048Lq_{\max}^{4}\sigma^2}{\mu^2 q_{\min}^{4}K}
    \log\lp\frac{\mu^2 q_{\min}^{4} V_0 K}{2048Lq_{\max}^{4}\sigma^2}\rp \\
    =&\ \frac{2048Lq_{\max}^{4}\sigma^2}{\mu^2 q_{\min}^{4}K}
    \lp 1 + \log\lp\frac{\mu^2 q_{\min}^{4} V_0 K}{2048Lq_{\max}^{4}\sigma^2}\rp\rp
    \le
    \tilde\ccalO\lp \frac{\kappa_1 \kappa_2^{4} \sigma^2}{\mu K}\rp.
    \nonumber
\end{align}
The proof of Theorem \ref{theorem:TT-convergence-scvx} is completed.
\end{proof}

Theorem \ref{theorem:TT-convergence-scvx} claims that {\ResidualLearning} converges in terms of both the upper and lower level objectives of problem {\linktoBPwop}. 
Compared with the classical $\mu$-strongly convex benchmark, {\ResidualLearning} preserves the $\ccalO\lp {\log K}/{K}\rp$ dependence up to logarithmic factors, which matches the rate of {\DigitalSGD} 
\citep{hazan2014beyond,karimi2016linear}.
Theorem \ref{theorem:TT-convergence-scvx} also suggests that the upper bound depends on two constants: $\kappa_1$ and $\kappa_2$. 

\textbf{Dependence on the objective condition number.}
In the classical convergence analysis of gradient-based optimization algorithms, $\kappa_1:= L/\mu$ denotes the \textit{objective condition number}, which characterizes the curvature of the objective landscape. 
If the objective is further twice differentiable, $L$ and $\mu$ are the maximal and minimal eigenvalues of the objective's Hessian. 
\citet{agarwal2009information} established the information-theoretic lower bound $\Omega\lp \sigma^2/(\mu K)\rp$ for stochastic $\mu$-strongly convex optimization. 
Using Polyak-Ruppert averaging, \citet{polyak1992acceleration} showed that the averaged iterate $f(Z_K)-f^*$ attains the optimal function value gap $\ccalO\lp \sigma^2/(\mu K)\rp$, where $Z_K$ denotes the weighted average of the iterates generated by {\DigitalSGD}.
This lower bound depends explicitly on $\mu$ but not on $\kappa_1$, indicating that the dependence on $\mu$ is fundamental, whereas the dependence on $\kappa_1$ can be removed.
However, the average iterate $Z_K$ introduces an additional variable to record the history of the iterates, which also suffers from state-dependent bias. Therefore, we consider only the last-iterate convergence in this work.
Under this setting, \citet{bottou2018optimization} showed that {\DigitalSGD} attains that the last-iterate function value gap $f(W_K)-f^*$ is bounded by a near-matching upper bound $\ccalO\lp \kappa_1\sigma^2/(\mu K)\rp$, with an extra dependency on $\kappa_1$ compared with the lower bound.
Analogously, Theorem \ref{theorem:TT-convergence-scvx} introduces a linear factor in $\kappa_1$, which is comparable to the classical upper bound. 
In light of the discussion above, the dependence on the objective condition number remains improvable via the Polyak-Ruppert averaging counterpart tailored for the biased hardware. We will investigate this in future work.

\textbf{Dependence on the hardware condition number.}
Similar to $\kappa_1$, we investigate the \textit{hardware condition number} $\kappa_2:= q_{\max}/q_{\min}$ as the ratio of the maximum to the minimum response function. 
The bound in Theorem \ref{theorem:TT-convergence-scvx} reveals a significantly stronger dependence on the hardware condition number ($\kappa_2^4$) compared to the objective condition number ($\kappa_1$), which implies that the convergence behavior is more sensitive to state-dependent bias induced by the hardware.
A natural question arises: is this strong dependence on $\kappa_2$ fundamental, or can it be improved by a more refined analysis?
The following theorem provides a lower bound on the convergence rate of {\ResidualLearning} with respect to $\kappa_2$, which suggests that the dependence on $\kappa_2$ is at least polynomial.

\begin{theorem}[Lower bound of the convergence error]
    \label{theorem:lower-bound}
    There exists a hard instance such that, under the same conditions as
    Theorem \ref{theorem:TT-convergence-scvx}, it holds for {\ResidualLearning} with sufficiently large $K$ that
    \begin{align}
        \label{inequality:RL-lower-bound}
        \mathbb{E}\bigl[
            f(W_K + \gamma (P_K-W^\diamond)) - f^*
            + \frac{C\gamma^2}{2}\lVert P^*(W_K)-W^\diamond\rVert^2
            \bigr]
        \ge \tilde\Omega\!\left(
            \frac{\kappa_2^2\sigma^2}{\mu K}
        \right).
    \end{align}
\end{theorem}

\begin{proof}[Proof]
At the center of the proof is a one-dimensional hard instance, analyzed in three stages.
First, we construct the instance, which is a strongly convex quadratic driven by skewed two-point noise.
Whenever the current iterate lies in the confinement region, the conditional mean of the analog increment has an explicit form; this calculation is isolated in Lemma~\ref{lemma:conditional-mean-analog-increment}.
Second, Lemma~\ref{lemma:terminal-mean-square-lower-bound} combines a global lower bound on the conditional-mean dynamics with the variance generated while the current iterate is confined.
This filtration-adapted estimate does not condition the noise on the future confinement event.
Third, Lemma~\ref{lemma:confinement} shows that every current iterate is confined with constant probability.
Taking total expectations in the one-step estimate and unrolling it then gives \eqref{inequality:RL-lower-bound}.

\medskip
\noindent\textbf{The hard instance and skewed noise.}
Consider the quadratic objective $f(W) = \frac{\mu}{2}(W-W^\star)^2$.
For this one-dimensional quadratic, the smoothness and strong-convexity constants coincide, so $L=\mu$ throughout the construction.
We consider the case where the inner and outer variables have different response functions. The outer variable is realized on a symmetric array with constant response $q_+(\cdotc)\equiv q_-(\cdotc)\equiv q_{\min}$ ($F(W)\equiv q_{\min}$, $G(W)\equiv0$). 
For the inner variable, we place the symmetric point $W^\diamond=0$ at the ceiling of the response, i.e., 
\begin{align}
    \label{eq:symmetric-point-at-ceiling}
    q_+(0)=q_-(0)=q_{\max}
    \qquad\Longleftrightarrow\qquad
    F(0)=q_{\max},\quad G(0)=0 ,
\end{align}
while the response attains $q_{\min}$ at some other state. Under this setting, \texttt{Residual Learning} reduces to
\begin{align}
    \label{eq:residual-learning-iterations}
    P_{k+1}
    =&\
    P_k-\alpha F(P_k)\bigl(\mu(\bar W_k-W^\star)+\xi_k\bigr)
    -\alpha|\mu(\bar W_k-W^\star)+\xi_k|~G(P_k),\\
    \label{eq:quadratic-example-simplified-W-update}
    W_{k+1}
    =&\
    W_k+\beta\gamma q_{\min} P_{k+1}.
\end{align}

The key step in proving the convergence lower bound is to analyze
$\mathbb{E}_{\xi_k}[F(P_k)\nabla f(\bar W_k;\xi_k)+|\nabla f(\bar W_k;\xi_k)|G(P_k)]$, where the current iterates $P_k$ and $\bar W_k$ are fixed and the expectation is taken only over the fresh noise $\xi_k$.
This is generally difficult because the analog operator applies a sign-dependent response to the stochastic gradient.
In our construction, the two-point noise is chosen so that the analog expectation can be evaluated explicitly whenever the deterministic gradient is confined to a sufficiently small range.
To formalize the one-step condition, define event $\mathcal{E}_k$ for every $0\le k\le K$, as well as the full-horizon confinement event $\mathcal{E}$, by
\begin{align}
    \label{eq:confinement-event-k}
    \mathcal{E}_k
    :=
    \left\{
    \max\bigl\{|\bar W_k-W^\star|,\,|W_k-W^\star|\bigr\}
    \le
    \frac{\sigma}{\mu}\sqrt{\frac{q_{\min}}{q_{\max}}}
    \right\}
    \quad\text{and}\quad
    \mathcal{E}
    :=
    \bigcap_{k=0}^{K}\mathcal{E}_k.
\end{align}
On this event, the quadratic structure gives
$|\nabla f(\bar W_k)|=\mu|\bar W_k-W^\star|
\le\sigma\sqrt{q_{\min}/q_{\max}}$ for every $0\le k\le K$, and the inner state
obeys
$|P_k|=\tfrac1\gamma|\bar W_k-W_k|
\le\tfrac1\gamma\bigl(|\bar W_k-W^\star|+|W_k-W^\star|\bigr)
\le\tfrac{2\sigma}{\mu\gamma}\sqrt{q_{\min}/q_{\max}}$.
Thus, if the two noise atoms have magnitudes at least
$\sigma\sqrt{q_{\min}/q_{\max}}$, the sign of each stochastic gradient
$\nabla f(\bar W_k;\xi_k)$ is determined by the corresponding noise atom. 
This removes the absolute value in the effective gradient increment and yields a closed-form conditional expectation.
Inspired by that, define the stochastic gradient by $\nabla f(\bar W_k; \xi_k) = \nabla f(\bar W_k)+\xi_k$, and set
\begin{align}
    \label{eq:skewed-noise}
    \xi_k
    :=
    \begin{cases}
        \sigma\sqrt{\frac{1-p_k}{p_k}},
        &\text{with probability }p_k,\\
        -\sigma\sqrt{\frac{p_k}{1-p_k}},
        &\text{with probability }1-p_k,
    \end{cases}
    ~~\text{where }
    p_k
    :=
    \frac{q_+(P_k)}{q_+(P_k)+q_-(P_k)}.
\end{align}
Conditional on the current iterates, it can be verified directly that $\mathbb{E}_{\xi_k}[\xi_k]
    =
    0,
    \mathbb{E}_{\xi_k}[\xi_k^2]
    =
    \sigma^2$ and Assumption~\ref{assumption:noise} holds.
In addition, since $p_j=\dfrac{q_+(P_j)}{q_+(P_j)+q_-(P_j)}$ and $1-p_j=\dfrac{q_-(P_j)}{q_+(P_j)+q_-(P_j)}$, the two endpoint magnitudes
satisfy
\begin{align}
    \label{eq:trajectory-sign-thresholds}
    \sigma\sqrt{\frac{1-p_j}{p_j}}
    &=
    \sigma\sqrt{\frac{q_-(P_j)}{q_+(P_j)}}
    \ge
    \sigma\sqrt{\frac{q_{\min}}{q_{\max}}},
    \quad
    \sigma\sqrt{\frac{p_j}{1-p_j}}
    =
    \sigma\sqrt{\frac{q_+(P_j)}{q_-(P_j)}}
    \ge
    \sigma\sqrt{\frac{q_{\min}}{q_{\max}}},
\end{align}
where the inequalities use the regularity of the response functions
(Definition~\ref{assumption:response-factor}). Therefore, whenever
$\mathcal{E}_k$ occurs,
\begin{align}
    \label{eq:conditional-absolute-value-sign-condition}
    -\sigma\sqrt{\frac{1-p_k}{p_k}}
    \le
    \nabla f(\bar W_k)
    \le
    \sigma\sqrt{\frac{p_k}{1-p_k}},
\end{align}
and consequently $\nabla f(\bar W_k) + \sigma\sqrt{\frac{1-p_k}{p_k}} \ge 0$ and $\nabla f(\bar W_k) - \sigma\sqrt{\frac{p_k}{1-p_k}} \le 0$. This property enables us to write the conditional expectation of the increment in closed form, as stated in the following lemma.

\begin{lemma}[Conditional mean of the analog increment]
    \label{lemma:conditional-mean-analog-increment}
    Fix $P_k$ and $\bar W_k$. If the sign condition \eqref{eq:conditional-absolute-value-sign-condition} holds, then the expectation over the fresh noise $\xi_k$ satisfies
    \begin{align}
        \label{eq:quadratic-example-A-mean}
        \mathbb{E}_{\xi_k}\bigl[
            -A\bigl(-\nabla f(\bar W_k;\xi_k);P_k\bigr)
        \bigr]
        =
        \frac{q_+(P_k)q_-(P_k)}{F(P_k)}\,\nabla f(\bar W_k)
        +\sigma\frac{\sqrt{q_+(P_k)q_-(P_k)}}{F(P_k)}\,G(P_k).
    \end{align}
\end{lemma}
The proof of Lemma \ref{lemma:conditional-mean-analog-increment} is deferred to \S\ref{section:proof-lemma:conditional-mean-analog-increment}.
By \eqref{eq:conditional-absolute-value-sign-condition}, the sign condition
needed by Lemma~\ref{lemma:conditional-mean-analog-increment} holds on $\mathcal{E}_k$. Consequently,
\begin{align}
    \label{eq:conditional-mean-recursion-P}
    \mathbb{E}_{\xi_k}\bigl[
        P_{k+1}
    \bigr]
    =&\
    P_k
    -\alpha
    \frac{q_+(P_k)q_-(P_k)}{F(P_k)}
    \mu(\bar W_k-W^\star)
    -\alpha\sigma
    \frac{\sqrt{q_+(P_k)q_-(P_k)}}{F(P_k)}
    G(P_k).
\end{align}
Substituting \eqref{eq:conditional-mean-recursion-P} into
\eqref{eq:quadratic-example-simplified-W-update} gives
\begin{align}
    \label{eq:quadratic-example-W-dynamics}
    \mathbb{E}_{\xi_k}\bigl[
        W_{k+1}-W^\star
    \bigr]
    =&\
    \Bigl(1-\alpha\beta\gamma q_{\min} \mu \frac{q_+(P_k)q_-(P_k)}{F(P_k)}\Bigr)(W_k-W^\star)
    \\
    &\
    +\beta\gamma q_{\min}\Bigl(1-\alpha\gamma \mu \frac{q_+(P_k)q_-(P_k)}{F(P_k)}\Bigr)P_k
    -\alpha\beta\gamma q_{\min}\sigma
    \frac{\sqrt{q_+(P_k)q_-(P_k)}}{F(P_k)}
    G(P_k).
    \nonumber
\end{align}
The inverse shear
$\beta\gamma q_{\min}P_k=\beta q_{\min}\bigl((\bar W_k-W^\star)-(W_k-W^\star)\bigr)$, together with \eqref{eq:conditional-mean-recursion-P} and \eqref{eq:quadratic-example-W-dynamics}, allows us to express the conditional means of $W_{k+1}-W^\star$ and $\bar W_{k+1}-W^\star$ as linear functions of the pair $(W_k-W^\star,\,\bar W_k-W^\star)$. The effective response
$\tfrac{q_+(P_k)q_-(P_k)}{F(P_k)}=\tfrac{2q_+(P_k)q_-(P_k)}{q_+(P_k)+q_-(P_k)}$ is the harmonic mean of $q_+(P_k)$ and $q_-(P_k)$, so the resulting two-dimensional linear recursion is
\begin{align}
    \label{eq:pair-W-mean}
    \mathbb{E}_{\xi_k}\bigl[W_{k+1}-W^\star\bigr]
    =&\
    \Bigl(1-\beta q_{\min}\Bigr)(W_k-W^\star)
    +\beta q_{\min}\Bigl(1-\alpha\gamma\mu \frac{q_+(P_k)q_-(P_k)}{F(P_k)}\Bigr)(\bar W_k-W^\star)
    \bkeqwn
    -\alpha\beta\gamma q_{\min}\sigma\frac{\sqrt{q_+(P_k)q_-(P_k)}}{F(P_k)}\,G(P_k), \nonumber\\
    \label{eq:pair-barW-mean}
    \mathbb{E}_{\xi_k}\bigl[\bar W_{k+1}-W^\star\bigr]
    =&\ 
    \Bigl(1+\beta q_{\min}-\alpha\gamma(1+\beta q_{\min})\mu \frac{q_+(P_k)q_-(P_k)}{F(P_k)}\Bigr)(\bar W_k-W^\star)
    \\
    &\ 
    -\beta q_{\min}(W_k-W^\star)
    -\alpha\gamma(1+\beta q_{\min})\sigma\frac{\sqrt{q_+(P_k)q_-(P_k)}}{F(P_k)}\,G(P_k).
    \nonumber
\end{align}
Stacking the scaled pair $W_k-W^\star$ and $\bar W_k-W^\star$ into a two-dimensional vector, we define 
\begin{align}
    \label{eq:pair-R-def}
    R_k := \sqrt{\frac{\mu}{2}}\,
    \Bigl[\,\Bigl(\tfrac{q_{\min}}{q_{\max}}\Bigr)^{1/4}(W_k-W^\star),\ \bar W_k-W^\star\,\Bigr]^\top .
\end{align}

\medskip
\noindent\textbf{Mean-square lower bound.}
The preceding displays characterize the conditional mean dynamics on $\mathcal{E}_k$.
To avoid conditioning the fresh noise on the future event $\mathcal{E}$, the next lemma uses a global conditional-mean bound and invokes confinement only for the conditional variance.
The indicator in the resulting estimate is measurable before $\xi_k$ is drawn.

\begin{lemma}[One-step mean-square lower bound]
    \label{lemma:terminal-mean-square-lower-bound}
    For the hard instance above and the same setting as Theorem \ref{theorem:lower-bound}, it holds for sufficiently large $K$ and every $0\le k\le K-1$ that
    \begin{align}
        \label{eq:terminal-mean-square-lower-bound}
        \mathbb{E}_{\xi_k}\bigl[\lVert R_{k+1}\rVert^2\bigr]
        \ge&\
        \left(1-12\alpha\mu\gamma q_{\max}\right)\lVert R_k\rVert^2
        +\frac{\mu}{8}\alpha^2\gamma^2q_{\max}^2\sigma^2
        \mathbf{1}_{\mathcal{E}_k}.
    \end{align}
    Here, $\mathbf{1}_{\mathcal{E}_k}$ denotes the indicator of the one-step confinement event $\mathcal{E}_k$ defined in \eqref{eq:confinement-event-k}.
\end{lemma}
The proof of Lemma \ref{lemma:terminal-mean-square-lower-bound} is deferred to \S\ref{section:proof-lemma:terminal-mean-square-lower-bound}.
The remaining step is to verify that the indicator in \eqref{eq:terminal-mean-square-lower-bound} is active with constant probability at every iteration.

\begin{lemma}[Constant-probability confinement]
    \label{lemma:confinement}
    For the one-dimensional hard instance, choose
    \begin{align}
        \label{eq:initial-lyapunov-choice}
        P_0
        &:=
        0,
        \qquad
        W_0
        :=
        W^\star
        +\frac{\sigma}{4\mu}
        \left(\frac{q_{\min}}{q_{\max}}\right)^{3/4}.
    \end{align}
    Under this initialization, the event $\mathcal{E}$ in \eqref{eq:confinement-event-k} satisfies $\mathbb{P}(\mathcal{E})\ge\frac{1}{2}$ for sufficiently large $K$.
\end{lemma}
The proof of Lemma \ref{lemma:confinement} is deferred to \S\ref{section:proof-lemma:confinement}.
We now take the total expectations in \eqref{eq:terminal-mean-square-lower-bound}.
By definition \eqref{eq:confinement-event-k}, $\mathcal{E}\subseteq\mathcal{E}_k$ for every $0\le k\le K-1$.
Hence Lemma~\ref{lemma:confinement} implies $\mathbb{E}[\mathbf{1}_{\mathcal{E}_k}]=\mathbb{P}(\mathcal{E}_k)\ge\mathbb{P}(\mathcal{E})\ge1/2$.
The coefficient in the resulting scalar recursion is nonnegative for sufficiently large $K$, as established in the proof of Lemma~\ref{lemma:terminal-mean-square-lower-bound}.
Dropping the nonnegative initial term, unrolling the scalar recursion from $k=0$ to $K-1$, and applying the step size choice \eqref{eq:alpha-beta-choice} yield
\begin{align}
    \label{eq:pair-floor-rate}
    \mathbb{E}\bigl[\lVert R_K\rVert^2\bigr]
    \ge
    \frac{\mu}{16}\alpha^2\gamma^2q_{\max}^2\sigma^2
    \sum_{k=0}^{K-1}
    \left(1-12\alpha\mu\gamma q_{\max}\right)^k
    \ge
    \frac{\alpha\gamma q_{\max}\sigma^2}{192}
    = \tilde\Omega\!\lp
        \frac{\kappa_2^2\sigma^2}{\mu K}
    \rp.
\end{align}
Finally, since $W^\diamond=0$, $\bar W_K=W_K+\gamma(P_K-W^\diamond)$, $P^*(W_K)=(W^\star-W_K)/\gamma$, and $C=\mu$, the hard instance satisfies
\begin{align}
    \label{eq:pair-objective-identity}
    \mathbb{E}\bigl[V_K\bigr]
    =&\
    \mathbb{E}\lB\,
        \frac{\mu}{2}(\bar W_K-W^\star)^2
        +\frac{\mu}{2}(W_K-W^\star)^2
    \,\rB
    \\
    \ge&\
    \mathbb{E}\lB\,
        \frac{\mu}{2}(\bar W_K-W^\star)^2
        +\frac{\mu}{2}\sqrt{\frac{q_{\min}}{q_{\max}}}(W_K-W^\star)^2
    \,\rB
    =\mathbb{E}\bigl[\lVert R_K\rVert^2\bigr].
    \nonumber
\end{align}
Combining \eqref{eq:pair-floor-rate} and \eqref{eq:pair-objective-identity} proves \eqref{inequality:RL-lower-bound}.
\end{proof}

Theorem \ref{theorem:lower-bound} complements the upper bound in Theorem \ref{theorem:TT-convergence-scvx} by showing that the hardware dependence cannot, in general, be removed from the rate.
In particular, even on a favorable problem class where {\ResidualLearning} achieves the optimal $\log K/K$ iteration dependence up to constants, the constant factor must still deteriorate polynomially with $\kappa_2$ in the worst case.
Therefore, the gap between the $\kappa_2^4$ upper bound and the $\kappa_2^2$ lower bound remains open, but the lower bound indicates that a $\kappa_2$-independent guarantee is impossible.
This result further suggests that $\kappa_2$ serves as an intrinsic measure of the hardware-induced optimization difficulty.
It remains an open question to design an optimal algorithm in terms of $\kappa_2$ and to determine the tight dependence on $\kappa_2$ for this class of problems in future work.

\begin{figure*}[t]
    \footnotesize
    \centering
    \includegraphics[width=0.32\linewidth]{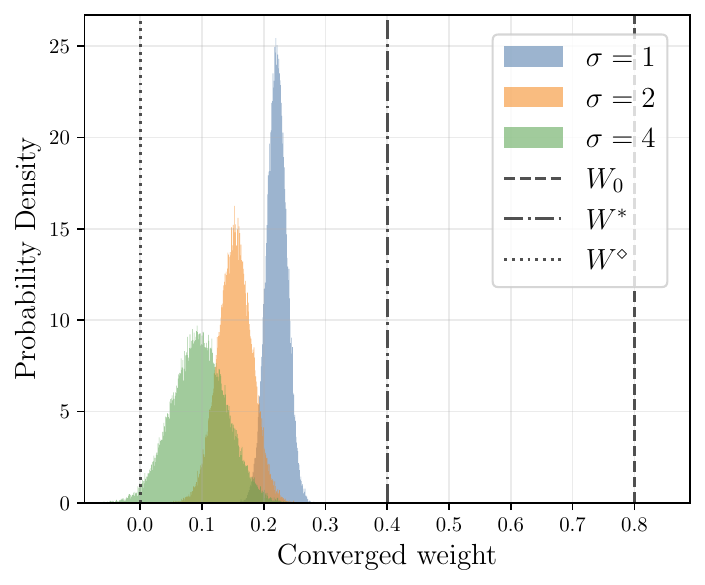}
    \includegraphics[width=0.32\linewidth]{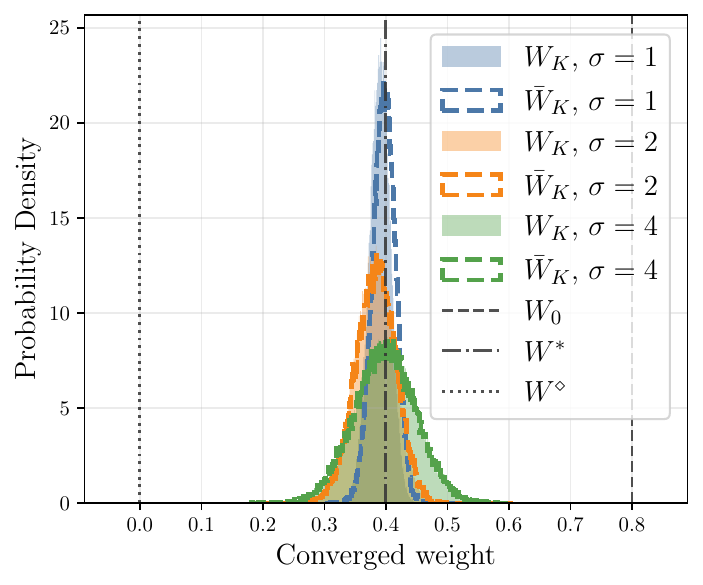}
    \includegraphics[width=0.32\linewidth]{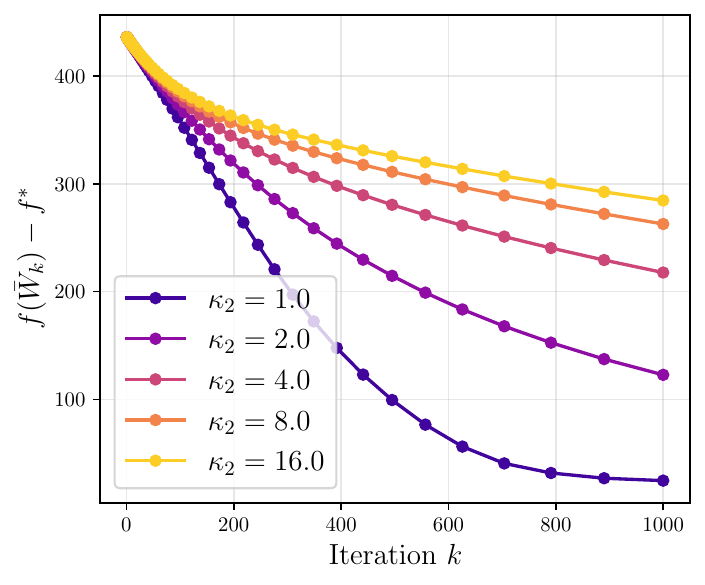}
    \caption{Numerical illustrations of the implicit penalty and the effect of the hardware condition number. \textbf{(Left)} {\AnalogSGD} and \textbf{(Middle)} {\ResidualLearning}: Probability density of iterates after convergence in a 1D quadratic toy example with power response under different noise variances. 
    \textbf{(Right)}: Error curves of {\ResidualLearning} under different hardware condition numbers $\kappa_2$.}
    \label{figure:toy-and-kappa}
\end{figure*}

\begin{figure*}[t]
    \footnotesize
    \centering
    \includegraphics[width=.49\linewidth]{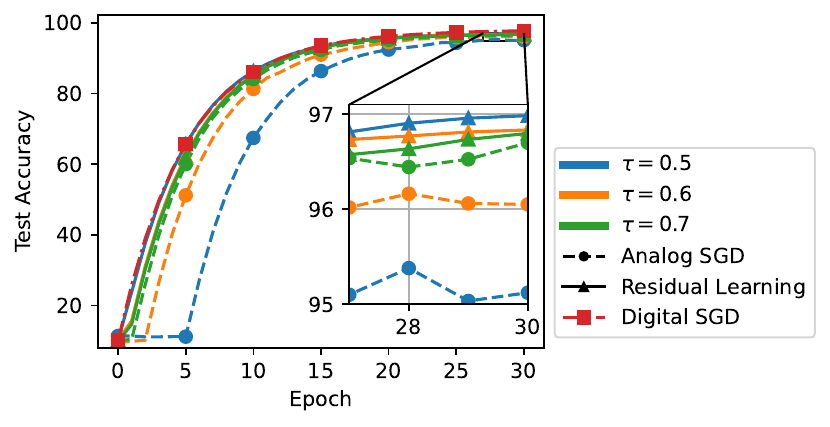}
    \includegraphics[width=0.49\linewidth]{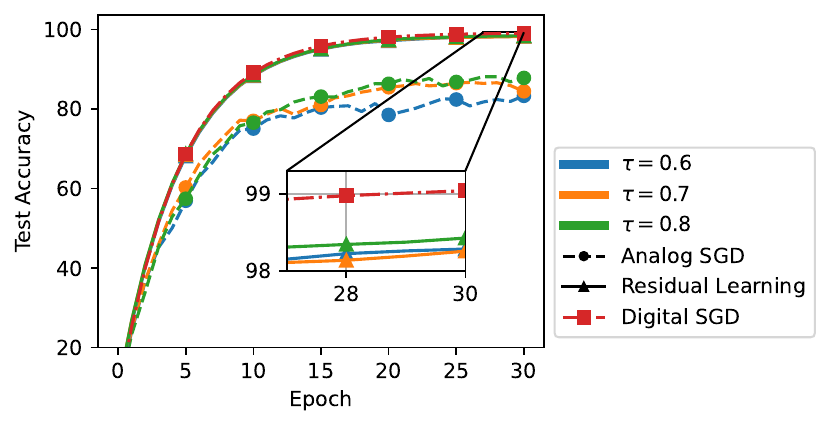}
        \vspace{-1em}
    \caption{Test accuracy curves for the models with the \textit{power response} on the MNIST dataset under different $\tau$; \textbf{(Left)} FCN. \textbf{(Right)} CNN. 
    }
    \label{figure:FCN-CNN-MNIST}
\end{figure*}

\section{Numerical Simulations}
\label{section:experiments}

In this section, we verify the theoretical results via simulations on both synthetic datasets and real datasets.
We use the open source toolkit \ac{AIHWKIT}~\citep{Rasch2021AFA} to simulate \AnalogSGD~and {\ResidualLearning}.
Each simulation is repeated three times, and the mean and standard deviation of the results are reported. 
The code of our simulation is available at \url{github.com/Zhaoxian-Wu/analog-training}.
We consider two types of response functions in our simulations: power and exponential responses.
The \emph{power response} is a power function, given by
\begin{align}
  q_{+}(w) = \lp 1 - \frac{w}{\tau}\rp^{\gamma_{\texttt{res}}}
    , 
    \quad\quad
  q_{-}(w) = \lp 1 + \frac{w}{\tau}\rp^{\gamma_{\texttt{res}}}
\end{align}
whose curves are determined by the dynamics radius $\tau$ and shape parameter $\gamma_{\texttt{res}}$.
The larger $\tau$ and $\gamma_{\texttt{res}}$ are, the higher the asymmetry degree is.
We also consider the \emph{exponential response}, whose response is an exponential function with the symmetric point 0, that is
\begin{align}
  q_{+}(w) = \frac{\exp\lp \gamma_{\texttt{res}}(1 - {w}/{\tau})\rp-1}{\exp\lp \gamma_{\texttt{res}}\rp-1}
    , 
    \quad\quad
  q_{-}(w) = \frac{\exp\lp \gamma_{\texttt{res}}(1 + {w}/{\tau})\rp-1}{\exp\lp \gamma_{\texttt{res}}\rp-1}.
\end{align}

\begin{table*}[b]
    \centering
    \begin{tabular}{c|c|c|c|c|c|c}
        \toprule
         & \multicolumn{3}{c|}{CIFAR10} & \multicolumn{3}{c}{CIFAR100} \\ 
         \cmidrule{2-4}\cmidrule{5-7}
         & \texttt{DSGD} & \texttt{ASGD} & \texttt{RL} 
         & \texttt{DSGD} & \texttt{ASGD} & \texttt{RL} \\ 
        \midrule
        ResNet18  & 95.43\stdv{$\pm$0.13} & 84.47\stdv{$\pm$3.40} & 94.81\stdv{$\pm$0.09} & 81.12\stdv{$\pm$0.25} & 68.98\stdv{$\pm$1.01} & 76.17\stdv{$\pm$0.23} \\
        ResNet34  & 96.48\stdv{$\pm$0.02} & 95.43\stdv{$\pm$0.12} & 96.29\stdv{$\pm$0.12} & 83.86\stdv{$\pm$0.12} & 78.98\stdv{$\pm$0.55} & 80.58\stdv{$\pm$0.11} \\
        ResNet50  & 96.57\stdv{$\pm$0.10} & 94.36\stdv{$\pm$1.16} & 96.34\stdv{$\pm$0.04} & 83.98\stdv{$\pm$0.11} & 79.88\stdv{$\pm$1.26} & 80.80\stdv{$\pm$0.22} \\
        \bottomrule
    \end{tabular}
    \caption{Test accuracy of fine-tuning ResNet models with the \emph{power response} on CIFAR10/100 dataset. 
    \texttt{DSGD}, \texttt{ASGD}, and \texttt{RL} represent \DigitalSGD, \AnalogSGD, {\ResidualLearning}, respectively.
    }
    \label{table:CIFAR-fine-tune-pow}
    \vspace{-1\baselineskip}
\end{table*}
\begin{table*}[b]
    \centering
    \begin{tabular}{c|c|c|c|c|c|c}
        \toprule
            & \multicolumn{3}{c|}{CIFAR10} & \multicolumn{3}{c}{CIFAR100} \\
            \cmidrule{2-4}\cmidrule{5-7}
            & \texttt{DSGD} & \texttt{ASGD} & \texttt{RL} 
            & \texttt{DSGD} & \texttt{ASGD} & \texttt{RL} \\
        \midrule
            ResNet18  & 95.43\stdv{$\pm$0.13} & 94.66\stdv{$\pm$0.11} & 94.70\stdv{$\pm$0.07} & 81.12\stdv{$\pm$0.25} & 73.55\stdv{$\pm$0.22} & 74.64\stdv{$\pm$0.24} \\
            ResNet34  & 96.48\stdv{$\pm$0.02} & 96.19\stdv{$\pm$0.04} & 96.24\stdv{$\pm$0.08} & 83.86\stdv{$\pm$0.12} & 78.10\stdv{$\pm$0.24} & 79.05\stdv{$\pm$0.21} \\
            ResNet50  & 96.57\stdv{$\pm$0.10} & 96.53\stdv{$\pm$0.06} & 96.40\stdv{$\pm$0.13} & 83.98\stdv{$\pm$0.11} & 81.40\stdv{$\pm$0.36} & 79.75\stdv{$\pm$0.10} \\
        \bottomrule
            \end{tabular}
    \caption{Test accuracy of fine-tuning ResNet models with the \emph{exponential response} on CIFAR10/100. 
    \texttt{DSGD}, \texttt{ASGD}, and \texttt{RL} represent \DigitalSGD, \AnalogSGD, {\ResidualLearning}, respectively.
    }
    \label{table:CIFAR-fine-tune-exp}
\end{table*}

\subsection{Verification of the implicit penalty}
We provide empirical verification of Theorem \ref{theorem:implicit-regularization}, which claims that an implicit penalty is imposed by the persistent state-dependent bias.

\noindent\textbf{1D toy example.}
We consider a 1D quadratic objective with optimal point $W^*$ to illustrate the implicit penalty. We construct stochastic gradients by injecting Gaussian noise with different variances and use a power response with $\gamma_{\texttt{res}}=1$ and $\tau=1$. Starting from the same initial point $W_0$, we plot the probability density of the iterates after $K=6000$ iterations in Fig.~\ref{figure:toy-and-kappa} (Left and Middle). The results confirm that \AnalogSGD~is attracted toward the symmetric point of the device (rather than $W^*$), and the drift amplifies as the noise variance increases, which validates the implicit penalty effect predicted by Theorem~\ref{theorem:implicit-regularization}. In contrast, {\ResidualLearning} concentrates around the optimal point $W^*$, which supports the claim in \ref{theorem:TT-convergence-scvx} that the bilevel construction enables the exact convergence despite the persistent state-dependent bias.

We further examine the weight distributions of models trained with {\DigitalSGD} or {\AnalogSGD} on real datasets. We visualize the weight distributions of models from the ResNet family on CIFAR10/CIFAR100 datasets after training 200 epochs, as shown in Fig. \ref{figure:verification-implicit-bias}.
The figures illustrate that the weights trained by \AnalogSGD~are more concentrated around the symmetric point (set as 0 in the simulations) than those trained by \DigitalSGD, which reflects the impact of the implicit penalty. 

\begin{figure*}[t]
    \footnotesize
    \centering
    \includegraphics[width=.4\linewidth]{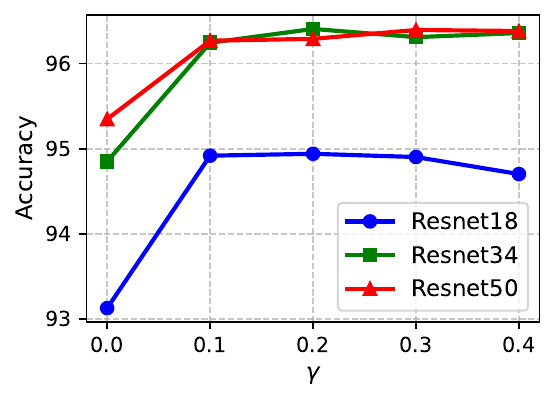}
    \includegraphics[width=0.4\linewidth]{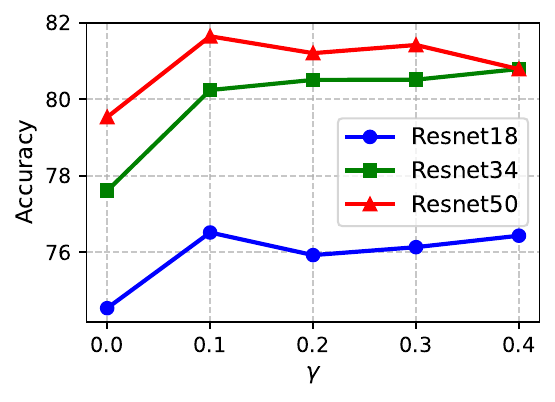}
    \caption{Test accuracy of ResNet models after 100 epochs trained by \ResidualLearning~under different $\gamma$ in \eqref{recursion:HD-P}; \textbf{(Left)} CIFAR10. \textbf{(Right)} CIFAR100. 
    }
    \label{figure:gamma-ablation}
        \vspace{-1em}
\end{figure*}

\subsection{Effect of hardware condition number}
\label{section:experiments-kappa}
We study the effect of the hardware condition number $\kappa_2 := q_{\max}/q_{\min}$ on the convergence speed of {\ResidualLearning}. We consider a $d=1024$-dimensional diagonal quadratic objective $f(W) = \tfrac{1}{2}(W - W^*)^\top H (W - W^*)$, where the components of $\operatorname{diag}(H)$ and $W^*$ are drawn uniformly from $[0.5, 4.0]$ and $[0.5, 1.5]$, respectively. To vary $\kappa_2$ in a controlled manner, we project the weights onto the dynamic range $[-\tau, \tau]$ with $\tau=1$ at each iteration, and parametrize the response functions as $q_+(w) = 1 - \delta \cdot w/\tau$ and $q_-(w) = 1 + \delta \cdot w/\tau$, where $\delta = (\kappa_2 - 1)/(\kappa_2 + 1)$; these functions are bounded within $[1-\delta, 1+\delta]$ and satisfy $q_{\max}/q_{\min} = \frac{1+\delta}{1-\delta} = \kappa_2$ by construction. We sweep $\kappa_2$ over $\{1, 2, 4, 8, 16\}$ with fixed step sizes $\alpha = \beta = 10^{-3}$, and run for $K=1000$ iterations. The convergence curves are presented in Fig.~\ref{figure:toy-and-kappa} (Right). The results demonstrate that increasing $\kappa_2$ consistently degrades the convergence rate, corroborating the polynomial dependence on $\kappa_2$ established in Theorem~\ref{theorem:TT-convergence-scvx}.

\subsection{Training with persistent state-dependent bias on real datasets}
We train vision models to perform image classification tasks on real datasets. 

\noindent
\textbf{MNIST FCN/CNN.} 
We train \acp{FCN} and \acp{CNN} on the MNIST dataset and compare the performance of \texttt{Analog\;SGD} and {\ResidualLearning} under various $\tau$ on power responses; see the results in Fig. \ref{figure:FCN-CNN-MNIST}.
By tracking residual, {\ResidualLearning}~outperforms \AnalogSGD~and reaches comparable accuracy with \DigitalSGD. 
For both architectures, the accuracy of {\ResidualLearning} drops by $<1\%$.
In contrast, \AnalogSGD~takes a few epochs to achieve an observable accuracy increment in FCN training, rendering a slower convergence rate than {\ResidualLearning}. In CNN training, \AnalogSGD's accuracy increases more slowly than {\ResidualLearning} and converges to about 80\%, which is consistent with the theoretical claims.

\begin{figure*}[t]
    \footnotesize
    \centering
    \includegraphics[height=13em]{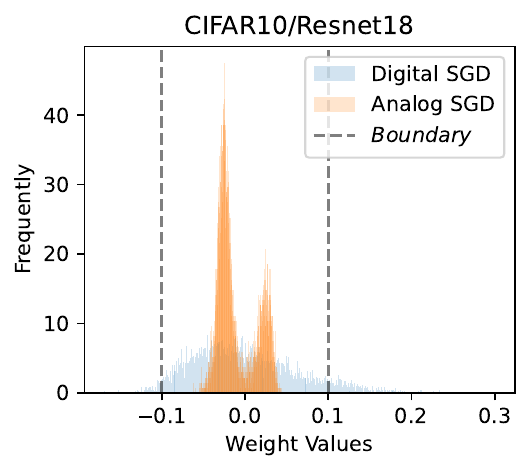}
    \includegraphics[height=13em]{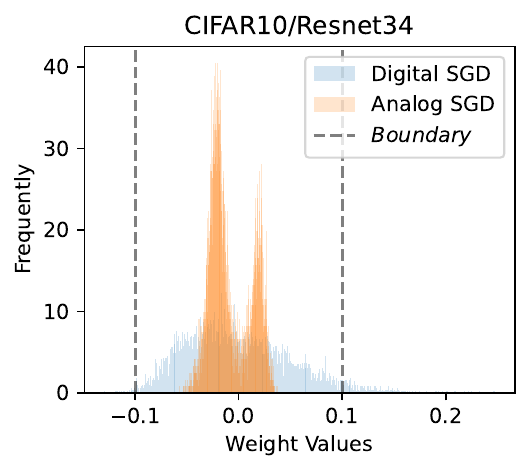}
    \includegraphics[height=13em]{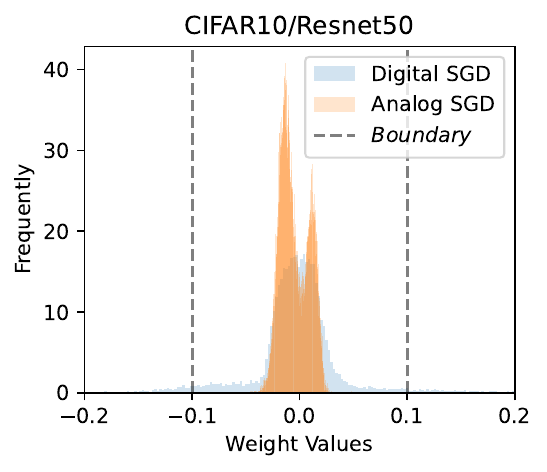}
    \includegraphics[height=13em]{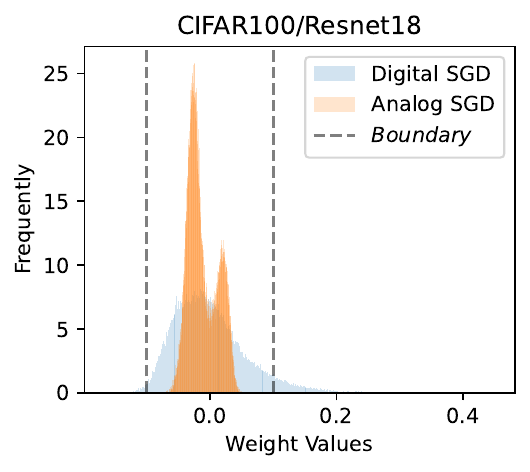}
    \includegraphics[height=13em]{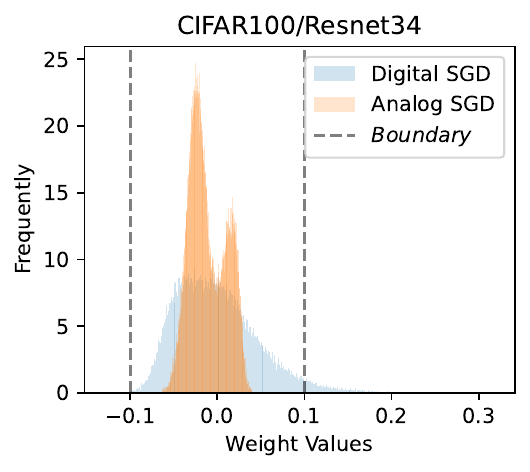}
    \includegraphics[height=13em]{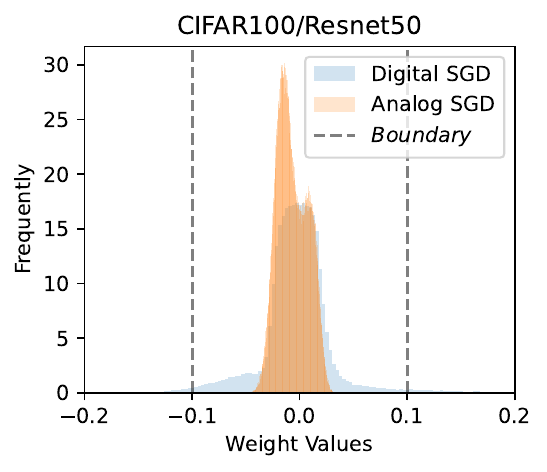}
    \caption{The weight distributions of ResNet models that are trained by \DigitalSGD~or \AnalogSGD~on CIFAR10/CIFAR100 datasets. The layers that are trained on analog hardware are visualized. ``Boundary'' lines mark the boundary of the dynamic range where the response functions are positive.
    }
    \label{figure:verification-implicit-bias}
\end{figure*}

\noindent
\textbf{CIFAR10/CIFAR100 ResNet.} We fine-tune three ResNet models with different scales on CIFAR10/CIFAR100 datasets.
Tables \ref{table:CIFAR-fine-tune-pow} and \ref{table:CIFAR-fine-tune-exp} show the final accuracy of all algorithms on power and exponential response functions, respectively.
Notably, {\ResidualLearning} consistently achieves performance that is comparable to the digital baseline. For example, {\ResidualLearning} improves upon {\AnalogSGD} by over 7\% for ResNet18 and around 2\% for ResNet34 and ResNet50 on CIFAR100, highlighting its ability to bridge the performance gap typically observed in analog training, particularly when contrasted with the less preferred outcomes of \AnalogSGD. These results suggest that the {\ResidualLearning} method effectively mitigates the asymmetric update issue.

\begin{table}[t]
\centering
\begin{tabular}{c|c|c|c|c|c|c}
\toprule
&  & \multirow{2}{*}{$\texttt{Digital SGD}$} 
& \multicolumn{2}{c|}{Power response function} & \multicolumn{2}{c}{Exponential response function} \\ 
\cmidrule{4-7}
&  &  & \texttt{Analog SGD} & \texttt{\ResidualLearning} & \texttt{Analog SGD} & \texttt{\ResidualLearning} \\ 
\midrule
\multirow{3}{*}{$\gamma_{\texttt{res}}=0.5$} 
  & $\tau=0.6$ & \multirow{9}{*}{98.17\stdv{$\pm$0.05}} & 96.01\stdv{$\pm$0.26} & 96.92\stdv{$\pm$0.19} & <15.00          & 97.27\stdv{$\pm$0.07} \\
  & $\tau=0.7$ &                                        & 97.40\stdv{$\pm$0.15} & 97.05\stdv{$\pm$0.05} & <15.00          & 97.39\stdv{$\pm$0.15} \\
  & $\tau=0.8$ &                                        & 97.38\stdv{$\pm$0.10} & 96.82\stdv{$\pm$0.17} & 94.00\stdv{$\pm$0.63} & 97.16\stdv{$\pm$0.16} \\ 
\cmidrule{1-2}\cmidrule{4-5}\cmidrule{6-7}
\multirow{3}{*}{$\gamma_{\texttt{res}}=1.0$} 
  & $\tau=0.6$ &                                        & <15.00          & 97.39\stdv{$\pm$0.05} & <15.00          & 97.46\stdv{$\pm$0.08} \\
  & $\tau=0.7$ &                                        & <15.00          & 97.33\stdv{$\pm$0.05} & <15.00          & 97.49\stdv{$\pm$0.04} \\
  & $\tau=0.8$ &                                        & <15.00          & 97.34\stdv{$\pm$0.09} & <15.00          & 97.25\stdv{$\pm$0.16} \\ 
\cmidrule{1-2}\cmidrule{4-5}\cmidrule{6-7}
\multirow{3}{*}{$\gamma_{\texttt{res}}=2.0$} 
  & $\tau=0.6$ &                                        & <15.00          & 96.93\stdv{$\pm$0.15} & <15.00          & 97.19\stdv{$\pm$0.16} \\
  & $\tau=0.7$ &                                        & <15.00          & 97.27\stdv{$\pm$0.02} & <15.00          & 97.72\stdv{$\pm$0.07} \\
  & $\tau=0.8$ &                                        & <15.00          & 97.18\stdv{$\pm$0.04} & <15.00          & 97.06\stdv{$\pm$0.10} \\ 
\bottomrule
\end{tabular}
\caption{Test accuracy comparison under different response function parameters $\tau$ and $\gamma_{\texttt{res}}$ for FCN training on the MNIST dataset with power or exponential response functions. 
}
\label{table:MNIST-pow-and-exp}
         \vspace{-1em}
\end{table}

\subsection{Ablation study on $\gamma$}
\label{section:ablation-gamma}
We conduct a series of simulations to study the impact of the mixing coefficient $\gamma$ in \eqref{recursion:HD-P} on the CIFAR10 or CIFAR100 dataset in the ResNet training tasks. The results are presented in Fig. \ref{figure:gamma-ablation}, which shows that {\ResidualLearning} achieves a great accuracy gain from increasing $\gamma$ from 0 to 0.1, while the gain saturates after that. Therefore, we conclude that {\ResidualLearning} benefits from a non-zero $\gamma$, and the performance is robust to the $\gamma$ selection.

\section{Conclusions, Limitations, and Future Work}
This paper studies the convergence of gradient-based optimization under persistent state-dependent
bias, motivated by analog in-memory computing. We prove that {\AnalogSGD} does not 
converge to the true minimizer of the original objective; instead, it converges to the 
minimizer of an implicitly penalized problem, and we characterize the resulting bias as a 
function of the hardware response model. To overcome this failure, we reformulate the 
training task as an equivalent bilevel optimization problem and propose {\ResidualLearning}, 
which provably achieves exact convergence to the solution of the original problem despite 
hardware imperfections. Beyond exact convergence, we establish that the convergence 
complexity of {\ResidualLearning} depends polynomially on the hardware condition number 
$\kappa_2$, and we show via a hard-instance construction that this polynomial dependence is 
unavoidable in general. The theoretical results are corroborated by numerical simulations. 
One limitation of this work is the gap between the upper bound $\tilde{\ccalO}(\kappa_2^4)$ and the lower bound $\tilde{\Omega}(\kappa_2^2)$ for {\ResidualLearning}; it remains an open question whether a fundamental limitation exists for general stochastic algorithms on this class of hardware.
Determining the sharp exponent for $\kappa_2$ and developing an optimal algorithm with respect to $\kappa_2$ remain open problems and constitute directions for future work.

\section*{Acknowledgment}
The authors thank Omobayode Fagbohungbe from IBM Research for his insightful discussions during the writing of this paper.

\section*{Declarations}

\noindent\textbf{Funding.}
The work of Z. Wu, Q. Xiao, and T. Chen was supported by the National Science Foundation (NSF) under Grants 2401297 and 2532653 and by a Cisco Research Award.

\noindent\textbf{Competing interests.}
The authors declare that they have no competing interests.

\noindent\textbf{Data and code availability.}
The datasets used in this study are publicly available. The code used for the numerical simulations is available at \url{https://github.com/Zhaoxian-Wu/analog-training}.
\bibliography{
    abrv,
    bilevel,
    byzantine,
    SA,
    bib,
    communication,
    optimization,
    DL,
    analog,
    LLM,
    textbook,
    math,
    publication
}

@article{robbins1951stochastic,
    title={A Stochastic Approximation Method},
    author={Robbins, Herbert and Monro, Sutton},
    journal=AnnMathStat,
    pages={400--407},
    year={1951},
    publisher={JSTOR}
}

@string{ICML    = "International Conference on Machine Learning"}

@string{NeurIPS = "Advances in Neural Information Processing Systems"}

@string{NIPS    = NeurIPS}

@string{ICLR    = "International Conference on Learning Representations"}

@string{AISTATS = "International Conference on Artificial Intelligence and Statistics"}

@string{AAAI    = "AAAI Conference on Artificial Intelligence"}

@string{COLT    = "Conference on Learning Theory"}

@string{AnnMathStat = "The Annals of Mathematical Statistics"}

@string{arxiv = "arXiv preprint arXiv: "}

@article{haensch2019next,
  title={The Next Generation of Deep Learning Hardware: Analog Computing},
  author={Haensch, Wilfried and Gokmen, Tayfun and Puri, Ruchir},
  journal={Proceedings of the IEEE},
  volume={107},
  number={1},
  pages={108--122},
  year={2019},
  publisher={IEEE}
}

@article{Y2020sebastianNatNano,
  title={Memory devices and applications for in-memory computing},
  author={Sebastian, Abu and Le Gallo, Manuel and Khaddam-Aljameh, Riduan and Eleftheriou, Evangelos},
  journal={Nature Nanotechnology},
  volume = 15,
  pages={529-544},
  year=2020,
  publisher={Nature Publishing Group}
}

@article{jain2019neural,
  title={Neural network accelerator design with resistive crossbars: Opportunities and challenges},
  author={Jain, Shubham and others},
  journal={IBM Journal of Research and Development},
  volume=63,
  number=6,
  pages={10--1},
  year=2019,
  publisher={IBM}
}

@ARTICLE{nandakumar2020,
AUTHOR={Nandakumar, S. R. and Le Gallo, Manuel and Piveteau, Christophe and Joshi, Vinay and Mariani, Giovanni and Boybat, Irem and Karunaratne, Geethan and Khaddam-Aljameh, Riduan and Egger, Urs and Petropoulos, Anastasios and Antonakopoulos, Theodore and Rajendran, Bipin and Sebastian, Abu and Eleftheriou, Evangelos},
TITLE={Mixed-Precision Deep Learning Based on Computational Memory},
JOURNAL={Frontiers in Neuroscience},
VOLUME=14,
YEAR=2020,
}

@article{Rasch2021AFA,
    title={A Flexible and Fast {PyTorch} Toolkit for Simulating Training and Inference on Analog Crossbar Arrays},
    author={Rasch, Malte J and Moreda, Diego and Gokmen, Tayfun and Le Gallo, Manuel and Carta, Fabio and Goldberg, Cindy and El Maghraoui, Kaoutar and Sebastian, Abu and Narayanan, Vijay},
    journal={IEEE International Conference on Artificial Intelligence Circuits and Systems},
    year={2021},
    pages={1-4}
}

@article{kim2019zero,
  title={Zero-shifting technique for deep neural network training on resistive cross-point arrays},
  author={Kim, Hyungjun and Rasch, Malte J and Gokmen, Tayfun and Ando, Takashi and Miyazoe, Hiroyuki and Kim, Jae-Joon and Rozen, John and Kim, Seyoung},
  journal={arXiv preprint arXiv:1907.10228},
  year={2019}
}

@article{Burr2016,
  author = {Burr, Geoffrey W and BrightSky, Matthew J and Sebastian, Abu and Cheng, Huai-Yu and Wu, Jau-Yi and Kim, Sangbum and Sosa, Norma E and Papandreou, Nikolaos and Lung, Hsiang-Lan and Pozidis, Haralampos and Eleftheriou, Evangelos and Lam, Chung H},
  issn = {2156-3357},
  journal = {IEEE Journal on Emerging and Selected Topics in Circuits and Systems},
  number = {2},
  pages = {146--162},
  title = {{Recent Progress in Phase-Change Memory Technology}},
  volume = {6},
  year = {2016}
}

@article{2020legalloJPD,
  title={An overview of phase-change memory device physics},
  author={Le Gallo, Manuel and Sebastian, Abu},
  journal={Journal of Physics D: Applied Physics},
  volume={53},
  number={21},
  pages={213002},
  year={2020},
  publisher={IOP Publishing}
}

@article{Jang2015,
    author = {Jang, Jun-Woo and Park, Sangsu and Burr, Geoffrey W and Hwang, Hyunsang and Jeong, Yoon-Ha},
    journal = {IEEE Electron Device Letters},
    number = {5},
    pages = {457--459},
    title = {Optimization of Conductance Change in {Pr$_{1-x}$Ca$_x$MnO$_3$}-Based Synaptic Devices for Neuromorphic Systems},
    volume = {36},
    year = {2015}
}

@inproceedings{Jang2014,
    author = {Jang, Jun-Woo and Park, Sangsu and Jeong, Yoon-Ha and Hwang, Hyunsang},
    booktitle = {IEEE International Symposium on Circuits and Systems},
    pages = {1054--1057},
    title = {{ReRAM}-based Synaptic Device for Neuromorphic Computing},
    year = {2014}
}

@article{onen2022nanosecond,
  title={Nanosecond protonic programmable resistors for analog deep learning},
  author={Onen, Murat and Emond, Nicolas and Wang, Baoming and Zhang, Difei and Ross, Frances M and Li, Ju and Yildiz, Bilge and Del Alamo, Jes{\'u}s A},
  journal={Science},
  volume={377},
  number={6605},
  pages={539--543},
  year={2022},
  publisher={American Association for the Advancement of Science}
}

@inproceedings{chen2015mitigating,
  title={Mitigating effects of non-ideal synaptic device characteristics for on-chip learning},
  author={Chen, Paiyu and Lin, Binbin and Wang, I-Ting and Hou, Tuohung and Ye, Jieping and Vrudhula, Sarma and Seo, Jae-sun and Cao, Yu and Yu, Shimeng},
  booktitle={IEEE/ACM International Conference on Computer-Aided Design},
  pages={194--199},
  year={2015},
  organization={IEEE}
}

@inproceedings{tang2018ecram,
  title={{ECRAM} as scalable synaptic cell for high-speed, low-power neuromorphic computing},
  author={Tang, Jianshi and Bishop, Douglas and Kim, Seyoung and Copel, Matt and Gokmen, Tayfun and Todorov, Teodor and Shin, SangHoon and Lee, Ko-Tao and Solomon, Paul and Chan, Kevin and others},
  booktitle={IEEE International Electron Devices Meeting},
  pages={13--1},
  year={2018},
  organization={IEEE}
}

@article{burr2015experimental,
  title={Experimental demonstration and tolerancing of a large-scale neural network (165 000 synapses) using phase-change memory as the synaptic weight element},
  author={Burr, Geoffrey W and Shelby, Robert M and Sidler, Severin and Di Nolfo, Carmelo and Jang, Junwoo and Boybat, Irem and Shenoy, Rohit S and Narayanan, Pritish and Virwani, Kumar and Giacometti, Emanuele U and others},
  journal={IEEE Transactions on Electron Devices},
  volume={62},
  number={11},
  pages={3498--3507},
  year={2015},
  publisher={IEEE}
}

@inproceedings{jouppi2023tpu,
  title={{TPU} v4: An optically reconfigurable supercomputer for machine learning with hardware support for embeddings},
  author={Jouppi, Norm and Kurian, George and Li, Sheng and Ma, Peter and Nagarajan, Rahul and Nai, Lifeng and Patil, Nishant and Subramanian, Suvinay and Swing, Andy and Towles, Brian and others},
  booktitle={Annual International Symposium on Computer Architecture},
  pages={1--14},
  year={2023}
}

@inproceedings{cosemans2019towards,
  author={Cosemans, Stefan and Verhoef, Bram-Ernst and Doevenspeck, Jonas and Papistas, Ioannis A. and Catthoor, Francky and Debacker, Peter and Mallik, Arindam and Verkest, Diederik},
  booktitle={IEEE International Electron Devices Meeting}, 
  title={Towards 10000{TOPS/W} {DNN} Inference with Analog in-Memory Computing – A Circuit Blueprint, Device Options and Requirements}, 
  year={2019},
  pages={22.2.1-22.2.4}
}

@inproceedings{papistas202122,
    title={A 22 nm, 1540 {TOP}/s/{W}, 12.1 {TOP}/s/mm 2 in-memory analog matrix-vector-multiplier for {DNN} acceleration},
    author={Papistas, Ioannis A and Cosemans, Stefan and Rooseleer, Bram and Doevenspeck, Jonas and Na, M-H and Mallik, Arindam and Debacker, Peter and Verkest, Diederik},
    booktitle={IEEE Custom Integrated Circuits Conference},
    pages={1--2},
    year={2021},
    organization={IEEE}
}

@article{stecconi2024analog,
  title={Analog Resistive Switching Devices for Training Deep Neural Networks with the Novel {Tiki-Taka} Algorithm},
  author={Stecconi, Tommaso and Bragaglia, Valeria and Rasch, Malte J and Carta, Fabio and Horst, Folkert and Falcone, Donato F and Ten Kate, Sofieke C and Gong, Nanbo and Ando, Takashi and Olziersky, Antonis and others},
  journal={Nano Letters},
  volume={24},
  number={3},
  pages={866--872},
  year={2024},
  publisher={ACS Publications}
}

@article{scellier2017equilibrium,
  title={Equilibrium propagation: Bridging the gap between energy-based models and backpropagation},
  author={Scellier, Benjamin and Bengio, Yoshua},
  journal={Frontiers in computational neuroscience},
  volume={11},
  pages={24},
  year={2017},
  publisher={Frontiers Media SA}
}

@article{wang2020ssm,
  title={{SSM}: a high-performance scheme for in situ training of imprecise memristor neural networks},
  author={Wang, Yaoyuan and Wu, Shuang and Tian, Lei and Shi, Luping},
  journal={Neurocomputing},
  volume={407},
  pages={270--280},
  year={2020},
  publisher={Elsevier}
}

@inproceedings{huang2020overcoming,
  title={Overcoming challenges for achieving high in-situ training accuracy with emerging memories},
  author={Huang, Shanshi and Sun, Xiaoyu and Peng, Xiaochen and Jiang, Hongwu and Yu, Shimeng},
  booktitle={Design, Automation \& Test in Europe Conference \& Exhibition},
  pages={1025--1030},
  year={2020},
  organization={IEEE}
}

@string{ICLR = "Proc. of International Conference on Learning Representations"}

@string{AISTATS = "Proc. of International Conference on Artificial Intelligence and Statistics"}

@string{ICML = "Proc. of International Conference on Machine Learning"}

@string{AAAI = "Proc. of AAAI Conference on Artificial Intelligence"}

@string{NeurIPS = "Proc. Advances in Neural Info. Process. Syst."}

@string{COLT = "Proc. of Conference on Learning Theory"}

@string{arxiv = "arXiv preprint: "}

@string{NIPS2021_loc= "virtual"}

@string{ICML2016_loc= "New York City, NY"}

@string{ICML2017_loc= "Sydney, Australia"}

@string{ICML2018_loc= "Stockholm, Sweden"}

@string{ICML2020_loc= "virtual"}

@string{ICML2021_loc= "virtual"}

@string{AISTATS2019_loc = "Naha, Japan"}

@inproceedings{chen2022single,
  title={A single-timescale method for stochastic bilevel optimization},
  author={Chen, Tianyi and Sun, Yuejiao and Xiao, Quan and Yin, Wotao},
  booktitle=AISTATS,
  year={2022},
  address=AISTATS2022_loc
}

@inproceedings{franceschi2017forward,
  title={Forward and reverse gradient-based hyperparameter optimization},
  author={Franceschi, Luca and Donini, Michele and Frasconi, Paolo and Pontil, Massimiliano},
  booktitle=ICML,
  year={2017},
  address=ICML2017_loc
}

@inproceedings{franceschi2018bilevel,
  title = {Bilevel Programming for Hyperparameter Optimization and Meta-Learning},
  author = {Franceschi, Luca and Frasconi, Paolo and Salzo, Saverio and Grazzi, Riccardo and Pontil, Massimilano},
  booktitle=ICML,
  year = {2018},
  address = ICML2018_loc
}

@article{bracken1973mathematical,
  title={Mathematical programs with optimization problems in the constraints},
  author={Bracken, Jerome and McGill, James T},
  journal={Operations Research},
  volume={21},
  number={1},
  pages={37--44},
  year={1973}
}

@article{ye1995optimality,
  title={Optimality conditions for bilevel programming problems},
  author={Ye, Jane J and Zhu, Daoli},
  journal={Optimization},
  volume={33},
  number={1},
  pages={9--27},
  year={1995}
}

@article{vicente1994bilevel,
  title={Bilevel and multilevel programming: A bibliography review},
  author={Vicente, Luis N and Calamai, Paul H},
  journal={Journal of Global optimization},
  volume={5},
  number={3},
  pages={291--306},
  year={1994}
}

@article{colson2007overview,
  title={An overview of bilevel optimization},
  author={Colson, Beno{\^\i}t and Marcotte, Patrice and Savard, Gilles},
  journal={Annals of operations research},
  volume={153},
  number={1},
  pages={235--256},
  year={2007}
}

@inproceedings{grazzi2020iteration,
  title={On the iteration complexity of hypergradient computation},
  author={Grazzi, Riccardo and Franceschi, Luca and Pontil, Massimiliano and Salzo, Saverio},
  booktitle=ICML,
  year={2020},
  address=ICML2020_loc
}

@inproceedings{yang2021provably,
  title={Provably faster algorithms for bilevel optimization},
  author={Yang, Junjie and Ji, Kaiyi and Liang, Yingbin},
  booktitle=NIPS,
  year={2021},
  address=NIPS2021_loc
}

@inproceedings{arbel2021amortized,
  title={Amortized Implicit Differentiation for Stochastic Bilevel Optimization},
  author={Arbel, Michael and Mairal, Julien},
  booktitle=ICLR,
  year={2022},
  address=ICLR2022_loc
}

@inproceedings{arbel2022nonconvex,
  title={Non-Convex Bilevel Games with Critical Point Selection Maps},
  author={Arbel, Michael and Mairal, Julien},
  booktitle=NIPS,
  year={2022},
  address=NIPS2022_loc
}

@inproceedings{liubome,
  title={BOME! Bilevel Optimization Made Easy: A Simple First-Order Approach},
  author={Liu, Bo and Ye, Mao and Wright, Stephen and Stone, Peter and others},
  booktitle=NIPS,
  year = 2022,
  address = NIPS2022_loc
}

@article{shen2023penalty,
  title={On Penalty-based Bilevel Gradient Descent Method},
  author={Shen, Han and Xiao, Quan and Chen, Tianyi},
  journal={Mathematical Programming},
  pages={1--51},
  year={2025}
}

@inproceedings{liu2023averaged,
  title={Averaged Method of Multipliers for Bi-Level Optimization without Lower-Level Strong Convexity},
  author={Liu, Risheng and Liu, Yaohua and Yao, Wei and Zeng, Shangzhi and Zhang, Jin},
  booktitle=ICML,
  year={2023},
  address=ICML2023_loc
}

@inproceedings{pedregosa2016hyperparameter,
  title={Hyperparameter optimization with approximate gradient},
  author={Pedregosa, Fabian},
  booktitle=ICML,
  year={2016},
  address=ICML2016_loc
}

@inproceedings{chen2021closing,
  title={Closing the Gap: Tighter Analysis of Alternating Stochastic Gradient Methods for Bilevel Problems},
  author={Chen, Tianyi and Sun, Yuejiao and Yin, Wotao},
  booktitle=NIPS,
  year={2021},
  address=NIPS2021_loc
}

@article{hong2020two,
  title={A two-timescale stochastic algorithm framework for bilevel optimization: Complexity analysis and application to actor-critic},
  author={Hong, Mingyi and Wai, Hoi-To and Wang, Zhaoran and Yang, Zhuoran},
  journal={SIAM Journal on Optimization},
  volume={33},
  number={1},
  pages={147--180},
  year={2023},
  publisher={SIAM}
}

@article{ghadimi2018approximation,
  title={Approximation methods for bilevel programming},
  author={Ghadimi, Saeed and Wang, Mengdi},
  journal={arXiv preprint arXiv:1802.02246},
  year={2018}
}

@inproceedings{ji2021bilevel,
  title={Bilevel optimization: Convergence analysis and enhanced design},
  author={Ji, Kaiyi and Yang, Junjie and Liang, Yingbin},
  booktitle=ICML,
  year={2021},
  address=ICML2021_loc
}

@inproceedings{li2022fully,
  title={A fully single loop algorithm for bilevel optimization without hessian inverse},
  author={Li, Junyi and Gu, Bin and Huang, Heng},
  booktitle=AAAI,
  year={2022},
  address=AAAI2022_loc
}

@inproceedings{khanduri2021near,
  title={A near-optimal algorithm for stochastic bilevel optimization via double-momentum},
  author={Khanduri, Prashant and Zeng, Siliang and Hong, Mingyi and Wai, Hoi-To and Wang, Zhaoran and Yang, Zhuoran},
  booktitle=NIPS,
  year={2021},
  address=NIPS2021_loc
}

@inproceedings{mehra2021penalty,
  title={Penalty method for inversion-free deep bilevel optimization},
  author={Mehra, Akshay and Hamm, Jihun},
  booktitle={Asian Conference on Machine Learning},
  year={2021},
  address={virtual}
}

@inproceedings{xiao2023generalized,
  title={A Generalized Alternating Method for Bilevel Optimization under the Polyak-{\L}ojasiewicz Condition},
  author={Xiao, Quan and Lu, Songtao and Chen, Tianyi},
  booktitle=NIPS,
  year={2023},
  address=NIPS2023_loc
}

@inproceedings{kwon2023fully,
  title={A fully first-order method for stochastic bilevel optimization},
  author={Kwon, Jeongyeol and Kwon, Dohyun and Wright, Stephen and Nowak, Robert D},
  booktitle=ICML,
  year={2023},
  address=ICML2023_loc
}

@article{lu2023first,
  title={First-order penalty methods for bilevel optimization},
  author={Lu, Zhaosong and Mei, Sanyou},
  journal={arXiv preprint arXiv:2301.01716},
  year={2023}
}

@inproceedings{kwon2023penalty,
  title={On Penalty Methods for Nonconvex Bilevel Optimization and First-Order Stochastic Approximation},
  author={Kwon, Jeongyeol and Kwon, Dohyun and Wright, Steve and Nowak, Robert},
  booktitle=ICLR,
  year={2024},
  address=ICLR2024_loc
}

@inproceedings{shaban2019truncated,
  title={Truncated back-propagation for bilevel optimization},
  author={Shaban, Amirreza and Cheng, Ching-An and Hatch, Nathan and Boots, Byron},
  booktitle=AISTATS,
  year={2019},
  address=AISTATS2019_loc
}

@inproceedings{wang2021fast,
  title={Fast algorithms for stackelberg prediction game with least squares loss},
  author={Wang, Jiali and Chen, He and Jiang, Rujun and Li, Xudong and Li, Zihao},
  booktitle=ICML,
  year={2021},
  address=ICML2021_loc
}

@inproceedings{wang2022solving,
  title={Solving Stackelberg Prediction Game with Least Squares Loss via Spherically Constrained Least Squares Reformulation},
  author={Wang, Jiali and Huang, Wen and Jiang, Rujun and Li, Xudong and Wang, Alex L},
  booktitle=ICML,
  year={2022},
  address=ICML2022_loc
}

@article{jeyakumar2016convergent,
  title={Convergent semidefinite programming relaxations for global bilevel polynomial optimization problems},
  author={Jeyakumar, Vaithilingam and Lasserre, Jean B and Li, Guoyin and Pham, TS},
  journal={SIAM Journal on Optimization},
  volume={26},
  number={1},
  pages={753--780},
  year={2016}
}

@article{vicente1994descent,
  title={Descent approaches for quadratic bilevel programming},
  author={Vicente, Luis and Savard, Gilles and J{\'u}dice, Joaquim},
  journal={Journal of Optimization theory and applications},
  volume={81},
  number={2},
  pages={379--399},
  year={1994}
}

@article{xiao2024unlocking,
  title={Unlocking Global Optimality in Bilevel Optimization: A Pilot Study},
  author={Xiao, Quan and Chen, Tianyi},
  journal={arXiv preprint arXiv:2408.16087},
  year={2024}
}

@article{zucchet2022beyond,
  title={Beyond backpropagation: bilevel optimization through implicit differentiation and equilibrium propagation},
  author={Zucchet, Nicolas and Sacramento, Jo{\~a}o},
  journal={Neural Computation},
  volume={34},
  number={12},
  pages={2309--2346},
  year={2022}
}

@article{scellier2021deep,
  title={A deep learning theory for neural networks grounded in physics},
  author={Scellier, Benjamin},
  journal={arXiv preprint arXiv:2103.09985},
  year={2021}
}

@inproceedings{jiang2024primal,
  title={A primal-dual-assisted penalty approach to bilevel optimization with coupled constraints},
  author={Jiang, Liuyuan and Xiao, Quan and Tenorio, Victor M and Real-Rojas, Fernando and Marques, Antonio G and Chen, Tianyi},
  booktitle=NIPS,
  year={2024},
  address=NIPS2024_loc
}

@article{yao2024overcoming,
  title={Overcoming Lower-Level Constraints in Bilevel Optimization: A Novel Approach with Regularized Gap Functions},
  author={Yao, Wei and Yin, Haian and Zeng, Shangzhi and Zhang, Jin},
  journal={arXiv preprint arXiv:2406.01992},
  year={2024}
}

@inproceedings{chen2024finding,
  title={On finding small hyper-gradients in bilevel optimization: Hardness results and improved analysis},
  author={Chen, Lesi and Xu, Jing and Zhang, Jingzhao},
  booktitle=COLT,
  pages={947--980},
  year={2024},
  organization={PMLR}
}

@inproceedings{karimireddybyzantine,
    title={Byzantine-Robust Learning on Heterogeneous Datasets via Bucketing},
    author={Karimireddy, Sai Praneeth and He, Lie and Jaggi, Martin},
    booktitle={International Conference on Learning Representations},
    year={2022}
}

@inproceedings{karimireddy2019error,
  title={Error feedback fixes signsgd and other gradient compression schemes},
  author={Karimireddy, Sai Praneeth and Rebjock, Quentin and Stich, Sebastian and Jaggi, Martin},
  booktitle={International Conference on Machine Learning},
  pages={3252--3261},
  year={2019},
  organization={PMLR}
}

@article{magnusson2020maintaining,
  title={On maintaining linear convergence of distributed learning and optimization under limited communication},
  author={Magn{\'u}sson, Sindri and Shokri-Ghadikolaei, Hossein and Li, Na},
  journal={IEEE Transactions on Signal Processing},
  volume={68},
  pages={6101--6116},
  year={2020},
  publisher={IEEE}
}

@article{stich2018sparsified,
  title={Sparsified {SGD} with memory},
  author={Stich, Sebastian U and Cordonnier, Jean-Baptiste and Jaggi, Martin},
  journal={NeurIPS},
  year={2018}
}

@article{wangni2018gradient,
  title={Gradient sparsification for communication-efficient distributed optimization},
  author={Wangni, Jianqiao and Wang, Jialei and Liu, Ji and Zhang, Tong},
  journal={NeurIPS},
  year={2018}
}

@article{safaryan2022uncertainty,
  title={Uncertainty principle for communication compression in distributed and federated learning and the search for an optimal compressor},
  author={Safaryan, Mher and Shulgin, Egor and Richt{\'a}rik, Peter},
  journal={Information and Inference: A Journal of the IMA},
  volume={11},
  number={2},
  pages={557--580},
  year={2022},
  publisher={Oxford University Press}
}

@article{alistarh2017qsgd,
  title={{QSGD}: Communication-efficient {SGD} via gradient quantization and encoding},
  author={Alistarh, Dan and Grubic, Demjan and Li, Jerry and Tomioka, Ryota and Vojnovic, Milan},
  journal={NeurIPS},
  year={2017}
}

@article{bottou2018optimization,
    title={Optimization methods for large-scale machine learning},
    author={Bottou, L{\'e}on and Curtis, Frank E and Nocedal, Jorge},
    journal={SIAM review},
    volume={60},
    number={2},
    pages={223--311},
    year={2018},
    publisher={SIAM}
}

@inproceedings{hazan2014beyond,
  title={Beyond the regret minimization barrier: Optimal algorithms for stochastic strongly-convex optimization},
  author={Hazan, Elad and Kale, Satyen},
  booktitle={Proceedings of The 27th Conference on Learning Theory},
  pages={421--436},
  year={2014},
  organization={PMLR}
}

@inproceedings{agarwal2009information,
  title={Information-theoretic lower bounds on the oracle complexity of stochastic convex optimization},
  author={Agarwal, Alekh and Bartlett, Peter L. and Ravikumar, Pradeep and Wainwright, Martin J.},
  booktitle={Advances in Neural Information Processing Systems},
  volume={22},
  year={2009}
}

@inproceedings{karimi2016linear,
  title={Linear convergence of gradient and proximal-gradient methods under the {Polyak-Lojasiewicz} condition},
  author={Karimi, Hamed and Nutini, Julie and Schmidt, Mark},
  booktitle={Joint European conference on machine learning and knowledge discovery in databases},
  pages={795--811},
  year={2016},
  organization={Springer}
}

@article{ajalloeian2020convergence,
  title={On the convergence of {SGD} with biased gradients},
  author={Ajalloeian, Ahmad and Stich, Sebastian U},
  journal={arXiv preprint arXiv:2008.00051},
  year={2020}
}

@article{huang2021improved,
  title={An improved analysis and rates for variance reduction under without-replacement sampling orders},
  author={Huang, Xinmeng and Yuan, Kun and Mao, Xianghui and Yin, Wotao},
  journal={Advances in Neural Information Processing Systems},
  volume={34},
  pages={3232--3243},
  year={2021}
}

@inproceedings{haochen2019random,
  title={Random shuffling beats SGD after finite epochs},
  author={Haochen, Jeff and Sra, Suvrit},
  booktitle={International Conference on Machine Learning},
  pages={2624--2633},
  year={2019},
  organization={PMLR}
}

@article{gurbuzbalaban2021random,
  title={Why random reshuffling beats stochastic gradient descent},
  author={Gurbuzbalaban, Mert and Ozdaglar, Asu and Parrilo, Pablo A},
  journal={Mathematical Programming},
  volume={186},
  number={1},
  pages={49--84},
  year={2021},
  publisher={Springer}
}

@article{khanh2024new,
  title={A new inexact gradient descent method with applications to nonsmooth convex optimization},
  author={Khanh, Pham Duy and Mordukhovich, Boris S and Tran, Dat Ba},
  journal={Optimization Methods and Software},
  pages={1--29},
  year={2024},
  publisher={Taylor \& Francis}
}

@inproceedings{nguyen2017sarah,
  title={{SARAH}: A novel method for machine learning problems using stochastic recursive gradient},
  author={Nguyen, Lam M and Liu, Jie and Scheinberg, Katya and Tak{\'a}{\v{c}}, Martin},
  booktitle={International conference on machine learning},
  pages={2613--2621},
  year={2017},
  organization={PMLR}
}

@article{schmidt2017minimizing,
  title={Minimizing finite sums with the stochastic average gradient},
  author={Schmidt, Mark and Le Roux, Nicolas and Bach, Francis},
  journal={Mathematical Programming},
  volume={162},
  number={1},
  pages={83--112},
  year={2017},
  publisher={Springer}
}

@article{cutkosky2019momentum,
  title={Momentum-based variance reduction in non-convex {SGD}},
  author={Cutkosky, Ashok and Orabona, Francesco},
  journal={Advances in neural information processing systems},
  volume={32},
  year={2019}
}

@article{polyak1992acceleration,
  title={Acceleration of stochastic approximation by averaging},
  author={Polyak, Boris T and Juditsky, Anatoli B},
  journal={SIAM journal on control and optimization},
  volume={30},
  number={4},
  pages={838--855},
  year={1992},
  publisher={SIAM}
}

@article{wu2020federated,
  title={Federated variance-reduced stochastic gradient descent with robustness to byzantine attacks},
  author={Wu, Zhaoxian and Ling, Qing and Chen, Tianyi and Giannakis, Georgios B},
  journal={IEEE Transactions on Signal Processing},
  volume={68},
  pages={4583--4596},
  year={2020},
  publisher={IEEE}
}

@inproceedings{wu2024towards,
    title={Towards Exact Gradient-based Training on Analog In-memory Computing},
    author={Wu, Zhaoxian and Gokmen, Tayfun and Rasch, Malte J and Chen, Tianyi},
    booktitle={Advances in Neural Information Processing Systems},
    year={2024}
}

@inproceedings{wu2025analog,
    title={Analog In-memory Training on General Non-ideal Resistive Elements: The Impact of Response Functions},
    author={Wu, Zhaoxian and Xiao, Quan and Gokmen, Tayfun and Fagbohungbe, Omobayode and Chen, Tianyi},
    booktitle={Advances in Neural Information Processing Systems},
    year={2025}
}

@book{Nesterov2013Introductory,
    title={Introductory Lectures on Convex Optimization: A Basic Course},
    author={Nesterov, Yurii},
    year={2013},
    publisher={Springer}
}

\bibliographystyle{spbasic} 

\clearpage
\appendix

\begin{center}
\large \textbf{Appendix for  ``\FullTitle''} \\
\end{center}

\vspace{-0.5cm}

\section{Useful Lemmas and Proofs}

\begin{lemma}[Lipschitz continuity of analog update]
    \label{lemma:lip-analog-update}
    The analog operator in \eqref{analog-update} is Lipschitz continuous in its second argument. 
    That is, for any $Z, \Delta Z_1,\Delta Z_2\in\reals^D$,
    it holds
    \begin{align}
        \label{eq:A-lip}
        \lnorm A(\Delta Z_1; Z)-A(\Delta Z_2; Z)\rnorm
        \le&\ q_{\max}\lnorm \Delta Z_1-\Delta Z_2\rnorm.
    \end{align}
\end{lemma}
\begin{proof}
    By definition of analog operator $A$ \eqref{analog-update}, each component of the operator has the representation
    \begin{align}
        \label{eq:A-component-representation}
        [A(\Delta Z; Z)]_d
        =
        \begin{cases}
            [\Delta Z]_d \, q_+([Z]_d), & [\Delta Z]_d\ge 0,\\
            [\Delta Z]_d \, q_-([Z]_d), & [\Delta Z]_d<0,
        \end{cases}
        \qquad d\in[D].
    \end{align}
    This scalar piecewise-linear map is increasing, and each of its slopes lies in $[q_{\min},q_{\max}]$, which leads to 
    \begin{align}
        \label{eq:A-component-lipschitz}
        |[A(\Delta Z_1; Z)]_d-[A(\Delta Z_2; Z)]_d|
        \le q_{\max}|[\Delta Z_1]_d-[\Delta Z_2]_d|,
        \qquad\forall d\in[D].
    \end{align}
    Squaring and summing \eqref{eq:A-component-lipschitz} gives \eqref{eq:A-lip}.
\end{proof}

\begin{lemma}[Operator coercivity]
    \label{lemma:A-op-noisy}
    The analog operator $A(\cdotc;\cdotc)$ satisfies the following properties.
    \begin{enumerate}[label=(\roman*)]
        \item \emph{Coercivity.} For any $Z,\Delta Z\in\reals^D$,
        \begin{align}
            \label{eq:A-mono}
            \la \Delta Z,A(\Delta Z; Z)\ra
            \ge&\ q_{\min}\lnorm \Delta Z\rnorm^2.
        \end{align}
        \item \emph{Coercivity with deterministic perturbation.} For any
        $Z,\Delta Z,\delta Z\in\reals^D$,
        \begin{align}
            \label{eq:A-perturbed-lower-young}
            \la \Delta Z,A(\Delta Z+\delta Z; Z)\ra
            \ge&\ \frac{q_{\min}}{2}\lnorm \Delta Z\rnorm^2
            +\frac{1}{2q_{\max}}\lnorm A(\Delta Z+\delta Z; Z)\rnorm^2
            -\frac{q_{\max}}{2}\lnorm\delta Z\rnorm^2.
        \end{align}
        The same inner product also admits a lower bound in terms of the realized update:
        \begin{align}
            \label{eq:A-perturbed-lower-realized-young}
            \la \Delta Z,A(\Delta Z+\delta Z; Z)\ra
            \ge&\ \frac{1}{2q_{\max}}\lnorm A(\Delta Z+\delta Z; Z)\rnorm^2
            -\frac{q_{\max}}{2}\lnorm\delta Z\rnorm^2.
        \end{align}
        \item \emph{Coercivity with stochastic perturbation.} If $\delta Z$ is random,
        $\mbE[\delta Z]=0$, and $\mbE[\|\delta Z\|^2]\le\sigma^2$, then
        \begin{align}
            \label{eq:A-noisy-lower-young}
            \mbE[\la \Delta Z,A(\Delta Z+\delta Z; Z)\ra]
            \ge&\ \frac{3q_{\min}}{4}\lnorm \Delta Z\rnorm^2
            -\frac{\sigma^2}{q_{\min}}\|G(Z)\|_\infty^2.
        \end{align}
    \end{enumerate}
\end{lemma}
\begin{proof}
    We prove the three properties in order.
    Applying the lower bound of the response function to \eqref{eq:A-component-representation} gives the strong monotonicity estimate
    \begin{align}
        \label{eq:A-mono-proof}
        \la \Delta Z,A(\Delta Z; Z)\ra
        =&\ \sum_{d\in[D]} [\Delta Z]_d[A(\Delta Z; Z)]_d
        \ge q_{\min}\sum_{d\in[D]}[\Delta Z]_d^2
        =q_{\min}\lnorm\Delta Z\rnorm^2.
    \end{align}
    For the coercivity with deterministic perturbation, we prove a joint coordinatewise bound using the response selected by the sign of the realized argument $[\Delta Z+\delta Z]_d$.
    If $[\Delta Z+\delta Z]_d\ge0$, then the component representation \eqref{eq:A-component-representation} gives
    \begin{align}
        \label{eq:A-perturbed-positive-coordinate}
        [\Delta Z]_d[A(\Delta Z+\delta Z; Z)]_d
        =&\
        q_+([Z]_d)[\Delta Z]_d[\Delta Z+\delta Z]_d
        \\
        =&\
        \frac{q_+([Z]_d)}{2}[\Delta Z]_d^2
        +\frac{1}{2q_+([Z]_d)}[A(\Delta Z+\delta Z; Z)]_d^2
        -\frac{q_+([Z]_d)}{2}[\delta Z]_d^2.
        \nonumber
    \end{align}
    If $[\Delta Z+\delta Z]_d<0$, then the analogous identity is
    \begin{align}
        \label{eq:A-perturbed-negative-coordinate}
        [\Delta Z]_d[A(\Delta Z+\delta Z; Z)]_d
        =&\
        q_-([Z]_d)[\Delta Z]_d[\Delta Z+\delta Z]_d
        \\
        =&\
        \frac{q_-([Z]_d)}{2}[\Delta Z]_d^2
        +\frac{1}{2q_-([Z]_d)}[A(\Delta Z+\delta Z; Z)]_d^2
        -\frac{q_-([Z]_d)}{2}[\delta Z]_d^2.
        \nonumber
    \end{align}
    The response bounds in Definition \ref{assumption:response-factor} imply, in either case,
    \begin{align}
        \label{eq:A-perturbed-coordinate-joint-bound}
        [\Delta Z]_d[A(\Delta Z+\delta Z; Z)]_d
        \ge&\
        \frac{q_{\min}}{2}[\Delta Z]_d^2
        +\frac{1}{2q_{\max}}[A(\Delta Z+\delta Z; Z)]_d^2
        -\frac{q_{\max}}{2}[\delta Z]_d^2
    \end{align}
    Summing \eqref{eq:A-perturbed-coordinate-joint-bound} over $d\in[D]$ gives \eqref{eq:A-perturbed-lower-young}.
    Dropping the nonnegative first term in \eqref{eq:A-perturbed-lower-young} gives \eqref{eq:A-perturbed-lower-realized-young}.

    For the coercivity with stochastic perturbation, leveraging the symmetric decomposition \eqref{biased-update} gives
    \begin{align}
        \label{eq:A-noisy-decomposition}
        &\ \la \Delta Z,A(\Delta Z+\delta Z; Z)\ra
        = \la \Delta Z,(\Delta Z + \delta Z) \odot F(Z) -|\Delta Z + \delta Z| \odot G(Z)\ra\\
        =&\ \la \Delta Z,A(\Delta Z; Z)\ra
        +\la \Delta Z,\delta Z\odot F(Z)\ra
        +\la \Delta Z,(|\Delta Z|-|\Delta Z+\delta Z|)\odot G(Z)\ra.
        \nonumber
    \end{align}
    The last terms in the \ac{RHS} of \eqref{eq:A-noisy-decomposition} can be bounded as
    \begin{align}
        &\ \mbE[\la \Delta Z,(|\Delta Z+\delta Z|-|\Delta Z|)\odot G(Z)\ra]
        \ge \frac{q_{\min}}{4}\lnorm\Delta Z\rnorm^2 - \frac{1}{q_{\min}}\mbE[\lnorm(|\Delta Z+\delta Z|-|\Delta Z|)\odot G(Z)\rnorm^2]
        \nonumber\\
        \ge&\ \frac{q_{\min}}{4}\lnorm\Delta Z\rnorm^2 - \frac{1}{q_{\min}}\mbE[\lnorm|\Delta Z+\delta Z|-|\Delta Z|\rnorm^2]\, \|G(Z)\|_\infty^2
        \ge \frac{q_{\min}}{4}\lnorm\Delta Z\rnorm^2 - \frac{1}{q_{\min}}\mbE[\lnorm|\delta Z\rnorm^2]\, \|G(Z)\|_\infty^2
        \nonumber\\
        \ge&\ \frac{q_{\min}}{4}\lnorm\Delta Z\rnorm^2 - \frac{1}{q_{\min}}\mbE[\lnorm|\delta Z\rnorm^2]\, \frac{1}{q_{\min}}\|G(Z)\|_\infty^2
        \ge \frac{q_{\min}}{4}\lnorm\Delta Z\rnorm^2 - \, \frac{\sigma^2}{q_{\min}}\|G(Z)\|_\infty^2
        \label{eq:A-noisy-perturbation}
    \end{align}
    where the second inequality applies
    $||a+b|-|a||\le |b|$ component-wise.
    Taking expectations of \eqref{eq:A-noisy-decomposition}, and plugging in \eqref{eq:A-noisy-perturbation} give the desired lower bound.
\end{proof}

\begin{lemma}[Quadratic growth and PL condition, \cite{karimi2016linear}]
    \label{lemma:QG}
    Strong convexity (Assumption \ref{assumption:Lip}) implies: (a) quadratic growth condition $\frac{2}{\mu}(f(W) - f^*) \ge \|W - W^*\|^2$; and (b) Polyak--Lojasiewicz (PL) condition $\|\nabla f(W)\|^2 \ge 2\mu(f(W)-f^*)$, for any $W$.
\end{lemma}

\section{Proof of Theorem \ref{theorem:implicit-regularization}: Implicit Penalty of Analog Training}
\label{section:proof-implicit-regularization}

\begin{proof}[Proof of Theorem \ref{theorem:implicit-regularization}]
We first prove that $H$ is non-singular and hence \eqref{problem:implicit-regularization-solution} is well-defined.
Since $R(\cdotc)$ acts component-wise, $J_R(W^\diamond)$ is diagonal, and the absolute values of the diagonal components are its singular values.
Using the condition $q'_+(W)<-c_0, q'_-(W)>c_0$, the $d$-th diagonal component of $J_R(W)$ is given by
\begin{align}
    \label{inequality:implicit-regularization-S0}
    &\ [J_R(W)]_{dd}
    = \lp\frac{G([W]_d)}{F([W]_d)}\rp'
    = \lp\frac{q_-([W]_d)-q_+([W]_d)}{q_-([W]_d)+q_+([W]_d)}\rp' \\
    =&\ \frac{2(q_-'([W]_d)q_+([W]_d)-q_+'([W]_d)q_-([W]_d)))}{(q_-([W]_d)+q_+([W]_d))^2}
    \ge \frac{c_0 q_{\min}}{q_{\max}^2}
    > 0
    \nonumber
\end{align}
which implies $J_R(W)$ is positive definite. 
Since all components of $\Sigma$ are positive and $\nabla^2 f(W^*)$ is positive definite, $H = \nabla^2 f(W^*)+\Diag(\Sigma)\,J_R(W^\diamond)$ is also positive definite and hence non-singular.
Now we separately show the inequalities of Theorem \ref{theorem:implicit-regularization}. 

\noindent
\textbf{(Step 1a) Upper bound of $\|\nabla f_{\Sigma}(\tdW^*)\|$.}
The gradient of $f_{\Sigma}(W)$ is given by
\begin{align}
    \nabla f_{\Sigma}(W) = \nabla f(W) + \Sigma \odot R(W).
\end{align}
We will analyze the gradient at $\tdW^*$ via local Taylor expansions at $\tdW^*$. 
We begin by showing that $\tdW^*$ lies in $\ccalS$ for $W^\diamond$ sufficiently close to $W^*$.
By the definition of $\tdW^*$, it holds that
\begin{align}
    \label{inequality:implicit-regularization-S0-1}
    \|W^\diamond-\tdW^*\|
    =&\ \|
    H^{-1}
    \nabla^2 f(W^*) ~(W^*-W^\diamond)\|
    \le \sigma_{\max}(H^{-1}
    \nabla^2 f(W^*))\|W^\diamond-W^*\|,
    \\
    \label{inequality:implicit-regularization-S0-2}
    \|W^*-\tdW^*\|
    =&\ \|
    H^{-1}
    \Diag(\Sigma)\,J_R(W^\diamond) (W^*-W^\diamond)\|
    \le \sigma_{\max}(H^{-1}
    \Diag(\Sigma)\,J_R(W^\diamond)) \|W^\diamond-W^*\|.
\end{align}
Indeed, the second inequality and the local boundedness of the matrix factor imply that $\|\tdW^*-W^*\|$ is bounded by a constant multiple of $\|W^\diamond-W^*\|$ for $W^\diamond$ near $W^*$.
Since $\ccalS$ is open and contains $W^*$, this bound implies that $\tdW^*\in\ccalS$ whenever $W^\diamond$ is sufficiently close to $W^*$.
Because $\ccalS$ is convex, the line segments from $W^*$ to $\tdW^*$ and from $W^\diamond$ to $\tdW^*$ are contained in $\ccalS$.
Since $\nabla f(W^*)=0$ and $R(W^\diamond)=G(W^\diamond)/F(W^\diamond)=0$, Taylor's theorem and the local Lipschitz continuity of $\nabla^2 f(\cdotc)$ and $J_R(\cdotc)$ give
\begin{align}
    \label{inequality:implicit-regularization-taylor}
    \nabla f(\tdW^*) =&\ \nabla^2 f(W^*)(\tdW^* - W^*) + \ccalO(\|\tdW^* - W^*\|^2), \\
    R(\tdW^*)=&\ J_R(W^\diamond) (\tdW^*-W^\diamond)  + \ccalO(\|\tdW^* -W^\diamond\|^2),
    \nonumber
\end{align}
where the two remainder vectors have norms bounded by constants times $\|\tdW^* - W^*\|^2$ and $\|\tdW^* -W^\diamond\|^2$, respectively.
We bound the gradient of the penalized objective as follows
\begin{align}
    \label{inequality:implicit-regularization-S1-T1}
		&\ \|\nabla f_{\Sigma}(\tdW^*)\|
	= \lnorm \nabla f(\tdW^*) + \Sigma\odot \frac{G(\tdW^*)}{F(\tdW^*)} \rnorm
	\\
		=&\ \| \nabla^2 f(W^*)(\tdW^* - W^*) + \ccalO(\|\tdW^* - W^*\|^2) + \Diag(\Sigma) J_R(W^\diamond) (\tdW^*-W^\diamond)  + \ccalO(\|\tdW^* -W^\diamond\|^2) \|
		\nonumber\\
		=&\ \|\ccalO(\|\tdW^* - W^*\|^2) + \ccalO(\|\tdW^* -W^\diamond\|^2)\|
		\le \ccalO(\|W^* -W^\diamond\|^2)
		\nonumber
\end{align}
where the last equality uses the definition of $\tdW^*$, and the last inequality follows from \eqref{inequality:implicit-regularization-S0-1} and \eqref{inequality:implicit-regularization-S0-2}.
Consequently, there exists a constant $c_{1a}$ such that $\|\nabla f_{\Sigma}(\tdW^*)\|\le c_{1a}\|W^* -W^\diamond\|^2$.

\noindent
\textbf{(Step 1b) Lower bound of $\|\nabla f_{\Sigma}(W^\diamond)\|$ and $\|\nabla f_{\Sigma}(W^*)\|$.}
\begin{align}
    \|\nabla f_{\Sigma}(W^\diamond)\|
    =&\ \lnorm \nabla f(W^\diamond) + \Sigma \odot R(W^\diamond) \rnorm
    = \lnorm \nabla f(W^\diamond)\rnorm
    \ge \mu'\|W^*-W^\diamond\|
\end{align}
where the inequality holds as $\|\nabla f(W^\diamond)\|=\|\nabla f(W^\diamond)-\nabla f(W^*)\|\ge \mu'\|W^*-W^\diamond\|$. Similarly, inequality \eqref{inequality:implicit-regularization-S0} leads to the bound that
\begin{align}
    \label{inequality:implicit-regularization-step-1b-S1}
    \lnorm \Sigma \odot R(W^*)\rnorm
    = \|\Sigma \odot R(W^*)-\Sigma \odot R(W^\diamond)\|
    \ge \lp\min_{d\in[D]}[\Sigma]_d\rp\frac{c_0 q_{\min}}{q_{\max}^2} \|W^\diamond-W^*\|.
\end{align}
With the constant $c_b := \lp\min_{d\in[D]}[\Sigma]_d\rp\frac{c_0 q_{\min}}{q_{\max}^2}$, \eqref{inequality:implicit-regularization-step-1b-S1} leads to the desired lower bound
\begin{align}
    \|\nabla f_{\Sigma}(W^*)\|
    =&\ \lnorm \nabla f(W^*) + \Sigma \odot R(W^*) \rnorm
    = \lnorm \Sigma \odot R(W^*)\rnorm
    \ge c_b \|W^\diamond-W^*\|.
\end{align}

\noindent
\textbf{(Step 2a) Upper bound of $\|T(\tdW^*)\|$.}
By the definition of $T(\tdW^*)$, we have
\begin{align}
    \label{inequality:implicit-regularization-S1}
    &\ \|T(\tdW^*)\|
    = \lnorm \mbE_\xi[\nabla f(\tdW^*; \xi)] + \mbE_\xi[|\nabla f(\tdW^*; \xi)|] \odot R(\tdW^*) \rnorm
    \\
    =&\ \lnorm \nabla f(\tdW^*) + \mbE_\xi[|\nabla f(\tdW^*; \xi)|] \odot R(\tdW^*)\rnorm
    \nonumber \\
    \le&\ \lnorm \nabla f(\tdW^*) + \mbE_\xi[|\nabla f(W^*; \xi)|] \odot R(\tdW^*) \rnorm
    +\lnorm (\mbE_\xi[|\nabla f(\tdW^*; \xi)|]-\mbE_\xi[|\nabla f(W^*; \xi)|]) \odot R(\tdW^*) \rnorm
    \nonumber \\
    =&\ \|\nabla f_{\Sigma}(\tdW^*)\|
    +\lnorm (\mbE_\xi[|\nabla f(\tdW^*; \xi)|]-\mbE_\xi[|\nabla f(W^*; \xi)|]) \odot R(\tdW^*) \rnorm.
    \nonumber
\end{align}
The first term in the \ac{RHS} of \eqref{inequality:implicit-regularization-S1} is bounded by \eqref{inequality:implicit-regularization-S1-T1}.
By applying $||x|-|y|| \le |x-y|$ for any $x, y\in\reals$ at all components, the second term in the \ac{RHS} of \eqref{inequality:implicit-regularization-S1} is bounded by
\begin{align}
    \label{inequality:implicit-regularization-S1-T2}
    &\ \lnorm (\mbE_\xi[|\nabla f(\tdW^*; \xi)|]-\mbE_\xi[|\nabla f(W^*; \xi)|]) \odot R(\tdW^*) \rnorm
   	\\
    \le&\ \lnorm \mbE_\xi[|\nabla f(\tdW^*; \xi)-\nabla f(W^*; \xi)|] \odot \lp R(\tdW^*)-R(W^\diamond)\rp \rnorm
   	\nonumber \\
    \le&\ \lnorm \mbE_\xi[|\nabla f(\tdW^*; \xi)-\nabla f(W^*; \xi)|]\rnorm \lnorm R(\tdW^*)-R(W^\diamond) \rnorm
   	\nonumber \\
    \le&\ \mbE_\xi\lB\lnorm \nabla f(\tdW^*; \xi)-\nabla f(W^*; \xi)\rnorm\rB \lnorm R(\tdW^*)-R(W^\diamond) \rnorm
   	\nonumber \\
   	\le&\ \ccalO(\|\tdW^*-W^*\| \|\tdW^* - W^\diamond\|)
   	= \ccalO(\|W^*-W^\diamond\|^2).
   	\nonumber
\end{align}
Plugging back \eqref{inequality:implicit-regularization-S1-T1} and \eqref{inequality:implicit-regularization-S1-T2} into \eqref{inequality:implicit-regularization-S1} shows that there exists a constant $c_{2a} > 0$ such that $\|T(\tdW^*)\| \le c_{2a}\|W^\diamond-W^*\|^2$.

\noindent
\textbf{(Step 2b) Lower bound of $\|T(W^\diamond)\|$ and $\|T(W^*)\|$.} Similar to Step 1b, it holds that
\begin{align}
    \|T(W^\diamond)\|
    =&\, \lnorm \mbE_\xi[\nabla f(W^\diamond; \xi)]  + \mbE_\xi[|\nabla f(W^\diamond; \xi)|] \odot R(W^\diamond) \rnorm
    = \lnorm \nabla f(W^\diamond)\rnorm
    \ge \mu'\|W^*-W^\diamond\|, \\
    \|T(W^*)\|
    =&\, \lnorm \mbE_\xi[\nabla f(W^*; \xi)] + \mbE_\xi[|\nabla f(W^*; \xi)|] \odot R(W^*) \rnorm
    = \lnorm \Sigma \odot R(W^*)\rnorm 
    \ge c_b\|W^*-W^\diamond\|.
\end{align}
Defining $c_1 := \max\{c_{1a}, c_{2a}\}$ and $c_2  := \min\{\mu', c_b\}$, we complete the proof.
\end{proof}

\section{Proof of Theorem \ref{theorem:TT-convergence-scvx}: Convergence of \ResidualLearning}
\label{section:proof-TT-convergence-scvx}
This section provides the proof of lemmas that are required by the proof of Theorem \ref{theorem:TT-convergence-scvx}.
\subsection{Proof of Lemma \ref{lemma:TT-barW-descent}: Descent of sequence $\bar W_k$}
\label{section:proof-lemma:TT-barW-descent}
\begin{proof}[Proof of Lemma \ref{lemma:TT-barW-descent}]
The $L$-smooth assumption (Assumption \ref{assumption:Lip}) implies that
\begin{align}
    \label{inequality:TT-convergence-1}
    &\ \mathbb{E}_{\xi_k}[f(\bar W_{k+1})] \le f(\bar W_k)
    +\mathbb{E}_{\xi_k}[\la \nabla f(\bar W_k), \bar W_{k+1}-\bar W_k\ra]
    + \frac{L}{2}\mathbb{E}_{\xi_k}[\|\bar W_{k+1}-\bar W_k\|^2] \\
    =&\ f(\bar W_k)
    + \underbrace{\gamma \mathbb{E}_{\xi_k}[\la \nabla f(\bar W_k), P_{k+1}-P_k\ra]}_{(a)}
    + \underbrace{\mathbb{E}_{\xi_k}[\la \nabla f(\bar W_k), W_{k+1}-W_k\ra]}_{(b)}
    + \underbrace{\frac{L}{2}\mathbb{E}_{\xi_k}[\|\bar W_{k+1}-\bar W_k\|^2]}_{(c)}.
    \nonumber
\end{align}
Next, we will handle each term in the \ac{RHS} of \eqref{inequality:TT-convergence-1} separately.

\noindent
\textbf{Bound of the second term (a).}
To bound term (a) in the \ac{RHS} of \eqref{inequality:TT-convergence-1}, write the $P$-update as
\begin{align}
    \label{inequality:TT-convergence-1-T2}
    P_{k+1}-P_k
    =-\alpha A\bigl(P_k,\nabla f(\bar W_k;\xi_k)\bigr).
\end{align}
Combining \eqref{inequality:TT-convergence-1-T2} with the zero-mean noisy lower bound in Lemma \ref{lemma:A-op-noisy}, applied with $Z=P_k$, $\Delta Z=\nabla f(\bar W_k)$, and $\delta Z=\nabla f(\bar W_k;\xi_k)-\nabla f(\bar W_k)$, gives
\begin{align}
    \label{inequality:TT-convergence-1-T2-mixed}
    \gamma\mathbb{E}_{\xi_k}[\la \nabla f(\bar W_k),P_{k+1}-P_k\ra]
    =&\ -\alpha\gamma\mathbb{E}_{\xi_k}\lB
    \la \nabla f(\bar W_k),A\bigl(P_k,\nabla f(\bar W_k;\xi_k)\bigr)\ra
    \rB
    \\
    \le&\ -\frac{3\alpha\gamma q_{\min}}{4}\lnorm\nabla f(\bar W_k)\rnorm^2
    +\frac{\alpha\gamma\sigma^2}{q_{\min}}\lnorm G(P_k)\rnorm_\infty^2.
    \nonumber
\end{align}
Thus, term (a) is bounded by
\begin{align}
    \label{inequality:TT-convergence-1-T2-2}
    \gamma\mathbb{E}_{\xi_k}[\la \nabla f(\bar W_k), P_{k+1}-P_k\ra]
    \le&\ -\frac{3\alpha\gamma q_{\min}}{4}\lnorm\nabla f(\bar W_k)\rnorm^2
    +\frac{\alpha\gamma\sigma^2}{q_{\min}}\lnorm G(P_k)\rnorm_\infty^2.
\end{align}

\noindent
\textbf{Bound of the third term (b).} By Young's inequality,
we have
\begin{align}
    \label{inequality:TT-convergence-1-T3}
    \mathbb{E}_{\xi_k}[\la \nabla f(\bar W_k), W_{k+1}-W_k\ra]
    \le&\ \frac{\alpha\gamma q_{\min}}{4}\lnorm\nabla f(\bar W_k)\rnorm^2
    + \frac{1}{\alpha\gamma q_{\min}}\mathbb{E}_{\xi_k}[\|W_{k+1}-W_k\|^2].
\end{align}

\noindent
\textbf{Bound of the third term (c).} Repeatedly applying inequality $\|U+V\|^2\le 2\|U\|^2+2\|V\|^2$ for any $U, V\in\reals^D$, and using \eqref{eq:A-lip}, we have
\begin{align}
    \label{inequality:TT-convergence-1-T4}
    \frac{L}{2}\mathbb{E}_{\xi_k}[\|\bar W_{k+1}-\bar W_k\|^2]
    \le&\ L\mathbb{E}_{\xi_k}[\|W_{k+1}-W_k\|^2]
    + \gamma^2 L\mathbb{E}_{\xi_k}[\|P_{k+1}-P_k\|^2] 
    \\
    \le&\ L\mathbb{E}_{\xi_k}[\|W_{k+1}-W_k\|^2]
    + \alpha^2\gamma^2 L q_{\max}^2
    \lp \lnorm\nabla f(\bar W_k)\rnorm^2+\sigma^2\rp
    \nonumber
\end{align}
where the last inequality follows from
\begin{align}
    \label{inequality:TT-convergence-1-T4-S1-T3}
    \mathbb{E}_{\xi_k}[\|P_{k+1}-P_k\|^2]
    =&\ \alpha^2\mathbb{E}_{\xi_k}\lB
    \lnorm A\bigl(P_k,\nabla f(\bar W_k;\xi_k)\bigr)\rnorm^2
    \rB\\
    \le&\ \alpha^2 q_{\max}^2
    \mathbb{E}_{\xi_k}\lB\lnorm\nabla f(\bar W_k;\xi_k)\rnorm^2\rB
    \le \alpha^2 q_{\max}^2
    \lp\lnorm\nabla f(\bar W_k)\rnorm^2+\sigma^2\rp.
    \nonumber
\end{align}

\noindent
\textbf{Combination of the upper bound $(a)$, $(b)$, and $(c)$.} Plugging \eqref{inequality:TT-convergence-1-T2-2}, \eqref{inequality:TT-convergence-1-T3}, and \eqref{inequality:TT-convergence-1-T4} into \eqref{inequality:TT-convergence-1}, and using the stepsize rules
$\alpha\le q_{\min}/(4\gamma Lq_{\max}^2)$ and
$\alpha\le 1/(\gamma Lq_{\min})$ together with the PL condition in Lemma \ref{lemma:QG}, we derive
\begin{align}
    \label{inequality:TT-convergence-2}
    \saveeq{inequality-saved:TT-barW-descent}{
        \mathbb{E}_{\xi_k}[f(\bar W_{k+1})-f^*]
        \le&\ \lp 1-\frac{\alpha\gamma\mu q_{\min}}{2}\rp
        (f(\bar W_k)-f^*)
        +\frac{\alpha\gamma\sigma^2}{q_{\min}}\lnorm G(P_k)\rnorm^2_\infty
        \\
        \hspace{-1em}
        &\ + \alpha^2\gamma^2 L q_{\max}^2 \sigma^2
        + \frac{2}{\alpha\gamma q_{\min}}~\mathbb{E}_{\xi_k}\lB\|W_{k+1}-W_k\|^2\rB
        \nonumber.
    }
\end{align}
Thus the proofs of \eqref{inequality:TT-barW-descent} is completed.
Leveraging the $L$-smooth assumption (Assumption \ref{assumption:Lip}), we have
\begin{align}
    \label{inequality:TT-P-convergence-1}
    &\ \mathbb{E}_{\xi_k}[f(\bar W_{k+\frac{1}{2}})]
    = f(\bar W_k)
    + \gamma {\mathbb{E}_{\xi_k}[\la \nabla f(\bar W_k), P_{k+1}-P_k\ra]}
    + {\frac{L\gamma^2}{2}\mathbb{E}_{\xi_k}[\|P_{k+1}-P_k\|^2]}.
\end{align}
Plugging \eqref{inequality:TT-convergence-1-T2-2} and \eqref{inequality:TT-convergence-1-T4-S1-T3} into \eqref{inequality:TT-P-convergence-1}, and using $\alpha\le q_{\min}/(2\gamma Lq_{\max}^2)$, we have
\begin{align}
    \saveeq{inequality-saved:TT-P-in-barW-descent}{
        \mathbb{E}_{\xi_k}[f(\bar W_{k+\frac{1}{2}})-f^*]
        \le&\ \lp 1-\frac{\alpha\gamma\mu q_{\min}}{2}\rp
        (f(\bar W_k)-f^*)
        + \frac{\alpha\gamma\sigma^2}{q_{\min}}\lnorm G(P_k)\rnorm^2_\infty
        + \frac{\alpha^2\gamma^2 L q_{\max}^2\sigma^2}{2}.
    }
\end{align}
Now the proof of \eqref{inequality:TT-P-in-barW-descent} is completed.
\end{proof}

\subsection{Proof of Lemma \ref{lemma:TT-W-descent}: Descent of sequence $W_k$}
\label{section:proof-lemma:TT-W-descent}
\begin{proof}[Proof of Lemma \ref{lemma:TT-W-descent}]
Recall $P^*(W) = (W^* - W) / \gamma + W^\diamond$. It holds that
\begin{align}
    \label{inequality:TT-convergence-scvx-S1}
    &\ \|P^*(W_{k+1})-W^\diamond\|^2
    = \frac{1}{\gamma^2}\|W_{k+1}-W^*\|^2\\
    =&\ \frac{1}{\gamma^2}\|W_k-W^*\|^2
    + \frac{2}{\gamma^2}\la W_k-W^*, W_{k+1}-W_k\ra
    + \frac{1}{\gamma^2}\|W_{k+1}-W_k\|^2
    \nonumber \\
    =&\ \|P^*(W_k)-W^\diamond\|^2 - \frac{2}{\gamma}\la P^*(W_k)-W^\diamond, W_{k+1}-W_k\ra + \frac{1}{\gamma^2}\|W_{k+1}-W_k\|^2.
    \nonumber
\end{align}
Writing \eqref{recursion:HD-W} in the argument order of \eqref{analog-update}, the $W$-update is
\begin{align}
    \label{eq:W-step-as-A}
    W_{k+1}-W_k
    =\beta\gamma A(P_{k+1}-W^\diamond;W_k).
\end{align}
We bound the second term in the \ac{RHS} of \eqref{inequality:TT-convergence-scvx-S1} by applying the joint deterministic perturbation bound \eqref{eq:A-perturbed-lower-young} in Lemma \ref{lemma:A-op-noisy}.
The bound is applied with $Z=W_k$, $\Delta Z=P^*(W_k)-W^\diamond$, and $\delta Z=P_{k+1}-P^*(W_k)$:
\begin{align}
    \label{inequality:TT-convergence-scvx-S1-T2-S1}
    &\ -\frac{2}{\gamma}\la P^*(W_k)-W^\diamond, W_{k+1}-W_k\ra
    = -2\beta
    \la P^*(W_k)-W^\diamond,
    A(P_{k+1}-W^\diamond;W_k)\ra
    \\
    \le&\
    -\beta q_{\min}\lnorm P^*(W_k)-W^\diamond\rnorm^2
    +\beta q_{\max}
    \lnorm P_{k+1}-P^*(W_k)\rnorm^2
    -\frac{1}{\beta\gamma^2 q_{\max}}\|W_{k+1}-W_k\|^2.
    \nonumber
\end{align}
Plugging \eqref{inequality:TT-convergence-scvx-S1-T2-S1} into \eqref{inequality:TT-convergence-scvx-S1} gives
\begin{align}
    \label{inequality:TT-convergence-scvx-S1-pre-simplified}
    \|P^*(W_{k+1})-W^\diamond\|^2 
    \le&\ 
    \lp 1- \beta q_{\min}\rp\|P^*(W_k)-W^\diamond\|^2\\
    &\ 
    + \beta q_{\max}
    \lnorm P_{k+1} - P^*(W_k)\rnorm^2
    -\frac{1}{\gamma^2}\lp \frac{1}{\beta q_{\max}}-1\rp\|W_{k+1}-W_k\|^2.
    \nonumber
\end{align}
Using $\frac{1}{\beta q_{\max}}-1\ge\frac{1}{2\beta q_{\max}}$, which follows from $\beta\le 1/(2q_{\max})$, in \eqref{inequality:TT-convergence-scvx-S1-pre-simplified} gives \eqref{inequality:TT-W-descent}:
\begin{align}
    \saveeq{inequality-saved:TT-W-descent}{
    \|P^*(W_{k+1})-W^\diamond\|^2 
    \le&\ 
    \lp 1- \beta q_{\min}\rp\|P^*(W_k)-W^\diamond\|^2
    + \beta q_{\max}
    \lnorm P_{k+1} - P^*(W_k)\rnorm^2\\
    &\ -\frac{1}{2\beta\gamma^2 q_{\max}}\|W_{k+1}-W_k\|^2
    \nonumber
    }
\end{align}
which completes the proof.
\end{proof}

\subsection{Proof of Lemma \ref{lemma:TT-coefficient-lower-bounds}: Coefficient lower bounds}
\label{section:proof-lemma:TT-coefficient-lower-bounds}
\begin{proof}[Proof of Lemma \ref{lemma:TT-coefficient-lower-bounds}]
The residual coupling in $A_1$ is controlled as follows:
\begin{align}
    \label{eq:TT-coupling-beta-check}
     \frac{2}{\mu\gamma^2}\cdot\frac{\beta C\gamma^2q_{\max}}{2}
    = \frac{16\beta^2 q_{\max}^2}{\alpha\mu\gamma q_{\min}}
    = \frac{\alpha\gamma\mu q_{\min}}{16},
\end{align}
where the first identity uses the choice of $C$, and the second identity uses the choice of $\beta$.
The lower bound on $\gamma$ in \eqref{eq:C-gamma-choice} controls the two hardware terms:
\begin{align}
    \label{eq:C-gamma-check}
    \frac{8\alpha\sigma^2 L_S^2}{\mu\gamma q_{\min}}
    \le \rho,
    \qquad
    \frac{4\alpha\gamma\sigma^2 L_S^2}{q_{\min}}
    \le \frac{\beta q_{\min}}{4} C\gamma^2.
\end{align}
Accordingly, by the definition of $A_1$,
\begin{align}
    A_1
    =&\ \restateeq{definition:coefficient-A1}\\
    \ge&\ \frac{\alpha\gamma\mu q_{\min}}{2}
    - \frac{\alpha\gamma\mu q_{\min}}{16}
    - \rho
    > \rho,
    \nonumber
\end{align}
where the last inequality uses $\rho=\alpha\gamma\mu q_{\min}^{2}/(32q_{\max})\le\alpha\gamma\mu q_{\min}/32$.
Similarly, by the definition of $A_2$,
\begin{align}
    A_2
    = \restateeq{definition:coefficient-A2}
    \ge&\ \frac{\beta q_{\min}}{2}C\gamma^2
    - \frac{\beta q_{\min}}{4}C\gamma^2
    = \rho C\gamma^2/2.
\end{align}
Finally, $C$ in \eqref{eq:C-gamma-choice} is chosen to satisfy the drift cancellation condition
\begin{align}
    A_3
    = \restateeq{definition:coefficient-A3}
    = \frac{2}{\alpha\gamma q_{\min}} - \frac{4}{\alpha\gamma q_{\min}}
    \le 0.
\end{align}
This proves the desired coefficient bounds.
\end{proof}

\section{Proof of Theorem \ref{theorem:lower-bound}}
\label{section:proof-lower-bound}
This section provides the proof of the lemmas used in the proof of Theorem \ref{theorem:lower-bound}.

\subsection{Proof of Lemma~\ref{lemma:conditional-mean-analog-increment}}
\label{section:proof-lemma:conditional-mean-analog-increment}

\begin{proof}
    In this one-step calculation, $P_k$ and $\bar W_k$ are fixed and the
    expectation is taken only over the two atoms of the fresh noise $\xi_k$ in \eqref{eq:skewed-noise}.
    Since $\nabla f(\bar W_k;\xi_k)=\nabla f(\bar W_k)+\xi_k$ and
    $p_k=\frac{q_+(P_k)}{q_+(P_k)+q_-(P_k)}$, $1-p_k=\frac{q_-(P_k)}{q_+(P_k)+q_-(P_k)}$, the sign condition
    \eqref{eq:conditional-absolute-value-sign-condition} removes the absolute
    values at the two noise atoms,
    \begin{align}
        \label{eq:conditional-absolute-value-remove-abs}
        \left|\nabla f(\bar W_k)
        +\sigma\sqrt{\frac{1-p_k}{p_k}}\right|
        &=
        \nabla f(\bar W_k)
        +\sigma\sqrt{\frac{1-p_k}{p_k}},\\
        \left|\nabla f(\bar W_k)
        -\sigma\sqrt{\frac{p_k}{1-p_k}}\right|
        &=
        \sigma\sqrt{\frac{p_k}{1-p_k}}
        -\nabla f(\bar W_k),
        \nonumber
    \end{align}
    so that, using the zero mean of $\xi_k$,
    \begin{align}
        \label{eq:conditional-absolute-value-proof}
        \mathbb{E}_{\xi_k}\bigl[
            |\nabla f(\bar W_k; \xi_k)|
        \bigr]
        =&\
        p_k\left(
            \nabla f(\bar W_k)
            +\sigma\sqrt{\frac{1-p_k}{p_k}}
        \right)
        +(1-p_k)\left(
            \sigma\sqrt{\frac{p_k}{1-p_k}}
            -\nabla f(\bar W_k)
        \right)
        \\
        =&\
        2\sigma\sqrt{p_k(1-p_k)}
        +(2p_k-1)\nabla f(\bar W_k)
        =
        \sigma\frac{\sqrt{q_+(P_k)q_-(P_k)}}{F(P_k)}
        -\frac{G(P_k)}{F(P_k)}\nabla f(\bar W_k),
        \nonumber
    \end{align}
    The last equality follows from the two probability identities
    \begin{align}
        \label{eq:noise-probability-identities}
        2\sqrt{p_k(1-p_k)}
        =&\
        \frac{2\sqrt{q_+(P_k)q_-(P_k)}}{q_+(P_k)+q_-(P_k)}
        =\frac{\sqrt{q_+(P_k)q_-(P_k)}}{F(P_k)},
        \\
        2p_k-1
        =&\
        \frac{q_+(P_k)-q_-(P_k)}{q_+(P_k)+q_-(P_k)}
        =-\frac{G(P_k)}{F(P_k)}.
        \nonumber
    \end{align}
    Expanding the effective gradient increment
    $F(P_k)\nabla f(\bar W_k;\xi_k)+|\nabla f(\bar W_k;\xi_k)|G(P_k)$, taking the conditional expectation, and inserting the previous display gives
    \begin{align}
        \label{eq:A-mean-derivation}
        &\ \mathbb{E}_{\xi_k}\bigl[
            F(P_k)\nabla f(\bar W_k;\xi_k)
            +|\nabla f(\bar W_k;\xi_k)|G(P_k)
        \bigr]
        \\
        =&\
        F(P_k)\nabla f(\bar W_k)
        +\Bigl(
            \sigma\frac{\sqrt{q_+(P_k)q_-(P_k)}}{F(P_k)}
            -\frac{G(P_k)}{F(P_k)}\nabla f(\bar W_k)
        \Bigr)G(P_k)
        \nonumber
        \\
        =&\
        \frac{F(P_k)^2-G(P_k)^2}{F(P_k)}\nabla f(\bar W_k)
        +\sigma\frac{\sqrt{q_+(P_k)q_-(P_k)}}{F(P_k)}\,G(P_k).
        \nonumber
    \end{align}
    and the identity $F(P_k)^2-G(P_k)^2=q_+(P_k)q_-(P_k)$ proves
    \eqref{eq:quadratic-example-A-mean}.
\end{proof}

\subsection{Proof of Lemma~\ref{lemma:terminal-mean-square-lower-bound}}
\label{section:proof-lemma:terminal-mean-square-lower-bound}

\begin{proof}
    Throughout this proof, the current iterates are fixed, and the expectation is taken only over the fresh noise $\xi_k$.
    The argument has three parts.
    We first prove a global lower bound on the squared conditional mean.
    We then obtain a $q_{\max}$-scale conditional variance on $\mathcal{E}_k$.
    Finally, the bias--variance decomposition combines these two estimates without conditioning on the future event $\mathcal{E}$.

    \textbf{(I) Contribution of conditional mean.}
    The update \eqref{eq:residual-learning-iterations}--\eqref{eq:quadratic-example-simplified-W-update}, expressed in terms of $R_k$, is
    \begin{align}
        \label{eq:pair-exact-one-step}
        R_{k+1}
        =&\ R_k
        + \underbrace{
            \sqrt{\frac{\mu}{2}}\,\beta\gamma q_{\min}P_k
            \begin{bmatrix}
                (\frac{q_{\min}}{q_{\max}})^{1/4}\\[2pt]
                1
            \end{bmatrix}
        }_{=:T_1}
        + \underbrace{
            \sqrt{\frac{\mu}{2}}\,\alpha\gamma
            A\bigl(-\nabla f(\bar W_k;\xi_k);P_k\bigr)
            \begin{bmatrix}
                (\frac{q_{\min}}{q_{\max}})^{1/4}\beta q_{\min}\\[2pt]
                1+\beta q_{\min}
            \end{bmatrix}
        }_{=:T_2}.
    \end{align}
    The definition \eqref{eq:pair-R-def}, Cauchy--Schwarz, and
    $q_{\min}\le q_{\max}$ imply
    \begin{align}
        \label{eq:pair-coordinate-control}
        \sqrt{\frac{\mu}{2}}\,|\bar W_k-W^\star|
        \le&\ \lVert R_k\rVert,
        \qquad
        \sqrt{\frac{\mu}{2}}\,\gamma|P_k|
        \le
        \sqrt{2}
        \left(\frac{q_{\max}}{q_{\min}}\right)^{1/4}
        \lVert R_k\rVert.
    \end{align}
    We first separately bound the two perturbation terms $T_1$ and $T_2$ in the one-step update. 
    Since $0<q_{\min}/q_{\max}\le1$, the coordinate bound and the exact choice of $\beta$ in \eqref{eq:alpha-beta-choice} give
    \begin{align}
        \label{eq:global-transfer-mean-bound}
        &\ 
        \|T_1\|
        = \sqrt{\frac{\mu}{2}}\,\beta\gamma q_{\min}|P_k|
        \sqrt{1+\sqrt{\frac{q_{\min}}{q_{\max}}}}
        \\
        \le&\
        2\beta q_{\min}
        \left(\frac{q_{\max}}{q_{\min}}\right)^{1/4}
        \lVert R_k\rVert
        =
        \frac{\alpha\gamma\mu q_{\max}}{8}
        \left(\frac{q_{\min}}{q_{\max}}\right)^{7/4}
        \lVert R_k\rVert
        \le
        \frac{\alpha\gamma\mu q_{\max}}{8}\lVert R_k\rVert.
        \nonumber
    \end{align}
    By the update rule \eqref{analog-update}, $-A(-g;P_k)=F(P_k)g+|g|G(P_k)$ for every scalar $g$.
    Since $F(P_k)+|G(P_k)|\le q_{\max}$,
    $|G(P_k)|\le L_S|P_k|$, and conditional Cauchy--Schwarz gives
    $\mathbb{E}_{\xi_k}[|\xi_k|]\le\sigma$, one has
    \begin{align}
        \label{eq:global-analog-mean-bound}
        \left|
            \mathbb{E}_{\xi_k}\bigl[
                A\bigl(-\nabla f(\bar W_k;\xi_k);P_k\bigr)
            \bigr]
        \right|
        \le&\
        F(P_k)\mu|\bar W_k-W^\star|
        +|G(P_k)|\mathbb{E}_{\xi_k}
        \bigl[|\mu(\bar W_k-W^\star)+\xi_k|\bigr]
        \\
        \le&\
        \mu q_{\max}|\bar W_k-W^\star|
        +L_S\sigma|P_k|.
        \nonumber
    \end{align}
    Combining \eqref{eq:global-analog-mean-bound} and \eqref{eq:pair-coordinate-control}, the norm of the conditional mean of the analog contribution is at most
    \begin{align}
        \label{eq:global-analog-vector-mean-bound}
        \|(T_2)\|
        = &\ \sqrt{\frac{\mu}{2}}\,\alpha\gamma
        \left|
            \mathbb{E}_{\xi_k}\bigl[
                A\bigl(-\nabla f(\bar W_k;\xi_k);P_k\bigr)
            \bigr]
        \right|
        \sqrt{
            \sqrt{\frac{q_{\min}}{q_{\max}}}\,\beta^2q_{\min}^2
            +(1+\beta q_{\min})^2
        }
        \\
        \le&\
        \frac{3\alpha}{2}
        \left(
            \mu\gamma q_{\max}
            +\sqrt{2}
            \left(\frac{q_{\max}}{q_{\min}}\right)^{1/4}
            L_S\sigma
        \right)\lVert R_k\rVert.
        \nonumber
    \end{align}
    We now combine the two perturbation bounds.
    Taking the conditional expectation in \eqref{eq:pair-exact-one-step}, the transfer vector remains deterministic.
    The norms of the transfer vector and the conditional mean of the analog vector are bounded by \eqref{eq:global-transfer-mean-bound} and \eqref{eq:global-analog-vector-mean-bound}, respectively.
    Therefore, the reverse triangle inequality gives
    \begin{align}
        \label{eq:global-pair-mean-lower}
        \left\lVert\mathbb{E}_{\xi_k}[R_{k+1}]\right\rVert
        \ge&\
        \lVert R_k\rVert
        -\frac{\alpha\gamma\mu q_{\max}}{8}\lVert R_k\rVert
        -\frac{3\alpha}{2}
        \left(
            \mu\gamma q_{\max}
            +\sqrt{2}
            \left(\frac{q_{\max}}{q_{\min}}\right)^{1/4}
            L_S\sigma
        \right)\lVert R_k\rVert
        \\
        \ge&\
        \left(
            1-2\alpha\left(
                2\mu\gamma q_{\max}
                +\sqrt{2}
                \left(\frac{q_{\max}}{q_{\min}}\right)^{1/4}
                L_S\sigma
            \right)
        \right)\lVert R_k\rVert.
        \nonumber
    \end{align}
    The lower bound on $\gamma$ in \eqref{eq:C-gamma-choice} and $\kappa_2\ge1$ imply $\sqrt{2}
        \left(\frac{q_{\max}}{q_{\min}}\right)^{1/4}
        L_S\sigma
        \le \mu\gamma q_{\max}$.
    Combining it and \eqref{eq:global-pair-mean-lower} gives $\left\lVert\mathbb{E}_{\xi_k}[R_{k+1}]\right\rVert
        \ge
        \left(1-6\alpha\mu\gamma q_{\max}\right)\lVert R_k\rVert$.
    The stepsize order $\alpha=\tdTheta\lp 1/K\rp$ ensures that $12\alpha\mu\gamma q_{\max}\le1$ for sufficiently large $K$ with fixed hardware parameters.
    Squaring it and using $(1-x)^2\ge1-2x$ yields the global estimate
    \begin{align}
        \label{eq:global-pair-mean-square-lower}
        \left\lVert\mathbb{E}_{\xi_k}[R_{k+1}]\right\rVert^2
        \ge
        \left(1-12\alpha\mu\gamma q_{\max}\right)\lVert R_k\rVert^2.
    \end{align}
    This bound is valid for every current iterate and does not depend on the confinement event.

    \textbf{(II) Contribution of variance.}
    We now suppose that $\mathcal{E}_k$ occurs.
    The ceiling placement \eqref{eq:symmetric-point-at-ceiling}, the $L_S$-Lipschitz property, and \eqref{eq:C-gamma-choice} then imply
    \begin{align}
        \label{eq:ceiling-plateau}
        q_\pm(P_k)
        \ge&\
        q_{\max}-L_S|P_k|
        \ge
        q_{\max}
        -\frac{2L_S\sigma}{\mu\gamma}
        \sqrt{\frac{q_{\min}}{q_{\max}}}
        \ge
        \frac{q_{\max}}{2}.
    \end{align}
    Fixing the current iterates, the map
    $x\mapsto -A\bigl(-(\nabla f(\bar W_k)+x);P_k\bigr)$ is continuous and piecewise linear, with every slope at least $q_{\max}/2$ by \eqref{eq:ceiling-plateau}. Hence, for all $x\le y$,
    \begin{align}
        \label{eq:increment-two-sided-lipschitz}
        &\ -A\bigl(-(\nabla f(\bar W_k)+y);P_k\bigr)
        +A\bigl(-(\nabla f(\bar W_k)+x);P_k\bigr)
        \ge
        \frac{q_{\max}}{2}(y-x).
    \end{align}
    Under \eqref{eq:skewed-noise}, $\xi_k$ takes the two values
    $\xi_k^+:=\sigma\sqrt{(1-p_k)/p_k}$ and
    $\xi_k^-:=-\sigma\sqrt{p_k/(1-p_k)}$. The two-point variance identity and \eqref{eq:increment-two-sided-lipschitz} give
    \begin{align}
        \label{eq:conditional-variance-sandwich}
        &\ \operatorname{Var}_{\xi_k}\bigl[
            A\bigl(-\nabla f(\bar W_k;\xi_k);P_k\bigr)
        \bigr]
        \\
        =&\
        p_k(1-p_k)
        \bigl(
            A\bigl(-(\nabla f(\bar W_k)+\xi_k^-);P_k\bigr)
            -A\bigl(-(\nabla f(\bar W_k)+\xi_k^+);P_k\bigr)
        \bigr)^2
        \nonumber
        \\
        \ge&\
        \frac{q_{\max}^2}{4}
        p_k(1-p_k)(\xi_k^+-\xi_k^-)^2
        =\frac{q_{\max}^2\sigma^2}{4}.
        \nonumber
    \end{align}
    Subtracting the conditional mean from \eqref{eq:pair-exact-one-step} and using \eqref{eq:conditional-variance-sandwich} therefore gives on $\mathcal{E}_k$:
    \begin{align}
        \label{eq:pair-fluctuation-2}
        &\ \mathbb{E}_{\xi_k}\bigl[
            \lVert R_{k+1}-\mathbb{E}_{\xi_k}[R_{k+1}]\rVert^2
        \bigr]
        \\
        =&\
        \frac{\mu}{2}\alpha^2\gamma^2
        \left(
            (1+\beta q_{\min})^2
            +\sqrt{\frac{q_{\min}}{q_{\max}}}\,\beta^2q_{\min}^2
        \right)
        \operatorname{Var}_{\xi_k}\bigl[
            A\bigl(-\nabla f(\bar W_k;\xi_k);P_k\bigr)
        \bigr]
        \ge
        \frac{\mu}{8}\alpha^2\gamma^2q_{\max}^2\sigma^2.
        \nonumber
    \end{align}
    Outside $\mathcal{E}_k$, the same conditional variance remains nonnegative. Thus, the indicator in \eqref{eq:terminal-mean-square-lower-bound} records exactly where the lower bound \eqref{eq:pair-fluctuation-2} is used.

    \textbf{(III) Combination.}
    The bias--variance decomposition combines the global conditional-mean estimate \eqref{eq:global-pair-mean-square-lower} with the localized variance estimate \eqref{eq:pair-fluctuation-2}. Consequently,
    \begin{align}
        \label{eq:pair-cond-second-moment}
        \mathbb{E}_{\xi_k}\bigl[\lVert R_{k+1}\rVert^2\bigr]
        =&\
        \left\lVert\mathbb{E}_{\xi_k}[R_{k+1}]\right\rVert^2
        +\mathbb{E}_{\xi_k}\bigl[
            \lVert R_{k+1}-\mathbb{E}_{\xi_k}[R_{k+1}]\rVert^2
        \bigr]
        \\
        \ge&\
        \left(1-12\alpha\mu\gamma q_{\max}\right)\lVert R_k\rVert^2
        +\frac{\mu}{8}\alpha^2\gamma^2q_{\max}^2\sigma^2
        \mathbf{1}_{\mathcal{E}_k}.
        \nonumber
    \end{align}
    This is \eqref{eq:terminal-mean-square-lower-bound}.
\end{proof}

\subsection{Proof of Lemma~\ref{lemma:confinement}}
\label{section:proof-lemma:confinement}

\begin{proof}
    The key point is that the one-step drift \eqref{inequality:TT-convergence-Lya-5} was established for the general strongly convex objective under Assumption~\ref{assumption:noise} and never invokes the sign condition \eqref{eq:conditional-absolute-value-sign-condition}.
    It therefore holds conditionally on every current iterate of the hard instance, whether or not the iterate is confined.
    We use this globally valid drift to construct a nonnegative supermartingale from $V_k$.
    Doob's maximal inequality then controls the probability of $\mathcal{E}^c$, which is precisely what is needed to activate the localized variance term in \eqref{eq:terminal-mean-square-lower-bound}.

    The initialization \eqref{eq:initial-lyapunov-choice} is chosen so that, using $\bar W_0=W_0$ and $P^*(W_0)-W^\diamond=(W^\star-W_0)/\gamma$,
    \begin{align}
        \label{eq:confinement-initial-V}
        V_0
        =&\
        f(\bar W_0)-f^*
        +\frac{C\gamma^2}{2}\lVert P^*(W_0)-W^\diamond\rVert^2
        =
        \frac{\sigma^2}{16\mu}
        \left(\frac{q_{\min}}{q_{\max}}\right)^{3/2}.
    \end{align}
    In particular, $V_0$ is independent of both $\alpha$ and $K$, so the stepsize definition in \eqref{eq:alpha-beta-choice} is non-circular.
    Dropping the contraction term in \eqref{inequality:TT-convergence-Lya-5} gives
    \begin{align}
        \label{eq:confinement-one-step-compensated}
        \mathbb{E}_{\xi_k}[V_{k+1}]
        \le
        V_k
        +
        2\alpha^2\gamma^2 Lq_{\max}^2\sigma^2 .
    \end{align}
    For the finite horizon $K$, define
    \begin{align}
        \label{eq:confinement-compensated-supermartingale}
        \tilde V_k
        :=
        V_k
        +
        2\alpha^2\gamma^2 Lq_{\max}^2\sigma^2(K-k),
        \qquad
        0\le k\le K,
    \end{align}
    This process is a supermartingale with respect to the natural history of the noise sequence because, conditionally on the current iterates,
    \begin{align}
        \label{eq:confinement-compensated-supermartingale-step}
        \mathbb{E}_{\xi_k}[\tilde V_{k+1}]
        =&\
        \mathbb{E}_{\xi_k}[V_{k+1}]
        +
        2\alpha^2\gamma^2 Lq_{\max}^2\sigma^2(K-k-1)
        \le
        V_k
        +
        2\alpha^2\gamma^2 Lq_{\max}^2\sigma^2(K-k)
        =
        \tilde V_k.
    \end{align}
    Using the stepsize order $\alpha=\tdTheta\lp q_{\max}/(\gamma\mu q_{\min}^2K)\rp$ and $L=\mu$ for the hard instance, we obtain after normalizing by the confinement threshold
    \begin{align}
        \label{eq:confinement-compensated-initial-bound}
        \frac{\mathbb{E}[\tilde V_0]}
        {\frac{\sigma^2}{2\mu}\left(\frac{q_{\min}}{q_{\max}}\right)^{3/2}}
        =&\
        \frac{1}{8}
        +4\alpha^2\gamma^2\mu^2q_{\max}^2K
        \left(\frac{q_{\max}}{q_{\min}}\right)^{3/2}
        \\
        =&\
        \frac{1}{8}
        +\tdO\!\lp\frac{1}{K}\rp
        \le \frac{1}{2}
        \qquad\text{for all sufficiently large $K$}.
        \nonumber
    \end{align}
    Here the $\tdO$ term hides logarithmic factors in $K$ and a polynomial factor in the fixed hardware condition number $\kappa_2=q_{\max}/q_{\min}$, together with constants depending on the remaining fixed problem parameters.
    Since $C=\mu$ by \eqref{eq:C-gamma-choice} and $P^*(W_k)-W^\diamond=(W^\star-W_k)/\gamma$, the Lyapunov function dominates \emph{both} coordinates, and in particular it satisfies the following weaker response-ratio bound:
    \begin{align}
        \label{eq:confinement-lyapunov-dominates-coordinates}
        \tilde V_k \ge V_k
        \ge \frac{\mu}{2}\min\!\left\{1,\sqrt{\frac{q_{\min}}{q_{\max}}}\right\}
        \max\bigl\{(\bar W_k-W^\star)^2,(W_k-W^\star)^2\bigr\}
        \qquad\text{for all }0\le k\le K,
    \end{align}
    The first inequality follows from $V_k\ge f(\bar W_k)-f^*=\tfrac{\mu}{2}(\bar W_k-W^\star)^2$ and $V_k\ge \tfrac{C\gamma^2}{2}\lVert P^*(W_k)-W^\diamond\rVert^2=\tfrac{C}{2}(W_k-W^\star)^2$.
    Hence, the nonnegative-supermartingale form of Doob's maximal inequality, applied to the process in \eqref{eq:confinement-compensated-supermartingale}, gives
    \begin{align}
        \label{eq:maximal-confinement}
        \mathbb{P}(\mathcal{E}^c)
        =&\
        \mathbb{P}\left(
            \bigcup_{k=0}^{K}\mathcal{E}_k^c
        \right)
        \\
        \le&\
        \mathbb{P}\left(
        \sup_{0\le k\le K}V_k
        >
        \frac{\sigma^2}{2\mu}\left(\frac{q_{\min}}{q_{\max}}\right)^{3/2}
        \right)
        \nonumber
        \\
        \le&\
        \mathbb{P}\left(
        \sup_{0\le k\le K}
        \tilde V_k
        >
        \frac{\sigma^2}{2\mu}\left(\frac{q_{\min}}{q_{\max}}\right)^{3/2}
        \right)
        \le
        \frac{\mathbb{E}[\tilde V_0]}
        {\frac{\sigma^2}{2\mu}\left(\frac{q_{\min}}{q_{\max}}\right)^{3/2}}
        \le
        \frac{1}{2}.
        \nonumber
    \end{align}
    The last inequality is \eqref{eq:confinement-compensated-initial-bound}.
    Therefore, $\mathbb{P}(\mathcal{E})\ge\frac12$ for sufficiently large $K$, proving Lemma \ref{lemma:confinement}.
\end{proof}

\end{document}